\documentclass[final,leqno,onefignum,onetabnum]{siamltex}

\usepackage{color}
\usepackage{amsmath,amsfonts,amssymb,bm,mathtools} %
\usepackage{graphicx,psfrag,subfigure}                    %
\usepackage{enumitem}
\usepackage{epstopdf}
\usepackage{latexsym}
\usepackage{mathrsfs}                           %

\allowdisplaybreaks

\renewcommand{\d}{\mathrm{d}}
\newcommand{\e}{\mathrm{e}}
\newcommand{\R}{\mathbb{R}}

\newcommand{\Z}{\mathbb{Z}}

\newcommand{\bv}{\bm{v}}
\newcommand{\bw}{\bm{w}}
\newcommand{\bx}{\bm{x}}

\newcommand{\hE}{\mathcal{E}}

\newcommand{\Mh}{\mathcal{M}_h}
\newcommand{\<}{\langle}
\renewcommand{\>}{\rangle}

\newcommand{\uinit}{u_{\text{\rm init}}}
\newcommand{\kp}{\kappa}

\newcommand{\nn}{\nonumber}
\newcommand{\dt}{{\tau}}
\newcommand{\eps}{\varepsilon}
\newcommand{\daoshu}[2]{\frac{\d #1}{\d #2}}

\newtheorem{remark}{Remark}[section]

\title{Generalized SAV-exponential integrator schemes \\for Allen--Cahn type gradient flows}

\author{
Lili Ju\footnotemark[2]
\and Xiao Li\footnotemark[3]
\and Zhonghua Qiao\footnotemark[4]
}

\begin{document}

\maketitle

\renewcommand{\thefootnote}{\fnsymbol{footnote}}

\footnotetext[2]
{Department of Mathematics, University of South Carolina, Columbia, SC 29208, USA ({ju@math.sc.edu}).
L. Ju's work is partially supported by US National Science Foundation grant DMS-2109633
and US Department of Energy grant DE-SC0020270.}
\footnotetext[3]
{Department of Applied Mathematics, The Hong Kong Polytechnic University, Hung Hom, Kowloon, Hong Kong ({xiao1li@polyu.edu.hk}).
X. Li's work is partially supported by the Hong Kong Research Council GRF grant 15300821
and the Hong Kong Polytechnic University grants 4-ZZMK and 1-BD8N.}
\footnotetext[4]
{Department of Applied Mathematics \& Research Institute for Smart Energy,
The Hong Kong Polytechnic University, Hung Hom, Kowloon, Hong Kong ({zqiao@polyu.edu.hk}).
Z. Qiao's work is partially supported by the Hong Kong Research Council RFS grant RFS2021-5S03 and GRF grant 15302919,
the Hong Kong Polytechnic University grant 4-ZZLS, and the CAS AMSS-PolyU Joint Laboratory of Applied Mathematics.
}

\renewcommand{\thefootnote}{\arabic{footnote}}

\begin{abstract}
The energy dissipation law and the maximum bound principle (MBP)
are two important physical features of the well-known Allen--Cahn equation.
While some commonly-used first-order time stepping schemes have turned out to
preserve unconditionally both energy dissipation law and MBP  for the equation,
restrictions on the time step size are still needed for existing second-order or even higher-order schemes in order to have
such simultaneous preservation.
In this paper, we develop and analyze novel first- and second-order linear numerical schemes
for  a class of Allen--Cahn type gradient flows. Our schemes combine the generalized scalar auxiliary variable (SAV) approach
and the exponential time integrator with a stabilization term, while the standard central difference stencil is used for discretization of
the spatial differential operator. We not only  prove their unconditional  preservation of  the energy dissipation law and the MBP
in the discrete setting, but also derive their optimal temporal error estimates under fixed spatial mesh.
Numerical experiments are also carried out to demonstrate the properties and performance of the proposed schemes.
\end{abstract}



\pagestyle{myheadings}
\thispagestyle{plain}
\markboth{LILI JU, XIAO LI, AND ZHONGHUA QIAO}
{Generalized SAV-Exponential Integrator Schemes for Allen--Cahn Type Gradient Flows}

\section{Introduction}

The classic Allen--Cahn equation,
originally introduced in \cite{AlCa79} to model the motion of the anti-phase boundaries in crystalline solids,
takes the following form:
\begin{equation}
\label{AllenCahn}
u_t = \eps^2\Delta u + f(u), \quad t > 0, \ \bx \in \Omega,
\end{equation}
where the spatial domain $\Omega\subset\R^d$,
$u(t,\bx):[0,\infty)\times\overline{\Omega}\to\R$ is the unknown function,
$\eps>0$ is an interfacial parameter,
and $f(u)=u-u^3$ is the nonlinear reaction.
Equipped with the periodic or homogeneous Neumann boundary condition,
the equation \eqref{AllenCahn} can be viewed as
the $L^2$ gradient flow with respect to the energy functional
\begin{equation}
\label{energy}
E(u) = \int_\Omega \Big( \frac{\eps^2}{2}|\nabla u(\bx)|^2 + F(u(\bx)) \Big) \, \d\bx,
\end{equation}
where $F(u)=\frac{1}{4}(u^2-1)^2$ (i.e., $-F' = f$) is the double-well potential function,
and thus satisfies the so-called \textit{energy dissipation law}
in the sense that the solution to \eqref{AllenCahn} decreases the energy \eqref{energy} along with the time,
i.e., $\daoshu{}{t}E(u(t))\le0$.
The solution $u$ usually represents the difference between
the concentrations of two components of the alloy,
and thus should be evaluated between $-1$ and $1$ naturally,
which corresponds to another important feature, the \textit{maximum bound principle} (MBP),
i.e., if the initial value falls pointwise between $-1$ and $1$,
then so is the solution for all time.
Recently, some variants of the Allen--Cahn equation \eqref{AllenCahn}
have been developed to model various processes of phase transition,
such as the nonlocal Allen--Cahn equation for phase separations within long-range interactions \cite{Bates06}
and the fractional Allen--Cahn equation for anomalous diffusion processes \cite{GuiZh15},
and they also satisfy the energy dissipation law with respect to their respective energy and the MBP.
Since the analytic solutions to these models are usually not available,
numerical solutions play a key role in their study and applications. In order to obtain efficient and stable
numerical simulations and avoid nonphysical results, it is highly desirable to design
accurate numerical methods in space and time which also preserve important physical features of the models,
such as the energy dissipation law and the MBP.

In recent years, numerical schemes  preserving the energy dissipation law
have attracted a lot of attention for time integration of  the Allen--Cahn type equations and other gradient flows,
including convex splitting schemes \cite{GuWaWi14,ShWaWaWi12,WiWaLo09},
stabilized implicit-explicit (IMEX) schemes \cite{FeTaYa13,ShYa10b,XuTa06},
discrete gradient schemes \cite{DuNi90,Furihata01,McQuRo99},
exponential time differencing (ETD) schemes \cite{DuJuLiQi21,JuLiQiZh18,ZhuJuZh16},
invariant energy quadratization (IEQ) schemes \cite{XuYaZhXi19,YangZh20},
scalar auxiliary variable (SAV) schemes \cite{ChengWa21,ShXu18,ShXuYa18,ShXuYa19},
and some variants of the SAV method \cite{ChengLiSh20,ChengLiSh21,HouAzXu19,HuangShYa20,LiuLi20}.
By combining the SAV approach with the Runge--Kutta (RK) method,
arbitrarily high-order linear schemes preserving energy dissipation law were developed in \cite{AkrivisLiLi19,GongZh19}.
In addition, there are also  a large amount of literature denoted to
MBP-preserving numerical schemes for the Allen--Cahn type gradient flow problems,
such as the stabilized IMEX schemes \cite{ShTaYa16,TaYa16}
and the exponential integrator methods \cite{DuJuLiQi21,LiJuCaFe21}.
Borrowing the idea of strong stability-preserving methods \cite{GoShTa01},
MBP-preserving RK-type schemes with high-order accuracy were studied theoretically
and up to fourth-order schemes were provided for practical computations
in \cite{JuLiQiYa21,LiLiJuFe21,ZhangYaQiSo21}.
However, among all schemes we have just mentioned,
only a few first-order schemes can preserve simultaneously the energy dissipation law and the MBP {\em unconditionally}
\cite{DuJuLiQi21,DuYaZh20,ShTaYa16} (i.e., without any restriction on the time step size),
while the second-order schemes always require certain restrictions on the time step size \cite{HoTaYa17,LiaoTaZh20}.
It is an interesting and important question whether there exist  second-order or even higher-order time stepping schemes
preserving both the energy dissipation law and the MBP unconditionally.
An initial improvement was made in \cite{YangYuZh22} by considering the high-order SAV-RK method \cite{AkrivisLiLi19}
to guarantee the energy dissipation, and the maximum bound is enforced by the cut-off post-processing but not by the scheme itself.

{
The $H^{-1}$ gradient flow with respect to the energy \eqref{energy}
gives the classic Cahn--Hilliard equation $u_t = -\Delta(\eps^2\Delta u + f(u))$,
which fails to satisfy the MBP due to the existence of the fourth-order dissipation term.
It is also worth noting that, if the nonlinear reaction function $f$ is changed to the logarithmic one defined by \eqref{f_fh} (i.e., corresponding to the Flory--Huggins potential)  tested in our numerical experiments,
the solution of the Cahn--Hilliard equation still remains in the open interval $(-1,1)$ for all the time under  some appropriate boundary conditions \cite{CherfilsMiZe11,ElliottLu91},
where $\pm1$ are the points near which the singularities occur.
Implicit or implicit-explicit numerical schemes preserving such uniform boundedness  have also been developed, 
where the singularity of the nonlinear term plays a crucial role
\cite{ChenWaWaWi19,DongWaZhZh19} in their construction.
}

In this paper, our main purpose  is to systematically develop first- and second-order  (in time)  linear numerical schemes
preserving both the energy dissipation law and the MBP unconditionally
for a family  of Allen--Cahn type gradient flows.
More precisely, we will consider the equation \eqref{AllenCahn}
with a more general reaction term $f:\R\to\R$ given by a continuously differentiable function satisfying
\begin{equation}
\label{assump}
\text{there exists a constant $\beta>0$ such that $f(\beta) \le 0\le f(-\beta)$}.
\end{equation}
The periodic or homogeneous Neumann boundary condition is equipped
and the initial condition is given  as
$u(0,\cdot) = \uinit$ on $\overline{\Omega}$.
Then, the MBP holds \cite{DuJuLiQi21} in the sense that
if the absolute value of the initial value is bounded pointwise by $\beta$,
then the absolute value of the solution is also bounded pointwise by $\beta$ for all time, i.e.,
\begin{equation}
\label{mbp_cont}
\max_{\bx\in\overline{\Omega}} |\uinit(\bx)| \le \beta \qquad \Longrightarrow \qquad
\max_{\bx\in\overline{\Omega}} |u(t,\bx)| \le \beta, \quad \forall \, t > 0.
\end{equation}
Furthermore, the energy dissipation law is also satisfied with respect to the energy \eqref{energy} with
 $F$ now being  a smooth potential function satisfying $F'=-f$.
The key ingredient is the appropriate combination of the  SAV approach  and  the exponential integrator method \cite{CoMa02,HoOs10}. Note that similar ideas have been applied to some nonlinear hyperbolic-type equations \cite{CuiXuWaJi21,JiangWaCa20}.
We first reformulate the model equation \eqref{AllenCahn} in an equivalent form
by defining an auxiliary variable, similar to the idea of the generalized SAV approach \cite{ChengLiSh21}.
Then, we introduce a stabilization term to the system and apply exponential integrators
to develop  first- and second-order linear schemes in time. We
show that both schemes simultaneously preserve the energy dissipation law and the MBP unconditionally
under  appropriate stabilizing constant.
In the error analysis, one of the  major difficulties is caused by the variable coefficients from the nonlinear reaction and the stabilization terms.
By using the energy dissipation law and the MBP,
these variable coefficients are shown to be bounded from above and below by  certain generic positive  constants, which helps us
to successfully prove the optimal temporal convergence under fixed spatial mesh.
To the best of our knowledge,
this is the first work providing a second-order linear numerical scheme for time integration of the model Allen--Cahn type gradient flows,
with provable unconditional  preservation of both the energy dissipation law and the MBP.

The  rest of this paper is organized as follows.
In Section \ref{sect_spacedis},
we present the spatial discretization with the central finite difference and some useful lemmas.
In Section \ref{sect_eisav}, we propose the first- and second-order generalized SAV-exponential integrator (GSAV-EI) schemes, and then
prove their unconditional preservation of both the energy dissipation law and the MBP, followed by their temporal convergence analysis.
Numerical experiments are carried out to validate the theoretical results and demonstrate the performance of  the proposed schemes in Section \ref{sect_experiment}.
Finally, some concluding remarks are given in Section \ref{sect_conclusion}.

\section{Spatial discretization and some preliminaries}
\label{sect_spacedis}

Throughout this paper,
we consider the two-dimensional square domain $\Omega=(0,L)\times(0,L)$
for the model equation \eqref{AllenCahn} with $f$ satisfying the assumption \eqref{assump}. Without loss of generality,  we impose the periodic boundary condition.
Extensions to the three-dimensional case and homogeneous Neumann boundary condition do not have any difficulties.
In this section, we will present some notations related to the spatial discretization
and a few preliminary lemmas for the analysis of the time integration schemes proposed later.


Given a positive integer $M$,
let $h=L/M$ be the size of the uniform mesh partitioning $\overline{\Omega}$,
and denote by $\Omega_h=\{(x_i,y_j)=(ih,jh)\,|\,1\le i,j\le M\}$ the set of mesh points.
For a grid function $v$ defined on $\Omega_h$, we denote $v_{ij}=v(x_i,y_j)$.
Let $\Mh$ be the set of all $M$-periodic grid functions on $\Omega_h$, i.e.,
$\Mh = \{v:\Omega_h\to\R \,|\, v_{i+kM,j+lM}=v_{ij}, \ k,l\in\Z, \ 1 \le i,j\le M\}$.
Let us apply the central finite difference method to approximate the spatial differential operators.
For any $v\in\Mh$, the discrete Laplace operator $\Delta_h$ is defined by
\[
\Delta_h v_{ij} = \frac{1}{h^2} (v_{i+1,j}+v_{i-1,j}+v_{i,j+1}+v_{i,j-1}-4v_{ij}), \quad 1 \le i,j \le M,
\]
and the discrete gradient operator $\nabla_h$ is defined by
\[
\nabla_h v_{ij} = \Big( \frac{v_{i+1,j}-v_{ij}}{h}, \frac{v_{i,j+1}-v_{ij}}{h}\Big)^T, \quad 1 \le i,j \le M.
\]
The eigenvalues of $\Delta_h$ are given by \cite{JuZhZhDu15}
\begin{equation}
\label{eig_laplace}
\lambda_{kl} = -\frac{4}{h^2} \Big(\sin^2\frac{k\pi}{M}+\sin^2\frac{l\pi}{M}\Big)\leq 0, \quad 0 \le k,l\le M-1.
\end{equation}
As usual, the discrete inner product $\<\cdot,\cdot\>$, the discrete $L^2$ norm $\|\cdot\|$,
and the discrete $L^\infty$ norm $\|\cdot\|_\infty$ can be defined respectively by
\[
\<v,w\> = h^2 \sum_{i,j=1}^M v_{ij} w_{ij}, \qquad
\|v\| = \sqrt{\<v,v\>}, \qquad
\|v\|_\infty = \max_{1\le i,j\le M} |v_{ij}|
\]
for any $v,w\in\Mh$, and
\[
\<\bv,\bw\> = \<v^1,w^1\> + \<v^2,w^2\>, \qquad
\|\bv\| = \sqrt{\<\bv,\bv\>}
\]
for any $\bv=(v^1,v^2)^T, \bw=(w^1,w^2)^T\in\Mh\times\Mh$.
By the periodicity, the summation-by-parts formula obviously holds:
\[
\<v,\Delta_hw\> = - \<\nabla_hv, \nabla_hw\> = \<\Delta_hv,w\>, \quad \forall\,v,w\in\Mh.
\]

The space-discrete problem corresponding to \eqref{AllenCahn} is then
to find a function $u_h:[0,\infty)\to\Mh$ satisfies
\begin{equation}
\label{semidis}
\daoshu{u_h}{t} = \eps^2\Delta_h u_h + f(u_h), \quad t > 0
\end{equation}
with $u_h(0)=\hat u_{{\rm init}}$, where $\hat u_{{\rm init}}$ is the pointwise projection of $\uinit$ onto $\Mh$. Throughout the paper, we do not differ $\hat u_{{\rm init}}$
and $u_{{\rm init}}$ anymore since there is no ambiguity.
It is easy to verify the energy dissipation law for the space-discrete problem \eqref{semidis}, i.e.,
$\daoshu{}{t} E_h(u_h(t)) \le 0$,
where $E_h$ is the spatially-discretized energy functional defined as
\begin{equation}
\label{dis_energy}
E_h(v) := \frac{\eps^2}{2} \|\nabla_h v\|^2 + \<F(v),1\>, \quad \forall \, v \in \Mh.
\end{equation}
As shown in \cite{DuJuLiQi21},
the MBP is also valid for the space-discrete problem \eqref{semidis}, i.e.,
\begin{equation}
\label{mbp}
\|u_{{\rm init}}\|_\infty \le \beta \qquad \Longrightarrow \qquad
\|u_h(t)\|_\infty \le \beta, \quad \forall \, t > 0.
\end{equation}


We have assumed that $f$ is continuously differentiable,
so $\|f'\|_{C[-\beta,\beta]}$ is finite.
Then the following result is valid \cite{DuJuLiQi21}.

\begin{lemma}
\label{lem_nonlinear}
Under the assumption \eqref{assump},
if $\kp\ge \|f'\|_{C[-\beta,\beta]}$ holds for some positive constant $\kp$,
then we have $|f(\xi)+\kp\xi|\le\kp\beta$ for any $\xi\in[-\beta,\beta]$.
\end{lemma}

Since $\Mh$ is a finite-dimensional linear space,
any grid function in $\Mh$ and any linear operator from $\Mh$ to $\Mh$
can be treated as a vector in $\R^{M^2}$ and a matrix in $\R^{M^2\times M^2}$, respectively.
For functions of matrix/operator, we have the following lemma (see \cite{Matfun08}).

\begin{lemma}
\label{lem_matfun}
Let $\phi$ be defined on the spectrum of a diagonalizable matrix $A\in\R^{m\times m}$, i.e.,
the values $\{\phi(\lambda_i)\}_{i=1}^m$ exist, where $\{\lambda_i\}_{i=1}^m$ are the eigenvalues of $A$.
Then
\begin{itemize}
\item [\rm (i)]   $\phi(A)$ commutes with $A$ and $\phi(A^T)=\phi(A)^T$;
\item [\rm (ii)]  the eigenvalues of $\phi(A)$ are $\{\phi(\lambda_i)\,|\,1\le i\le m\}$;
\item [\rm (iii)] $\phi(P^{-1}AP)=P^{-1}\phi(A)P$ for any nonsingular matrix $P\in\R^{m\times m}$.
\end{itemize}
\end{lemma}

We still use the notations $\|\cdot\|$ and $\|\cdot\|_\infty$
to denote the matrix induced-norms consistent with $\|\cdot\|$ and $\|\cdot\|_\infty$ defined before, respectively.
By viewing $\Delta_h$ as a matrix,
we know that $\Delta_h$ is symmetric, negative semi-definite,
and weakly diagonally dominant with all diagonal entries negative.
Moreover, we have the following useful estimate,
which comes from the fact that
$\Delta_h$ is the generator of a contraction semigroup \cite{DuJuLiQi21},
and the proof can be found in \cite{DuJuLiQi21,JuLiQiYa21}.

\begin{lemma}
\label{lem_lapdiff}
For any real numbers $a\ge0$ and $b\ge0$, we have $\|\e^{a\Delta_h-b I}\|_\infty \le \e^{-b}$,
where $I\in\R^{M^2\times M^2}$ is the identity matrix.
\end{lemma}

{
\begin{remark}
Apart from the central difference discretization discussed above,
the lumped mass finite element method with piecewise linear basis functions can also be adopted
and Lemma \ref{lem_lapdiff} still holds correspondingly \cite{DuJuLiQi21}.
In addition, there have been some initial explorations on the MBP-preserving methods
using the fourth-order accurate spatial discretization,
such as the compact difference approximation \cite{TianJiLu18}
and the finite difference formulation of the $Q^2$ spectral element method \cite{ShenZh21},
combined with the Euler-type time-stepping approaches.
However, it is not obvious that
whether these fourth-order discrete Laplace operators satisfy Lemma \ref{lem_lapdiff},
and thus it is worthy of further investigations on the combination of  higher-order spatial discretizations
with the exponential integrator methods studied in this paper.
\end{remark}
}

\section{Generalized SAV-exponential integrator schemes}
\label{sect_eisav}

From now on, we always assume the initial value $\uinit$ has the enough regularity as needed.
Let us define the bulk energy term $E_{1h}(v) := \<F(v),1\>$ for any $v \in \Mh$.
The continuity of $F$ implies that $F$ is bounded from below on $[-\beta,\beta]$.
Therefore, according to the MBP \eqref{mbp},
there exist two constants $C_*\ge 0$ and $C^*\ge 0$ such that
\begin{equation}
\label{energy_bounds}
-C_* \le E_{1h}(u_h) \le C^*.
\end{equation}
Motivated by the idea of the generalized SAV approach \cite{ChengLiSh21},
we define the auxiliary variable $s_h(t) = E_{1h}(u_h(t))$, and rewrite the space-discrete  equation \eqref{semidis} in an alternate but equivalent form as below:
\begin{subequations}
\label{eq_semidis}
\begin{align}
\daoshu{u_h}{t} & = \eps^2\Delta_h u_h + \frac{\sigma(s_h)}{\sigma(E_{1h}(u_h))} f(u_h), \label{eq_semidisa} \\
\daoshu{s_h}{t} & = - \frac{\sigma(s_h)}{\sigma(E_{1h}(u_h))} \Big\<f(u_h),\daoshu{u_h}{t}\Big\>, \label{eq_semidisb}
\end{align}
\end{subequations}
where $\sigma:\R\to\R$ is a one-variable function satisfying the following  two conditions:
\begin{itemize}
\item[$(\Sigma_1)$] $\sigma>0$ on $\R$;
\item[$(\Sigma_2)$] $\sigma$ is continuously differentiable and $\sigma'\ge0$ on $\R$.
\end{itemize}
These conditions are crucial to the MBP preservation and error analysis of the proposed time integration schemes in this paper.
For any $v\in\Mh$ and $r\in\R$,  define
\begin{equation}
\label{g_def}
g(v,r): = \frac{\sigma(r)}{\sigma(E_{1h}(v))}
\end{equation}
(clearly $g(v,r)>0$ due to the condition $(\Sigma_1)$) and the modified energy
\begin{equation}
\label{modified_energy}
\hE_h(v,r) := \frac{\eps^2}{2} \|\nabla_h v\|^2 + r.
\end{equation}
Obviously, $g(u_h,s_h)\equiv1$ and $\hE_h(u_h,s_h)\equiv E_h(u_h)$ without time discretization.

\begin{remark}
\label{sgchoice}
The form \eqref{eq_semidis} and the conditions $(\Sigma_1)$ and $(\Sigma_2)$
differs from those given in \cite{ChengLiSh21},
where $\sigma$ may not be defined on the whole real line $\R$
and the constructed schemes may be nonlinear,
while our schemes developed later are all linear.
A trivial  choice for the function $\sigma$ satisfying $(\Sigma_1)$ and $(\Sigma_2)$
is the positive constant mapping, e.g., $\sigma\equiv1$. This gives a degenerate case since we still have exactly $g(u_h,s_h)\equiv 1$ with whatever time discretization and consequently $s_h$ does not provide any feedback  to the update of $u_h$ at each time step,  which is similar to the idea adopted in \cite{QiaoSuZh15}.
Some nontrivial choices \cite{ChengLiSh21} include elementary functions such as
$\sigma(x) = \e^x$, $\sigma(x) = \frac{\pi}{2} + \arctan(x)$,  $\sigma(x) = 1 + \tanh(x)$,
or  even special functions constructed by the integration such as
$\sigma(x) = \int_{-\infty}^x \eta(y) \, \d y$
for some continuous function $\eta\geq 0$  on $\R$.
\end{remark}

Next we develop  exponential integrators for  the space-discrete system \eqref{eq_semidis}, instead of the original one \eqref{semidis}.
Let us partition the time interval using the nodes $\{t_n=n\dt\}_{n\ge0}$ with  a uniform time step size $\dt>0$,
and set $u^n$ and $s^n$ as the approximations of $u_{h,e}(t_n)$ and $s_{h,e}(t_n)=E_{1h}(u_{h,e}(t_n))$ respectively, where
$u_{h,e}$ denotes the exact solution to the problem \eqref{semidis} (equivalently \eqref{eq_semidis}).

\subsection{First-order GSAV-EI scheme}

Set $u^0=\uinit$ and $s^0=E_{1h}(u^0)$.
Suppose the numerical solution $(u^n,s^n)$ is known for some $n\ge0$.
Introducing a positive stabilizing constant $\kp>0$,
the equation \eqref{eq_semidisa} is equivalent to
\begin{align*}
\daoshu{u_h}{t} & = \eps^2 \Delta_h u_h + g(u_h,s_h) f(u_h) + \kp g(u^{n}, s^{n}) (u_h-u_h) \\
& = - L_\kp^n u_h + N_\kp^n(u_h,s_h),
\end{align*}
where the linear operator $L_\kp^n=\kp g(u^n,s^n) I-\eps^2\Delta_h$ and the nonlinear operator
\[
N_\kp^n(v,r) = g(v,r) f(v) + \kp g(u^n,s^n) v, \quad \forall \, v \in \Mh, \ \forall \, r \in \R.
\]
We know that $L_\kp^n$ is self-adjoint and positive definite since $g(u^n,s^n)>0$ and $\kappa>0$.
Using the variation-of-constants formula on $[t_n,t_{n+1}]$, we have
\begin{equation*}
u_h(t_{n+1}) = \e^{-\dt L_\kp^n} u_h(t_n) + \int_0^\dt \e^{-(\dt-\theta)L_\kp^n}
N_\kp^n(u_h(t_n+\theta),s_h(t_n+\theta)) \, \d \theta.
\end{equation*}
By approximating the term $N_\kp^n$ by its value at $\theta=0$, i.e.,
$N_\kp^n(u_h(t_n+\theta),s_h(t_n+\theta))\approx N_\kp^n(u_h(t_n),s_h(t_n))$,
we obtain the first-order exponential integrator scheme for computing $u^{n+1}$ as
\begin{subequations}
\label{eq_eisav1}
\begin{align}
u^{n+1} & = \e^{-\dt L_\kp^n} u^n + \bigg(\int_0^\dt \e^{-(\dt-\theta)L_\kp^n} \,\d \theta\bigg) N_\kp^n(u^n,s^n) \label{eq_eisav1a} \\
& = \e^{-\dt L_\kp^n} u^n + \dt \phi_1(-\dt L_\kp^n) N_\kp^n(u^n,s^n), \nn
\end{align}
where $\phi_1(a)=a^{-1}(\e^a-1)$ for $a\not=0$.
Integrating \eqref{eq_semidisb} from $t_n$ to $t_{n+1}$
and using the approximation
$g(u_h(t_n+t),s_h(t_n+t)) f(u_h(t_n+t))
\approx g(u_h(t_n),s_h(t_n)) f(u_h(t_n))$,
we obtain the first-order formula for computing $s^{n+1}$ as
\begin{equation}
\label{eq_eisav1b}
s^{n+1} = s^n - g(u^n,s^n)\<f(u^n), u^{n+1}-u^n\>.
\end{equation}
\end{subequations}
The combination of \eqref{eq_eisav1a} and \eqref{eq_eisav1b} defines the first-order generalized SAV-exponential integrator (GSAV-EI1) scheme,
which  is unique solvable for any $\dt>0$ due to its explicit formulation.

\subsubsection{Energy dissipation and MBP}

We first show the unconditional preservation of
the energy dissipation law with respect to the modified energy $\hE_h$ defined by \eqref{modified_energy}  and the MBP of the GSAV-EI1 scheme \eqref{eq_eisav1}.
As a consequence, we then prove the uniform boundedness of  $g(u^n,s^n)$,
which is crucial to  the convergence analysis.

\begin{theorem}[Energy dissipation of GSAV-EI1]
\label{thm_eisav1_es}
The GSAV-EI1 scheme \eqref{eq_eisav1} is unconditionally energy dissipative
in the time-discrete sense that  $\hE_h(u^{n+1},s^{n+1})\le \hE_h(u^n,s^n)$ holds for any  $\dt>0$ and $n\ge 0$.
\end{theorem}

\begin{proof}
Using \eqref{eq_eisav1b}, some simple calculations yield
\begin{align}
& \hE_h(u^{n+1},s^{n+1}) - \hE_h(u^n,s^n)
 = \frac{\eps^2}{2} \|\nabla_h u^{n+1}\|^2 - \frac{\eps^2}{2} \|\nabla_h u^n\|^2 + s^{n+1}-s^n
 \label{thm_eisav1es_pf1}\\
& \qquad = \eps^2\<\nabla_h u^{n+1},\nabla_h u^{n+1}-\nabla_h u^n\> - \frac{\eps^2}{2}\|\nabla_h u^{n+1}-\nabla_h u^n\|^2 \nn \\
& \qquad\quad - g(u^n,s^n)\<f(u^n), u^{n+1}-u^n\> \nn \\
& \qquad = -\<\eps^2\Delta_h u^{n+1}+g(u^n,s^n)f(u^n),u^{n+1}-u^n\>
- \frac{\eps^2}{2}\|\nabla_h u^{n+1}-\nabla_h u^n\|^2 \nn \\
& \qquad = \<L_\kp^n u^{n+1}-N_\kp^n(u^n,s^n),u^{n+1}-u^n\> \nn \\
& \qquad\quad - \kp g(u^n,s^n)\|u^{n+1}-u^n\|^2 - \frac{\eps^2}{2}\|\nabla_h u^{n+1}-\nabla_h u^n\|^2. \nn
\end{align}
It can also be derived from \eqref{eq_eisav1a} that
\[
u^{n+1}-u^n = (\e^{-\dt L_\kp^n}-I) u^n + \dt \phi_1(-\dt L_\kp^n) N_\kp^n(u^n,s^n).
\]
Multiplying $[\dt\phi_1(-\dt L_\kp^n)]^{-1}=(I-\e^{-\dt L_\kp^n})^{-1}L_\kp^n$ on both sides of the above equation leads to
\begin{align*}
[\dt\phi_1(-\dt L_\kp^n)]^{-1}(u^{n+1}-u^n)
& = -L_\kp^n u^n + N_\kp^n(u^n,s^n) \\
& = L_\kp^n(u^{n+1}-u^n) - L_\kp^n u^{n+1} + N_\kp^n(u^n,s^n),
\end{align*}
and thus,
\begin{equation}
\label{thm_eisav1es_pf2}
L_\kp^n u^{n+1} - N_\kp^n(u^n,s^n) = [L_\kp^n-(I-\e^{-\dt L_\kp^n})^{-1}L_\kp^n](u^{n+1}-u^n).
\end{equation}
Note that $L_\kp^n$ is positive definite and $a-(1-\e^{-a})^{-1}a<0$ for any $a>0$,
which means, by Lemma \ref{lem_matfun}, that $L_\kp^n-(I-\e^{-\dt L_\kp^n})^{-1}L_\kp^n$ is negative definite.
Combining \eqref{thm_eisav1es_pf1} and \eqref{thm_eisav1es_pf2},
we obtain the energy dissipation $\hE_h(u^{n+1},s^{n+1})\le \hE_h(u^n,s^n)$ .
\end{proof}

\begin{remark}
Theorem \ref{thm_eisav1_es} states that
the GSAV-EI1 scheme \eqref{eq_eisav1} is energy dissipative
with respect to the modified energy $\hE_h(u^n,s^n)$ rather than the original energy $E_h(u^n)$.
Note that $\hE_h(u^n,s^n)$ is only an approximation of $E_h(u^n)$ after time discretization since usually $s^n\not=E_{1h}(u^n)$ for $n> 0$.
\end{remark}

\begin{corollary}
\label{cor_sav1_es}
For any  $\dt>0$ and $n\ge 0$, it holds that $s^n\le E_h(\uinit)$.
\end{corollary}

\begin{proof}
Since $s^0=E_{1h}(\uinit)$, we have by Theorem \ref{thm_eisav1_es} that
\[
\frac{\eps^2}{2} \|\nabla_h u^n\|^2 + s^n = \hE_h(u^n,s^n) \le \hE_h(u^{n-1},s^{n-1})
\le \cdots \le \hE_h(u^0,s^0)= E_h(\uinit).
\]
Dropping off the nonnegative term $\frac{\eps^2}{2} \|\nabla_h u^n\|^2$  leads to the expected result.
\end{proof}

\begin{theorem}[MBP of GSAV-EI1]
\label{thm_eisav1_mbp}
If $\kp\ge \|f'\|_{C[-\beta,\beta]}$, then
the GSAV-EI1 scheme \eqref{eq_eisav1} preserves the MBP unconditionally, i.e.,
for any $\dt>0$, the time-discrete version of \eqref{mbp} is valid as follows:
\begin{equation}
\label{dismbp}
\|\uinit\|_\infty \le \beta \qquad \Longrightarrow \qquad
\|u^n\|_\infty \le \beta, \quad \forall \, n\ge0.
\end{equation}
\end{theorem}

\vskip -0.4cm
\begin{proof}
Suppose $(u^n,s^n)$ is given and $\|u^n\|_\infty\le\beta$ for some $n\ge 0$.
By Lemma \ref{lem_lapdiff}, we get
$\|\e^{-\dt L_\kp^n}\|_\infty \le \e^{-\dt\kp g(u^n,s^n)}$.
Since $\kp\ge \|f'\|_{C[-\beta,\beta]}$ and $g(u^n,s^n)>0$,
by Lemma \ref{lem_nonlinear}, we have
\[
\|N_\kp^n(u^n,s^n)\|_\infty = g(u^n,s^n) \|f(u^n)+\kp u^n\|_\infty \le \kp \beta g(u^n,s^n).
\]
Therefore, we obtain from \eqref{eq_eisav1a} that
\begin{align*}
\|u^{n+1}\|_\infty
& \le \e^{-\dt\kp g(u^n,s^n)} \beta
+ \bigg(\int_0^\dt \e^{-(\dt-\theta)\kp g(u^n,s^n)} \,\d \theta\bigg) \cdot \kp\beta g(u^n,s^n) \\
& = \e^{-\dt\kp g(u^n,s^n)} \beta + \frac{1-\e^{-\dt\kp g(u^n,s^n)}}{\kp g(u^n,s^n)} \cdot \kp\beta g(u^n,s^n)
= \beta.
\end{align*}
By induction, we have $\|u^n\|_\infty\le\beta$ for any $n\ge0$.
\end{proof}

\begin{remark}
\label{rmk_sESAV1}
By applying $\e^{-\dt L_\kp^n}\approx(I+\dt L_\kp^n)^{-1}$ to  \eqref{eq_eisav1a}, we can obtain
\begin{equation}
\label{GSAV1}
\frac{u^{n+1} - u^n}{\dt} = \eps^2 \Delta_h u^{n+1} + g(u^n,s^n) f(u^n) + \kp g(u^n,s^n) (u^{n+1}-u^n).
\end{equation}
The scheme formed by \eqref{GSAV1} and \eqref{eq_eisav1b}  can be regarded as the stabilizing version of
the generalized-SAV scheme \cite{ChengLiSh21}, which also satisfies
the energy dissipation law (Theorem \ref{thm_eisav1_es}) and the MBP (Theorem \ref{thm_eisav1_mbp}).
In particular, by setting $\sigma(x)=\e^x$, it  recovers exactly the first-order stabilized exponential-SAV scheme \cite{JuLiQi21}.
\end{remark}

Unlike the time-continuous case in which $g(u_h,s_h)\equiv1$,
the coefficient $g(u^n,s^n)$ may vary at each time step unless $\sigma$ is chosen as a constant function,
which lead to some difficulties for the error analysis.
Fortunately, by the energy dissipation law and the MBP,
we can show that $g(u^n,s^n)$ is bounded uniformly in $n$.
The following lemma (without proof) is useful to estimate some exponential-related functions of matrices.

\begin{lemma}
\label{lem_expfuns}
For any $a>0$, the following inequalities hold:
\[
0 < 1 - \e^{-a} < a, \qquad
0 < \phi_1(-a) < 1, \qquad
1 < (1+a) \phi_1(-a) < 2. 
\]
\end{lemma}

\vskip -0.4cm

\begin{corollary}
\label{cor_g_bound}
Given any fixed  $h>0$ and $T>0$.
If $\kp\ge\|f'\|_{C[-\beta,\beta]}$ and $\|\uinit\|_\infty\le\beta$, then
there are two constants $G_*>0$ and $G^*>0$ such that
\[
G_* \le g(u^n,s^n) \le G^*, \quad 0 \le n \le \lfloor T/\dt \rfloor,
\]
where $G_*$ and $G^*$ depend on $C_*$, $C^*$, $|\Omega|$, $T$, $\uinit$, $\kp$, $\eps$, and $\|f\|_{C[-\beta,\beta]}$,
but are independent of $\dt$.
\end{corollary}

\begin{proof}
Since $\|u^n\|_\infty\le\beta$ (by Theorem \ref{thm_eisav1_mbp}),
according to \eqref{energy_bounds} and the conditions $(\Sigma_1)$ and $(\Sigma_2)$, it holds
$0<\sigma(-C_*) \le \sigma(E_{1h}(u^n)) \le \sigma(C^*)$.
According to Corollary \ref{cor_sav1_es}, we have
\[
g(u^n,s^n) = \frac{\sigma(s^n)}{\sigma(E_{1h}(u^n))}
\le \frac{\sigma(E_h(\uinit))}{\sigma(-C_*)} := G^*.
\]
Using \eqref{eig_laplace}, we then obtain the uniform bound of the spectral radius of $L_\kp^n$, $\rho(L_\kp^n)$, as
\begin{equation}
\label{Lkp_specrad}
\rho(L_\kp^n) \le \kp g(u^n,s^n) + \eps^2 \rho(\Delta_h) \le  G^*\kp + \frac{8\eps^2}{h^2} := M_h.
\end{equation}
Next we show the existence of the lower bound of $\{s^n\}$.
By making use of the MBP and the first inequality in Lemma \ref{lem_expfuns},
we derive from \eqref{eq_eisav1a} that
\begin{align*}
\|u^{n+1} - u^n\|
& \le \|I-\e^{-\dt L_\kp^n}\| \|u^n\| + \dt \|\phi_1(-\dt L_\kp^n)\| \|N_\kp(u^n,s^n)\| \\
& \le \dt \rho(L_\kp^n) \cdot \beta |\Omega|^{\frac{1}{2}} + \dt \cdot \kp\beta g(u^n,s^n) |\Omega|^{\frac{1}{2}}  \le \dt \beta |\Omega|^{\frac{1}{2}} (M_h+G^*\kp).
\end{align*}
Since $\|f(u^n)\|\le F_0|\Omega|^{\frac{1}{2}}$ with $F_0:=\|f\|_{C[-\beta,\beta]}$,
we then obtain from \eqref{eq_eisav1b} that
\begin{equation}
\label{s_lowbound_pf}
s^{n+1} \ge s^n - g(u^n,s^n) \|f(u^n)\| \|u^{n+1}-u^n\|
\ge s^n - G^* F_0 \beta |\Omega| (M_h+G^*\kp)\dt.
\end{equation}
By recursion, noting that $s^0 = E_{1h}(\uinit) \ge -C_*$, we obtain
\begin{equation}
\label{s_lowbound}
s^n \ge s^0 - G^* F_0 \beta |\Omega| (M_h+G^*\kp) n\dt \ge -C_* - G^* F_0 \beta |\Omega| (M_h+G^*\kp)T := S_*.
\end{equation}
Thus
$g(u^n,s^n)
\ge \sigma(S_*)/\sigma(C^*) := G_*,$
which completes the proof.
\end{proof}

\subsubsection{Temporal error analysis}

Note that  \eqref{eq_eisav1a} is equivalent to find $u^{n+1}=w(\dt)$ with $w(\theta)$ satisfying
\begin{equation}
\label{etd1_eqv}
\begin{dcases}
\daoshu{w(\theta)}{\theta} + L_\kp^n w(\theta) =  N_\kp^n(u^n,s^n), & \theta \in(0,\dt], \\
w(0) = u^n.
\end{dcases}
\end{equation}
 Let $s_{h,e}(t)=E_{1h}(u_{h,e}(t))$.
Define $w_{h,e}(\theta)=u_{h,e}(t_n+\theta)$ for $\theta\in[0,\dt]$. It holds
\begin{equation}
\label{etd1_eqv_trun}
\begin{dcases}
\daoshu{w_{h,e}(\theta)}{\theta} + L_\kp^n w_{h,e}(\theta)
=  N_\kp^n(u_{h,e}(t_n),s_{h,e}(t_n)) + R_{1u}^n(\theta), & \theta \in(0,\dt], \\
w_{h,e}(0) = u_{h,e}(t_n),
\end{dcases}
\end{equation}
where the truncation error
\begin{align*}
R_{1u}^n(\theta)
& = g(u_{h,e}(t_n+\theta),s_{h,e}(t_n+\theta))f(u_{h,e}(t_n+\theta)) - g(u_{h,e}(t_n),s_{h,e}(t_n))f(u_{h,e}(t_n)) \\
& \quad + \kp g(u^n,s^n) (u_{h,e}(t_n+\theta) - u_{h,e}(t_n)).
\end{align*}
For $s_{h,e}(t)$, we have
\begin{eqnarray}
&&s_{h,e}(t_{n+1})-s_{h,e}(t_n)\label{sav1trunb}\\
&&\qquad = - g(u_{h,e}(t_n),s_{h,e}(t_n)) \<f(u_{h,e}(t_n)),u_{h,e}(t_{n+1})-u_{h,e}(t_n)\> + \dt R_{1s}^n, \nn
\end{eqnarray}
for some truncation residual $R_{1s}^n$. Furthermore, it is easy to verify that
\begin{equation}
\label{eisav1trunerr}
{\sup_{\theta\in(0,\dt)}}\|R_{1u}^n(\theta)\| \le C_{e,h}\dt,\qquad |R_{1s}^n|\le C_{e,h}\dt,
\end{equation}
where  the constant $C_{e,h}>0$ depends on $u_{h,e}$, $\kp$, $\eps$, and $\|f\|_{C^1[-\beta,\beta]}$.
Now we define the error functions as
\begin{equation}
\label{errorfuns}
e_u^n = u^n - u_{h,e}(t_n), \qquad e_s^n = s^n - s_{h,e}(t_n).
\end{equation}

\begin{lemma}
\label{lem_sav1_nonlinear}
Given any fixed  $h>0$ and $T>0$.
If $\kp\ge\|f'\|_{C[-\beta,\beta]}$ and $\|\uinit\|_\infty\le\beta$, we have
\[
\|g(u^n,s^n)f(u^n) - g(u_{h,e}(t_n),s_{h,e}(t_n))f(u_{h,e}(t_n))\| \le C_g (\|e_u^n\| + |e_s^n|),
\quad 0 \le n \le \lfloor T/\dt \rfloor,
\]
where $C_g>0$ is a constant depending on $C_*$, $|\Omega|$, $\uinit$, and $\|f\|_{C^1[-\beta,\beta]}$.
\end{lemma}

A special case of this lemma with $\sigma(x)=\e^x$ has been proved in  \cite{JuLiQi21}.
Using \eqref{energy_bounds} and the conditions $(\Sigma_1)$ and $(\Sigma_2)$,
there is no essential difficulty to obtain the general result,
so we omit the proof.

\begin{theorem}[Temporal error estimate of GSAV-EI1]
\label{thm_etd1_error}
Given any fixed  $h>0$ and $T>0$, and let $\kp\ge \|f'\|_{C[-\beta,\beta]}$.
Assume that  the exact solution $u_{h,e}$ is smooth enough on $[0,T]$ and $\|\uinit\|_\infty\le\beta$.
If $\dt>0$ is sufficiently small, then we have the following error estimate
for the GSAV-EI1 scheme \eqref{eq_eisav1}:
\begin{equation}
\label{etd1_error}
\|u^n - u_{h,e}(t_n)\| + |s^n - s_{h,e}(t_n)| \le C_{h,1}\dt,\quad 0 \le n \le \lfloor T/\dt \rfloor,
\end{equation}
where the constant $C_{h,1}>0$ 
is independent of $\dt$.
\end{theorem}

\begin{proof}
Define $e(\theta)=w(\theta)-w_{h,e}(\theta)$.
The difference between \eqref{etd1_eqv} and \eqref{etd1_eqv_trun} leads to
\[
\begin{dcases}
\daoshu{e(\theta)}{\theta} + L_\kp^n e(\theta)
= N_\kp^n(u^n,s^n) - N_\kp^n(u_{h,e}(t_n),s_{h,e}(t_n)) - R_{1u}^n(\theta), \quad \theta \in(0,\dt], \\
e(0) =  e_u^n,
\end{dcases}
\]
whose solution $e(\dt)=u^{n+1} - u_{h,e}(t_{n+1}) = e_u^{n+1}$ can be expressed as
\begin{align}
e_u^{n+1} & = \e^{-\dt L_\kp^n} e_u^n
+ \dt \phi_1(-\dt L_\kp^n) [N_\kp^n(u^n,s^n) - N_\kp^n(u_{h,e}(t_n),s_{h,e}(t_n))] \label{etd1_err} \\
& \quad - \int_0^\dt \e^{-(\dt-\theta)L_\kp^n} R_{1u}^n(\theta) \, \d \theta. \nn
\end{align}
Acting $I+\dt L_\kp^n$ on both sides of \eqref{etd1_err}, we obtain
\begin{align}
& (1+\dt\kp g(u^n,s^n)) (e_u^{n+1}-e_u^n) - \dt\eps^2(\Delta_h e_u^{n+1} - \Delta_h e_u^n) \label{etd1_err_pf1} \\
& \qquad = \dt (I+\dt L_\kp^n) \phi_1(-\dt L_\kp^n)
[N_\kp^n(u^n,s^n) - N_\kp^n(u_{h,e}(t_n),s_{h,e}(t_n))] \nn \\
& \qquad\quad + (I+\dt L_\kp^n)(\e^{-\dt L_\kp^n}-I) e_u^n
- \int_0^\dt (I+\dt L_\kp^n)\e^{-(\dt-\theta)L_\kp^n} R_{1u}^n(\theta) \, \d \theta. \nn
\end{align}
Taking the discrete inner product of \eqref{etd1_err_pf1} with $\delta_t e_u^{n+1}:=(e_u^{n+1}-e_u^n)/\dt$, we get the reformulation of
the left-hand side (LHS) of \eqref{etd1_err_pf1} as 
\[
\text{LHS} = \dt \|\delta_t e_u^{n+1}\|^2 + \kp g(u^n,s^n) \|e_u^{n+1}-e_u^n\|^2 + \eps^2\|\nabla_h e_u^{n+1} - \nabla_h e_u^n\|^2,
\]
and the reformulation of the right-hand side (RHS) of \eqref{etd1_err_pf1} as
\begin{align*}
&\text{RHS}
 = \dt \<(I+\dt L_\kp^n) \phi_1(-\dt L_\kp^n)
[N_\kp^n(u^n,s^n) - N_\kp^n(u_{h,e}(t_n),s_{h,e}(t_n))], \delta_t e_u^{n+1}\> \\
&\quad + \<(I+\dt L_\kp^n)(\e^{-\dt L_\kp^n}-I) e_u^n, \delta_t e_u^{n+1}\>
- \int_0^\dt \<(I+\dt L_\kp^n)\e^{-(\dt-\theta)L_\kp^n}R_{1u}^n(\theta),\delta_t e_u^{n+1}\> \,\d\theta.
\end{align*}
By using Corollary \ref{cor_g_bound}, the identity
$\|e_u^{n+1}-e_u^n\|^2 = \|e_u^{n+1}\|^2 - \|e_u^n\|^2 - 2\dt \<e_u^n,\delta_te_u^{n+1}\>$, and the Young's inequality, we obtain
\begin{align}
\text{LHS}
& \ge \dt \|\delta_t e_u^{n+1}\|^2 + G_*\kp \|e_u^{n+1}\|^2 - G_*\kp \|e_u^n\|^2 - 2G_*\kp\dt \<e_u^n,\delta_te_u^{n+1}\>
\label{etd1_err_pf1_lhs} \\
& \ge \frac{7\dt}{8} \|\delta_t e_u^{n+1}\|^2 + G_*\kp \|e_u^{n+1}\|^2 - G_*\kp \|e_u^n\|^2 - 8G_*^2\kp^2\dt \|e_u^n\|^2. \nn
\end{align}
According to \eqref{Lkp_specrad}, when $\dt\le M_h^{-1}$, we have
$\|I+\dt L_\kp^n\| \le 1 + \dt \rho(L_\kp^n) \le 2$.
By Lemma \ref{lem_expfuns}, we have $\|\e^{-\dt L_\kp^n}-I\| \le \|\dt L_\kp^n\| \le M_h \dt$.
Thus, for $\dt\le M_h^{-1}$, we have
\begin{align}
\text{RHS}
& \le 2\dt \|N_\kp^n(u^n,s^n) - N_\kp^n(u_{h,e}(t_n),s_{h,e}(t_n))\| \|\delta_t e_u^{n+1}\| \label{ttt}\\
& \quad + 2M_h\dt \|e_u^n\| \|\delta_t e_u^{n+1}\|
+ 2\dt {\sup_{\theta\in(0,\dt)}} \|R_{1u}^n(\theta)\| \|\delta_t e_u^{n+1}\| \nn\\
& \le 8\dt \|N_\kp^n(u^n,s^n) - N_\kp^n(u_{h,e}(t_n),s_{h,e}(t_n))\|^2 \nn \\
& \quad + 8M_h^2\dt \|e_u^n\|^2 + 8\dt {\sup_{\theta\in(0,\dt)}} \|R_{1u}^n(\theta)\|^2 + \frac{3\dt}{8} \|\delta_t e_u^{n+1}\|^2.\nn
\end{align}
Using Lemma \ref{lem_sav1_nonlinear}, we have
\begin{align}
& \|N_\kp^n(u^n,s^n) - N_\kp^n(u_{h,e}(t_n),s_{h,e}(t_n))\| \label{etd1_err_nonlinear} \\
& \qquad \le \|g(u^n,s^n) f(u^n) - g(u_{h,e}(t_n),s_{h,e}(t_n)) f(u_{h,e}(t_n))\| + \kp g(u^n,s^n) \|e_u^n\| \nn \\
& \qquad \le (C_g+G^*\kp) \|e_u^n\| + C_g |e_s^n|, \nn
\end{align}
and thus, we obtain from \eqref{ttt} and \eqref{etd1_err_nonlinear} that
\begin{align}
\text{RHS} & \le [16(C_g+G^*\kp)^2+8M_h^2]\dt \|e_u^n\|^2 + 16C_g^2\dt |e_s^n|^2 \label{sss} \\
& \quad + 8\dt {\sup_{\theta\in(0,\dt)}} \|R_{1u}^n(\theta)\|^2 + \frac{3\dt}{8} \|\delta_t e_u^{n+1}\|^2. \nn
\end{align}
Combining \eqref{sss} with \eqref{etd1_err_pf1_lhs}, we obtain
\begin{align}
& G_*\kp \|e_u^{n+1}\|^2 - G_*\kp \|e_u^n\|^2 + \frac{\dt}{2} \|\delta_t e_u^{n+1}\|^2  \label{etd1_err_pf2} \\
& \qquad \le [16(C_g+G^*\kp)^2+8M_h^2+8G_*^2\kp^2]\dt \|e_u^n\|^2 + 16C_g^2\dt |e_s^n|^2 \nn \\
& \qquad\quad + 8\dt {\sup_{\theta\in(0,\dt)}} \|R_{1u}^n(\theta)\|^2. \nn
\end{align}

The difference between \eqref{eq_eisav1b} and \eqref{sav1trunb} leads to
\begin{align*}
e_s^{n+1}-e_s^n
& = \< g(u_{h,e}(t_n),s_{h,e}(t_n)) f(u_{h,e}(t_n)) - g(u^n,s^n) f(u^n), u_{h,e}(t_{n+1})-u_{h,e}(t_n) \> \\
& \quad - g(u^n,s^n) \<f(u^n),e_u^{n+1}-e_u^n\> - \dt R_{1s}^n. \nn
\end{align*}
Multiplying the above equation by $2e_s^{n+1}$ yields
\begin{align}
& |e_s^{n+1}|^2 - |e_s^n|^2 + |e_s^{n+1}-e_s^n|^2 \label{sav1err_pf6} \\
& \quad = 2e_s^{n+1} \< g(u_{h,e}(t_n),s_{h,e}(t_n)) f(u_{h,e}(t_n)) - g(u^n,s^n) f(u^n),  \nn \\
& \qquad\quad u_{h,e}(t_{n+1})-u_{h,e}(t_n) \>- 2\dt g(u^n,s^n) e_s^{n+1} \<f(u^n),\delta_te_u^{n+1}\> - 2\dt R_{1s}^n e_s^{n+1}. \nn
\end{align}
For the first term on the right-hand side of \eqref{sav1err_pf6}, using Lemma \ref{lem_sav1_nonlinear}, we have
\begin{align}
& 2e_s^{n+1} \< g(u_{h,e}(t_n),s_{h,e}(t_n)) f(u_{h,e}(t_n)) - g(u^n,s^n) f(u^n), u_{h,e}(t_{n+1})-u_{h,e}(t_n) \>\label{sav1err_pf7a}\\
& \le 2|e_s^{n+1}| \| g(u_{h,e}(t_n),s_{h,e}(t_n)) f(u_{h,e}(t_n)) - g(u^n,s^n) f(u^n)\| \|u_{h,e}(t_{n+1})-u_{h,e}(t_n)\| \nn \\
& \le 2C_{g} \dt |e_s^{n+1}|  (\|e_u^n\| + |e_s^n|) \|(u_{h,e})_t(\theta_n)\| \qquad \text{($t_n<\theta_n<t_{n+1}$)} \nn \\
& \le C_1 \dt (\|e_u^n\|^2 + |e_s^n|^2 + |e_s^{n+1}|^2),\nn
\end{align}
where $C_1>0$ depends on $C_*$, $|\Omega|$, $u_{h,e}$, and $\|f\|_{C^1[-\beta,\beta]}$.
For the second term on the right-hand side of \eqref{sav1err_pf6}, we have
\begin{align}
-2\dt g(u^n,s^n) e_s^{n+1} \<f(u^n),\delta_te_u^{n+1}\>&
\le 2\dt G^*\|f(u^n)\| |e_s^{n+1}| \|\delta_te_u^{n+1}\|\label{sav1err_pf7b}\\
&\le C_2\dt |e_s^{n+1}|^2 + \frac{\dt}{2}\|\delta_te_u^{n+1}\|^2,\nn
\end{align}
where $C_2>0$ depends on $C_*$, $|\Omega|$, $\uinit$, and $\|f\|_{C[-\beta,\beta]}$.
For the third term on the right-hand side of \eqref{sav1err_pf6}, we have
\begin{equation}
\label{sav1err_pf7c}
- 2\dt R_{1s}^n e_s^{n+1} \le \dt |R_{1s}^n|^2 + \dt |e_s^{n+1}|^2.
\end{equation}
Substituting \eqref{sav1err_pf7a}--\eqref{sav1err_pf7c} into \eqref{sav1err_pf6} leads to
\begin{equation}
\label{etd1_err_pf3}
|e_s^{n+1}|^2 - |e_s^n|^2
\le C_1 \dt \|e_u^n\|^2 + C_1 \dt |e_s^n|^2
+ (1 + C_1 + C_2) \dt |e_s^{n+1}|^2 + \frac{\dt}{2}\|\delta_te_u^{n+1}\|^2 + \dt |R_{1s}^n|^2.
\end{equation}

Adding \eqref{etd1_err_pf2} and \eqref{etd1_err_pf3}
and using \eqref{eisav1trunerr},
we reach
\[
G_*\kp (\|e_u^{n+1}\|^2 - \|e_u^n\|^2) + (|e_s^{n+1}|^2 - |e_s^n|^2)
\le C_3 \dt (\|e_u^n\|^2 + |e_s^n|^2 + |e_s^{n+1}|^2) + 9C_{e,h}^2\dt^3,
\]
where $C_3>0$ depends on $C_*$, $|\Omega|$, $T$, $u_{h,e}$, $\kp$, $\eps$, and $\|f\|_{C^1[-\beta,\beta]}$.
Finally, by applying the discrete Gronwall's inequality, we obtain
\[
G_*\kp \|e_u^n\|^2 + |e_s^n|^2 \le \tilde C_{h,1} \dt^2,
\]
where $\tilde C_{h,1}>0$ is a constant independent of $\dt$, which finally gives us \eqref{etd1_error} by taking $C_{h,1} = \sqrt{\tilde C_{h,1}/\min\{1,\widehat{G}_*\kp\}}$.
\end{proof}

\subsection{Second-order GSAV-EI scheme}

Now we present the second-order generalized SAV-exponential integrator (GSAV-EI2) scheme, which is developed in the prediction-correction fashion.
Let $\kp>0$ again be the  stabilizing constant.
First, we adopt the GSAV-EI1 scheme \eqref{eq_eisav1}
to generate a solution $(\widetilde{u}^{n+1},\widetilde{s}^{n+1})$ as the prediction and define
$
\widetilde{u}^{n+\frac{1}{2}} = (u^n+\widetilde{u}^{n+1})/{2}$ and
$\widetilde{s}^{n+\frac{1}{2}} = (s^n+\widetilde{s}^{n+1})/{2}$.
Then, we rewrite the space-discrete system \eqref{eq_semidisa} in the equivalent form as
\begin{align*}
\daoshu{u_h}{t}
& = \eps^2 \Delta_h u_h + g(u_h,s_h) f(u_h) + \kp g(\widetilde{u}^{n+\frac{1}{2}},\widetilde{s}^{n+\frac{1}{2}}) (u_h-u_h) \\
& = - L_\kp^{n+\frac{1}{2}} u_h + N_\kp^{n+\frac{1}{2}}(u_h,s_h),
\end{align*}
where the linear operator $L_\kp^{n+\frac{1}{2}}=\kp g(\widetilde{u}^{n+\frac{1}{2}},\widetilde{s}^{n+\frac{1}{2}}) I-\eps^2\Delta_h$
is self-adjoint and positive definite since $g(\widetilde{u}^{n+\frac{1}{2}},\widetilde{s}^{n+\frac{1}{2}})>0$ and $\kp > 0$,
and the nonlinear operator
\begin{equation}
\label{Nkp_def2}
N_\kp^{n+\frac{1}{2}}(v,r) = g(v,r) f(v) + \kp g(\widetilde{u}^{n+\frac{1}{2}},\widetilde{s}^{n+\frac{1}{2}}) v,
\quad \forall \, v \in \Mh, \ \forall \, r \in \R.
\end{equation}
Applying the variation-of-constants formula gives us
\begin{equation}
\label{etd2_exact}
u_h(t_{n+1}) = \e^{-\dt L_\kp^{n+\frac{1}{2}}} u_h(t_n) + \int_0^\dt \e^{-(\dt-\theta)L_\kp^{n+\frac{1}{2}}}
N_\kp^{n+\frac{1}{2}}(u_h(t_n+\theta),s_h(t_n+\theta)) \, \d \theta. \vspace{-0.2cm}
\end{equation}
Approximating the term $N_\kp^{n+\frac{1}{2}}$
by its value at $\theta=\frac{\dt}{2}$ in the above equation, i.e.,
\begin{equation}
\label{etd2_approx}
N_\kp^{n+\frac{1}{2}}(u_h(t_n+\theta),s_h(t_n+\theta)) \approx
N_\kp^{n+\frac{1}{2}}(u_h(t_{n+\frac{1}{2}}),s_h(t_{n+\frac{1}{2}})),
\end{equation}
we obtain the GSAV-EI2 scheme for computing $u^{n+1}$ as
\begin{subequations}
\label{eq_eisav2}
\begin{align}
u^{n+1} & = \e^{-\dt L_\kp^{n+\frac{1}{2}}} u^n +
\bigg(\int_0^\dt \e^{-(\dt-\theta)L_\kp^{n+\frac{1}{2}}} \,\d\theta\bigg)
N_\kp^{n+\frac{1}{2}}(\widetilde{u}^{n+\frac{1}{2}},\widetilde{s}^{n+\frac{1}{2}}) \label{eq_eisav2a} \\
& = \e^{-\dt L_\kp^{n+\frac{1}{2}}} u^n + \dt \phi_1(-\dt L_\kp^{n+\frac{1}{2}})
N_\kp^{n+\frac{1}{2}}(\widetilde{u}^{n+\frac{1}{2}},\widetilde{s}^{n+\frac{1}{2}}). \nn
\end{align}
To update $s^{n+1}$, we discretize \eqref{eq_semidisb} at $t=t_{n+\frac{1}{2}}$ to give
\begin{align}
\label{eq_eisav2b}
s^{n+1} & =  s^n- g(\widetilde{u}^{n+\frac{1}{2}},\widetilde{s}^{n+\frac{1}{2}}) \<f(\widetilde{u}^{n+\frac{1}{2}}), u^{n+1}-u^n\>\\
& \quad + \frac{\kp}{2} g(\widetilde{u}^{n+\frac{1}{2}},\widetilde{s}^{n+\frac{1}{2}}) \<u^{n+1}-\widetilde{u}^{n+1}, u^{n+1}-u^n\>.\nn
\end{align}
\end{subequations}
The second term on the right-hand side of \eqref{eq_eisav2b} is based on the Crank--Nicolson discretization,
and the third term is an artificial stabilization term of high order.

\subsubsection{Energy dissipation and MBP}

Similar to the analysis of the GSAV-EI1 scheme, we first prove the unconditional preservation of
the energy dissipation law and the MBP of the GSAV-EI2 scheme \eqref{eq_eisav2},
then we show the uniform boundedness of  $g(u^n,s^n)$ and $g(\widetilde{u}^{n+\frac{1}{2}},\widetilde{s}^{n+\frac{1}{2}})$,
which is important to  the error analysis.

\begin{theorem}[Energy dissipation of GSAV-EI2]
\label{thm_eisav2_es}
The GSAV-EI2 scheme \eqref{eq_eisav2} is unconditionally energy dissipative
in the time-discrete sense that $\hE_h(u^{n+1},s^{n+1})\le \hE_h(u^n,s^n)$ holds for any  $\dt>0$ and $n\ge 0$.
\end{theorem}

\begin{proof}
Similar to the proof of Theorem \ref{thm_eisav1_es},
some simple calculations give us
\begin{align*}
& \hE_h(u^{n+1},s^{n+1}) - \hE_h(u^n,s^n)  = \<L_\kp^{n+\frac{1}{2}} u^{n+1}-N_\kp^{n+\frac{1}{2}}(\widetilde{u}^{n+\frac{1}{2}},\widetilde{s}^{n+\frac{1}{2}}),u^{n+1}-u^n\> \nn \\
& \qquad\quad
- \kp g(\widetilde{u}^{n+\frac{1}{2}},\widetilde{s}^{n+\frac{1}{2}}) \<u^{n+1}-\widetilde{u}^{n+\frac{1}{2}},u^{n+1}-u^n\> \nn \\
& \qquad\quad - \frac{\eps^2}{2}\|\nabla_h u^{n+1}-\nabla_h u^n\|^2
+ \frac{\kp}{2} g(\widetilde{u}^{n+\frac{1}{2}},\widetilde{s}^{n+\frac{1}{2}}) \<u^{n+1}-\widetilde{u}^{n+1}, u^{n+1}-u^n\> \nn \\
& \qquad = \<(L_\kp^{n+\frac{1}{2}}-(I-\e^{-\dt L_\kp^{n+\frac{1}{2}}})^{-1}L_\kp^{n+\frac{1}{2}})(u^{n+1}-u^n),u^{n+1}-u^n\> \\
& \qquad\quad - \frac{\eps^2}{2}\|\nabla_h u^{n+1}-\nabla_h u^n\|^2 - \frac{\kp}{2}g(\widetilde{u}^{n+\frac{1}{2}},\widetilde{s}^{n+\frac{1}{2}})\|u^{n+1}-u^n\|^2.
\end{align*}
Then, the energy dissipation comes from the negative definiteness of the operator
$L_\kp^{n+\frac{1}{2}}-(I-\e^{-\dt L_\kp^{n+\frac{1}{2}}})^{-1}L_\kp^{n+\frac{1}{2}}$.
\end{proof}

\begin{corollary}
\label{cor_sav2_es}
For any  $\dt>0$ and $n\ge 0$, it holds that $s^n\le E_h(\uinit)$ and $\widetilde{s}^{n+1}\le E_h(\uinit)$ for the the GSAV-EI2 scheme \eqref{eq_eisav2}.
\end{corollary}
\begin{proof}
Similar to the proof of Corollary \ref{cor_sav1_es},
the uniform boundedness of $\{s^n\}$ is a direct consequence of Theorem \ref{thm_eisav2_es}.
Since $\widetilde{s}^{n+1}$ is generated by the GSAV-EI1 scheme \eqref{eq_eisav1},
we have $\widetilde{s}^{n+1}\le \hE_h(\widetilde{u}^{n+1},\widetilde{s}^{n+1})\le \hE_h(u^n,s^n)$ according to Theorem \ref{thm_eisav1_es},
and thus, $\widetilde{s}^{n+1}\le E_h(\uinit)$.
\end{proof}

\begin{theorem}[MBP of GSAV-EI2]
\label{thm_eisav2_mbp}
If $\kp\ge \|f'\|_{C[-\beta,\beta]}$, then
the GSAV-EI2 scheme \eqref{eq_eisav2}  preserves the MBP unconditionally,
i.e., for any $\dt>0$, the time-discrete version of MBP \eqref{dismbp} is valid.
\end{theorem}

\begin{proof}
Suppose $(u^n,s^n)$ is given and $\|u^n\|_\infty\le\beta$ for some $n$.
According to Theorem \ref{thm_eisav1_mbp},
we know $\|\widetilde{u}^{n+1}\|_\infty\le\beta$,
and thus $\|\widetilde{u}^{n+\frac{1}{2}}\|_\infty\le\beta$.
We also have from Theorem \ref{thm_eisav2_es} that $\widetilde{s}^{n+\frac{1}{2}}\le E_h(\uinit)$.
Noting that \eqref{eq_eisav2a} has the same form as \eqref{eq_eisav1a},
the proof can be completed in the similar way to that of Theorem \ref{thm_eisav1_mbp}.
\end{proof}

\begin{corollary}
\label{cor_g_bound2}
Given any  fixed $h>0$ and $T>0$.
If $\kp\ge\|f'\|_{C[-\beta,\beta]}$ and $\|\uinit\|_\infty\le\beta$,
then there are two  constants $\widehat{G}_*>0$ and $G^*>0$ such that
\[
 \widehat{G}_* \le g(u^n,s^n) \le G^*, \;\;
\widehat{G}_* \le g(\widetilde{u}^{n+\frac{1}{2}},\widetilde{s}^{n+\frac{1}{2}}) \le G^*, \quad
0 \le n \le \lfloor T/\dt \rfloor-1,
\]
where $G^*$ is the same constant defined in Corollary \ref{cor_g_bound},
and $\widehat{G}_*$ depends on $C_*$, $C^*$, $|\Omega|$, $T$, $\uinit$, $\kp$, $\eps$, and $\|f\|_{C[-\beta,\beta]}$,
but is independent of $\dt$.
\end{corollary}

\begin{proof}
As done in the proof of Corollary \ref{cor_g_bound},
the upper bound  $G^*$ of $\{g(u^n,s^n)\}$ and $\{g(\widetilde{u}^{n+\frac{1}{2}},\widetilde{s}^{n+\frac{1}{2}})\}$
is the direct result of the monotonicity of $\sigma$ and Theorems \ref{thm_eisav2_es} and \ref{thm_eisav2_mbp},
and thus, $\rho(L_\kp^n)\le M_h$ and $\rho(L_\kp^{n+\frac{1}{2}})\le M_h$
with $M_h>0$ the same constant defined in \eqref{Lkp_specrad}.
For the existence of the lower bound $\widehat{G}_*$,  it suffices to show the existence of
the lower bounds of $\{s^n\}$ and $\{\widetilde{s}^{n+\frac{1}{2}}\}$.
Similar to the derivations in the proof of Corollary \ref{cor_g_bound}, we have
\[
\|u^{n+1} - u^n\| \le \dt \beta |\Omega|^{\frac{1}{2}} (M_h+G^*\kp), \qquad
\|\widetilde{u}^{n+1} - u^n\| \le \dt \beta |\Omega|^{\frac{1}{2}} (M_h+G^*\kp),
\]
and thus $\|u^{n+1}-\widetilde{u}^{n+1}\|\le 2T\beta |\Omega|^{\frac{1}{2}} (M_h+G^*\kp)$.
Then, we derive from \eqref{eq_eisav2b} that
\begin{align*}
s^{n+1}
& \ge s^n - g(\widetilde{u}^{n+\frac{1}{2}},\widetilde{s}^{n+\frac{1}{2}}) \|f(\widetilde{u}^{n+\frac{1}{2}})\| \|u^{n+1}-u^n\| \\
& \quad -\frac{\kp}{2}g(\widetilde{u}^{n+\frac{1}{2}},\widetilde{s}^{n+\frac{1}{2}})\|u^{n+1}-\widetilde{u}^{n+1}\|\|u^{n+1}-u^n\| \\
& \ge s^n - G^* F_0 \beta |\Omega| (M_h+G^*\kp)\dt
- G^*\kp T\beta^2 |\Omega| (M_h+G^*\kp)^2\dt.
\end{align*}
By recursion, we obtain
\[
s^n \ge -C_* - G^* F_0 \beta |\Omega| (M_h+G^*\kp)T - G^*\kp\beta^2 |\Omega| (M_h+G^*\kp)^2T^2.
\]
Finally, according to \eqref{s_lowbound_pf}, we also have
\[
\widetilde{s}^{n+1} \ge s^n - G^* F_0 \beta |\Omega| (M_h+G^*\kp)T,
\]
which completes the proof.
\end{proof}

\subsubsection{Temporal  error analysis}

The following lemma claims that
the temporal truncation error of \eqref{eq_eisav2} is of second order.
The proof involves some careful computations in calculus,
and we present it in Appendix \ref{app1}.

\begin{lemma}
\label{lem_etd2_error}
Given any fixed  $h>0$ and $T>0$ and
assume that  the exact solution $u_{h,e}$ is smooth enough on $[0,T]$.
Define
$
\widetilde{u}_{h,e}^{n+\frac{1}{2}}=(u_{h,e}(t_n)+u_{h,e}(t_{n+1}))/2$ and
$\widetilde{s}_{h,e}^{n+\frac{1}{2}}=(s_{h,e}(t_n)+s_{h,e}(t_{n+1}))/2$.
It holds that
\begin{subequations}
\label{etd2_trun}
\begin{align}
u_{h,e}(t_{n+1}) & = \e^{-\dt L_\kp^{n+\frac{1}{2}}} u_{h,e}(t_n) + \dt \phi_1(-\dt L_\kp^{n+\frac{1}{2}})
N_\kp^{n+\frac{1}{2}}(\widetilde{u}_{h,e}^{n+\frac{1}{2}},\widetilde{s}_{h,e}^{n+\frac{1}{2}}) + \dt R_{2u}^n, \label{etd2_truna} \\
s_{h,e}(t_{n+1}) & = s_{h,e}(t_n) - g(\widetilde{u}_{h,e}^{n+\frac{1}{2}},\widetilde{s}_{h,e}^{n+\frac{1}{2}})
\<f(\widetilde{u}_{h,e}^{n+\frac{1}{2}}), u_{h,e}(t_{n+1})-u_{h,e}(t_n)\> + \dt R_{2s}^n, \label{etd2_trunb}
\end{align}
\end{subequations}
with the truncation terms $R_{2u}^n$ and $R_{2s}^n$ satisfying
\begin{equation}
\label{etd2_trunerr}
\|R_{2u}^n\| \le C_{e,h} \dt^2,
\qquad |R_{2s}^n| \le C_{e,h} \dt^2,
\end{equation}
where the constant $C_{e,h}>0$ 
is independent of $\dt$.
\end{lemma}

The error functions $e_u^n$ and $e_s^n$ are defined by \eqref{errorfuns}.
In addition, we define
\[
\widetilde{e}_u^{n+1} = \widetilde{u}^{n+1} - u_{h,e}(t_{n+1}), \qquad
\widetilde{e}_s^{n+1} = \widetilde{s}^{n+1} - s_{h,e}(t_{n+1}).
\]
We first present a lemma on the estimates with respect to $\widetilde{e}_u^{n+1}$ and $\widetilde{e}_s^{n+1}$.
The proof is similar to that of Theorem \ref{thm_etd1_error}
and will be given in Appendix \ref{app2}.

\begin{lemma}
\label{lem_etd20_error}
Given any fixed  $h>0$ and $T>0$, and let $\kp\ge \|f'\|_{C[-\beta,\beta]}$.
Assume that  the exact solution $u_{h,e}$ is smooth enough on $[0,T]$ and $\|\uinit\|_\infty\le\beta$.
If $\dt$ is sufficiently small, then it holds that
\begin{equation}
\label{etd2_err_pf11}
\|\widetilde{e}_u^{n+1}\|^2 + |\widetilde{e}_s^{n+1}|^2
\le \widetilde{C}_h (\|e_u^n\|^2 + |e_s^n|^2) + \widetilde{C}_h C_{e,h}^2\dt^4,\quad 0 \le n \le \lfloor T/\dt \rfloor,
\end{equation}
where the constant $\widetilde{C}_h>0$ 
is independent of $\dt$.
\end{lemma}

\begin{theorem}[Temporal error estimate of GSAV-EI2]
\label{thm_etd2_error}
Given any fixed  $h>0$ and $T>0$, and let $\kp\ge \|f'\|_{C[-\beta,\beta]}$.
Assume that  the exact solution $u_{h,e}$ is smooth enough on $[0,T]$ and $\|\uinit\|_\infty\le\beta$.
If $\dt$ is sufficiently small, then we have the following error estimate
for the GSAV-EI2 scheme \eqref{eq_eisav2}:
\begin{equation}
\label{etd2_error}
\|u^n-u_{h,e}(t_n)\| + |s^n-s_{h,e}(t_n)| \le C_{h,2}\dt^2,\quad 0 \le n \le \lfloor T/\dt \rfloor,
\end{equation}
where the constant $C_{h,2}>0$ is independent of $\dt$.
\end{theorem}

\begin{proof}
The difference between \eqref{eq_eisav2a} and \eqref{etd2_truna} gives
\[
e_u^{n+1} = \e^{-\dt L_\kp^{n+\frac{1}{2}}} e_u^n + \dt \phi_1(-\dt L_\kp^{n+\frac{1}{2}})
[N_\kp^{n+\frac{1}{2}}(\widetilde{u}^{n+\frac{1}{2}},\widetilde{s}^{n+\frac{1}{2}})
- N_\kp^{n+\frac{1}{2}}(\widetilde{u}_{h,e}^{n+\frac{1}{2}},\widetilde{s}_{h,e}^{n+\frac{1}{2}})]
- \dt R_{2u}^n.
\]
Acting $I+\dt L_\kp^{n+\frac{1}{2}}$ on both sides of the above equation, we obtain
\begin{align}
\label{qqq}
& (1+\dt\kp g(\widetilde{u}^{n+\frac{1}{2}},\widetilde{s}^{n+\frac{1}{2}})) (e_u^{n+1}-e_u^n) - \dt\eps^2(\Delta_h e_u^{n+1} - \Delta_h e_u^n) \\
& \quad = \dt (I+\dt L_\kp^{n+\frac{1}{2}}) \phi_1(-\dt L_\kp^{n+\frac{1}{2}})
[N_\kp^{n+\frac{1}{2}}(\widetilde{u}^{n+\frac{1}{2}},\widetilde{s}^{n+\frac{1}{2}})
- N_\kp^{n+\frac{1}{2}}(\widetilde{u}_{h,e}^{n+\frac{1}{2}},\widetilde{s}_{h,e}^{n+\frac{1}{2}})] \nn \\
& \quad\quad + (I+\dt L_\kp^{n+\frac{1}{2}})(\e^{-\dt L_\kp^{n+\frac{1}{2}}}-I) e_u^n
- \dt (I+\dt L_\kp^{n+\frac{1}{2}}) R_{2u}^n.\nn
\end{align}
Taking the discrete inner product of \eqref{qqq} with $\delta_te_u^{n+1}$, similarly to \eqref{etd1_err_pf1_lhs},
we estimate the left-hand side (LHS) of the above identity as
\begin{align*}
\text{LHS}
& = \dt \|\delta_t e_u^{n+1}\|^2 + \kp g(\widetilde{u}^{n+\frac{1}{2}},\widetilde{s}^{n+\frac{1}{2}}) \|e_u^{n+1}-e_u^n\|^2
+ \eps^2\|\nabla_h e_u^{n+1} - \nabla_h e_u^n\|^2 \\
& \ge \dt \|\delta_t e_u^{n+1}\|^2 + \kp \widehat{G}_* \|e_u^{n+1}-e_u^n\|^2 \nn \\
& = \dt \|\delta_t e_u^{n+1}\|^2 + \widehat{G}_*\kp \|e_u^{n+1}\|^2 - \widehat{G}_*\kp \|e_u^n\|^2
- 2\widehat{G}_*\kp\dt \<e_u^n,\delta_te_u^{n+1}\> \nn \\
& \ge \frac{7\dt}{8} \|\delta_t e_u^{n+1}\|^2 + \widehat{G}_*\kp \|e_u^{n+1}\|^2 - \widehat{G}_*\kp \|e_u^n\|^2
- 8\widehat{G}_*^2\kp^2\dt \|e_u^n\|^2, \nn
\end{align*}
and, when $\dt\le M_h^{-1}$, we estimate the right-hand side (RHS) of \eqref{qqq} as
\begin{align*}
\text{RHS}
& = \dt \<(I+\dt L_\kp^{n+\frac{1}{2}}) \phi_1(-\dt L_\kp^{n+\frac{1}{2}})
[N_\kp^{n+\frac{1}{2}}(\widetilde{u}^{n+\frac{1}{2}},\widetilde{s}^{n+\frac{1}{2}})
- N_\kp^{n+\frac{1}{2}}(\widetilde{u}_{h,e}^{n+\frac{1}{2}},\widetilde{s}_{h,e}^{n+\frac{1}{2}})],\delta_te_u^{n+1}\> \nn \\
& \quad + \<(I+\dt L_\kp^{n+\frac{1}{2}})(\e^{-\dt L_\kp^{n+\frac{1}{2}}}-I) e_u^n,\delta_te_u^{n+1}\>
- \dt \<(I+\dt L_\kp^{n+\frac{1}{2}}) R_{2u}^n,\delta_te_u^{n+1}\> \nn \\
& \le 2\dt \|N_\kp^{n+\frac{1}{2}}(\widetilde{u}^{n+\frac{1}{2}},\widetilde{s}^{n+\frac{1}{2}})
- N_\kp^{n+\frac{1}{2}}(\widetilde{u}_{h,e}^{n+\frac{1}{2}},\widetilde{s}_{h,e}^{n+\frac{1}{2}})\| \|\delta_t e_u^{n+1}\|  \\
& \quad + 2M_h\dt \|e_u^n\| \|\delta_t e_u^{n+1}\|
+ 2\dt \|R_{2u}^n\| \|\delta_t e_u^{n+1}\| \nn\\
& \le 8\dt \|N_\kp^{n+\frac{1}{2}}(\widetilde{u}^{n+\frac{1}{2}},\widetilde{s}^{n+\frac{1}{2}})
- N_\kp^{n+\frac{1}{2}}(\widetilde{u}_{h,e}^{n+\frac{1}{2}},\widetilde{s}_{h,e}^{n+\frac{1}{2}})\|^2 \nn \\
& \quad + 8M_h^2\dt \|e_u^n\|^2 + 8\dt \|R_{2u}^n\|^2 + \frac{3\dt}{8} \|\delta_t e_u^{n+1}\|^2. \nn
\end{align*}
Note that
\begin{align*}
& N_\kp^{n+\frac{1}{2}}(\widetilde{u}^{n+\frac{1}{2}},\widetilde{s}^{n+\frac{1}{2}})
- N_\kp^{n+\frac{1}{2}}(\widetilde{u}_{h,e}^{n+\frac{1}{2}},\widetilde{s}_{h,e}^{n+\frac{1}{2}}) \\
& \qquad = g(\widetilde{u}^{n+\frac{1}{2}},\widetilde{s}^{n+\frac{1}{2}}) f(\widetilde{u}^{n+\frac{1}{2}})
- g(\widetilde{u}_{h,e}^{n+\frac{1}{2}},\widetilde{s}_{h,e}^{n+\frac{1}{2}}) f(\widetilde{u}_{h,e}^{n+\frac{1}{2}})
+ \kp g(\widetilde{u}^{n+\frac{1}{2}},\widetilde{s}^{n+\frac{1}{2}}) \widetilde{e}_u^{n+\frac{1}{2}},
\end{align*}
and $\widetilde{u}^{n+\frac{1}{2}},\widetilde{u}_{h,e}^{n+\frac{1}{2}},
\widetilde{s}^{n+\frac{1}{2}},\widetilde{s}_{h,e}^{n+\frac{1}{2}}$ are all bounded uniformly
according to the energy dissipation law and the MBP.
Similar to the proof of Lemma \ref{lem_sav1_nonlinear}, we can obtain
\begin{align}\label{aaa}
&\|N_\kp^{n+\frac{1}{2}}(\widetilde{u}^{n+\frac{1}{2}},\widetilde{s}^{n+\frac{1}{2}})
- N_\kp^{n+\frac{1}{2}}(\widetilde{u}_{h,e}^{n+\frac{1}{2}},\widetilde{s}_{h,e}^{n+\frac{1}{2}})\|\\
&\qquad\le \frac{C_{g}+G^*\kp}{2} (\|e_u^n\|+\|\widetilde{e}_u^{n+1}\|) + \frac{C_{g}}{2} (|e_s^n|+|\widetilde{e}_s^{n+1}|).\nn
\end{align}
Combining \eqref{aaa} with estimates of the LHS and RHS of \eqref{qqq}, we get
\begin{align}
& \widehat{G}_*\kp \|e_u^{n+1}\|^2 - \widehat{G}_*\kp \|e_u^n\|^2 + \frac{\dt}{2} \|\delta_t e_u^{n+1}\|^2 \label{etd2_err_pf3} \\
& \qquad \le [8(C_{g}+G^*\kp)^2+8M_h^2+8\widehat{G}_*^2\kp^2]\dt \|e_u^n\|^2 + 8(C_{g}+G^*\kp)^2\dt \|\widetilde{e}_u^{n+1}\|^2  \nn \\
& \qquad\quad  + 8C_{g}^2\dt (|e_s^n|^2+|\widetilde{e}_s^{n+1}|^2) + 8\dt \|R_{2u}^n\|^2. \nn
\end{align}

For the error equation \eqref{etd2_trunb}, we can add a zero-value term to the right-hand side to give
\begin{align}\label{etd2_err_pf4}
& s_{h,e}(t_{n+1}) - s_{h,e}(t_n) \\
&\quad=  - g(\widetilde{u}_{h,e}^{n+\frac{1}{2}},\widetilde{s}_{h,e}^{n+\frac{1}{2}})
\<f(\widetilde{u}_{h,e}^{n+\frac{1}{2}}), u_{h,e}(t_{n+1})-u_{h,e}(t_n)\>  + \frac{\kp}{2} g(\widetilde{u}^{n+\frac{1}{2}},\widetilde{s}^{n+\frac{1}{2}})\nn\\
& \qquad \cdot
\<u_{h,e}(t_{n+1})-u_{h,e}(t_{n+1}), u_{h,e}(t_{n+1})-u_{h,e}(t_n)\> + \dt R_{2s}^n. \nn
\end{align}
The difference between \eqref{eq_eisav2b} and \eqref{etd2_err_pf4} leads to
\begin{align*}
e_s^{n+1} - e_s^n
& = \big\<g(\widetilde{u}_{h,e}^{n+\frac{1}{2}},\widetilde{s}_{h,e}^{n+\frac{1}{2}}) f(\widetilde{u}_{h,e}^{n+\frac{1}{2}})
- g(\widetilde{u}^{n+\frac{1}{2}},\widetilde{s}^{n+\frac{1}{2}})
f(\widetilde{u}^{n+\frac{1}{2}}), u_{h,e}(t_{n+1})-u_{h,e}(t_n)\big\> \nn \\
& \quad + \frac{\kp}{2} g(\widetilde{u}^{n+\frac{1}{2}},\widetilde{s}^{n+\frac{1}{2}}) \<u^{n+1}-\widetilde{u}^{n+1}, e_u^{n+1}-e_u^n\>\\
&\quad- g(\widetilde{u}^{n+\frac{1}{2}},\widetilde{s}^{n+\frac{1}{2}})
\<f(\widetilde{u}^{n+\frac{1}{2}}), e_u^{n+1}-e_u^n\> \nn \\
& \quad + \frac{\kp}{2} g(\widetilde{u}^{n+\frac{1}{2}},\widetilde{s}^{n+\frac{1}{2}})
\<e_u^{n+1}-\widetilde{e}_u^{n+1}, u_{h,e}(t_{n+1})-u_{h,e}(t_n)\> - \dt R_{2s}^n.
\end{align*}
Multiplying the above equation by $2e_s^{n+1}$ yields
\begin{align*}
& |e_s^{n+1}|^2 - |e_s^n|^2 + |e_s^{n+1} - e_s^n|^2 \\
& \quad = 2e_s^{n+1} \big\<g(\widetilde{u}_{h,e}^{n+\frac{1}{2}},\widetilde{s}_{h,e}^{n+\frac{1}{2}})
f(\widetilde{u}_{h,e}^{n+\frac{1}{2}})
- g(\widetilde{u}^{n+\frac{1}{2}},\widetilde{s}^{n+\frac{1}{2}})
f(\widetilde{u}^{n+\frac{1}{2}}), u_{h,e}(t_{n+1})-u_{h,e}(t_n)\big\> \nn \\
& \quad\quad + \dt\kp g(\widetilde{u}^{n+\frac{1}{2}},\widetilde{s}^{n+\frac{1}{2}})
e_s^{n+1} \<u^{n+1}-\widetilde{u}^{n+1}, \delta_te_u^{n+1}\>\nn\\
&\quad\quad- 2\dt e_s^{n+1} g(\widetilde{u}^{n+\frac{1}{2}},\widetilde{s}^{n+\frac{1}{2}})
\<f(\widetilde{u}^{n+\frac{1}{2}}), \delta_te_u^{n+1}\> \nn \\
& \quad\quad + \kp g(\widetilde{u}^{n+\frac{1}{2}},\widetilde{s}^{n+\frac{1}{2}}) e_s^{n+1}
\<e_u^{n+1}-\widetilde{e}_u^{n+1}, u_{h,e}(t_{n+1})-u_{h,e}(t_n)\> - 2\dt e_s^{n+1}R_{2s}^n. \nn
\end{align*}
Similar to the deduction from \eqref{sav1err_pf6} to \eqref{etd1_err_pf3},
we then obtain
\begin{align}
|e_s^{n+1}|^2 - |e_s^n|^2
& \le C_5 \dt (\|e_u^n\|^2 + \|e_u^{n+1}\|^2 + \|\widetilde{e}_u^{n+1}\|^2 \label{etd2_err_pf7} \\
& \quad + |e_s^n|^2 + |e_s^{n+1}|^2 + |\widetilde{e}_s^{n+1}|^2)
+ \frac{\dt}{2} \|\delta_te_u^{n+1}\|^2 + \dt |R_{2s}^n|^2 \nn
\end{align}
with $C_5$ depending on $C_*$, $|\Omega|$, $u_{h,e}$, $\kp$, and $\|f\|_{C^1[-\beta,\beta]}$.

Adding \eqref{etd2_err_pf3} and \eqref{etd2_err_pf7}, we obtain
\begin{align}
& \widehat{G}_*\kp (\|e_u^{n+1}\|^2 - \|e_u^n\|^2) + (|e_s^{n+1}|^2 - |e_s^n|^2) \label{etd2_err_pf8} \\
& \qquad \le C_6 \dt (\|e_u^n\|^2 + \|e_u^{n+1}\|^2 + \|\widetilde{e}_u^{n+1}\|^2
+ |e_s^n|^2 + |e_s^{n+1}|^2 + |\widetilde{e}_s^{n+1}|^2)\nn\\
&\qquad\quad+ 8\dt \|R_{2u}^n\|^2+ \dt |R_{2s}^n|^2, \nn
\end{align}
where $C_6>0$ depends on $C_*$, $|\Omega|$, $u_{h,e}$, and $\|f\|_{C^1[-\beta,\beta]}$.
Substituting \eqref{etd2_err_pf11} into \eqref{etd2_err_pf8} and using the estimate \eqref{etd2_trunerr},
we obtain
\begin{align}\label{ddd}
& \widehat{G}_*\kp (\|e_u^{n+1}\|^2 - \|e_u^n\|^2) + (|e_s^{n+1}|^2 - |e_s^n|^2)  \\
& \qquad \le C_6(\widetilde{C}_h+1) \dt (\|e_u^n\|^2 + \|e_u^{n+1}\|^2 + |e_s^n|^2 + |e_s^{n+1}|^2)
 + (C_6 \widetilde{C}_h + 9) C_{e,h}^2\dt^5.\nn
\end{align}
When $\dt$ is sufficiently small,
similar to the last paragraph in the proof of Theorem \ref{thm_etd1_error},
applying the discrete Gronwall's inequality to \eqref{ddd} leads to
\[
\widehat{G}_*\kp \|e_u^n\|^2 + |e_s^n|^2 \le \tilde C_{h,2}\dt^4,
\]
where $\tilde C_{h,2}>0$ is a constant independent of $\tau$,  which gives us \eqref{etd2_error} by taking $C_{h,2} = \sqrt{\tilde C_{h,2}/\min\{1,\widehat{G}_*\kp\}}$.
\end{proof}

{
\begin{remark}
In addition to the periodic or homogeneous Neumann boundary condition considered above,
one can also equip the equation \eqref{AllenCahn} with the Dirichlet boundary condition
$u(t,\bx) = \psi(t,\bx)$ for $t > 0$ and $\bx \in \partial\Omega$.
Then, it is  shown in \cite{DuJuLiQi21} that 
the solution satisfies the MBP \eqref{mbp_cont}
if $|\psi(t,\bx)|\le\beta$ for any $t>0$ and $\bx \in \partial\Omega$,
and the energy dissipation law is also valid
if $\psi(t,\bx)=\psi(\bx)$ is independent of $t$.
In particular, for the equation \eqref{AllenCahn} with
a time-independent boundary  value $\|\psi\|_{C(\partial\Omega)}\le\beta$,
we are still able to develop the GSAV-EI schemes simultaneously preserving the MBP and the energy dissipation law,
based on a slight modification of the space-discrete system  \eqref{eq_semidis}.
The main idea is to add an extra term $B_h$ to \eqref{eq_semidisa},
where $B_h$ depends only on the ratio $\eps^2/h^2$ and the boundary value $\psi$; see \cite{DuJuLiQi21} for details of the form of $B_h$.
For example, the GSAV-EI2 scheme can be established  by combining \eqref{eq_eisav2a}
with $N_\kp^{n+\frac{1}{2}}$ replaced by $N_\kp^{n+\frac{1}{2}}+B_h$ and
 \eqref{eq_eisav2b} with  $-\<B_h,u^{n+1}-u^n\>$ added to its right-hand side.
Since $B_h$ is time-independent,
the $B_h$-related terms do not affect the order of the truncation error in time.
The first-order scheme can be developed in the similar spirit.
We omit the details due to the limited space.
\end{remark}
}

\section{Numerical experiments}
\label{sect_experiment}

Let us  consider the model equation \eqref{AllenCahn} for Allen--Cahn type  gradient flows in 2D square domain $\Omega=(0,1)\times(0,1)$
equipped with periodic boundary condition {or homogeneous Neumann boundary condition. In either case}, the product of a matrix exponential with a vector
can be efficiently implemented by using the fast transform based on Lemma \ref{lem_matfun}-(iii).
We also set the interfacial parameter $\eps=0.01$.
There are two commonly-used forms of the nonlinear function $f(u)$.
One is given by the cubic function
\begin{equation}
\label{f_dw}
f(u)=-F'(u)=u-u^3,
\end{equation}
where $F(u)=\frac{1}{4}(1-u^2)^2$ is the double-well potential.
In this case, one can set $\beta=1$ and $\|f'\|_{C[-1,1]}=2$.
The other one is 
determined by the Flory--Huggins potential 
$F(u) = \frac{\theta}{2}[(1+u)\ln(1+u) + (1-u)\ln(1-u)] - \frac{\theta_c}{2}u^2$,
and
\begin{equation}
\label{f_fh}
f(u) = -F'(u) = \frac{\theta}{2}\ln\frac{1-u}{1+u} + \theta_c u
\end{equation}
where $\theta_c>\theta>0$.
In the following numerical experiments,
we set $\theta=0.8$ and $\theta_c=1.6$, then the positive root of $f(u)=0$ is $\beta\approx 0.9575$
and $\|f'\|_{C[-\beta,\beta]}\approx8.02$.
We always set $\kp=\|f'\|_{C[-\beta,\beta]}$ for both cases in the following experiments.
In addition,
we always adopt the exponential function with a constant parameter $a>0$ as
\begin{equation}
\label{test_sigma}
\sigma(x)=\e^{ax}, \quad x \in \R.
\end{equation}
Clearly, $\sigma(0) = 1$ and $\sigma'(0)=a$.
\begin{remark}
\label{rmk_sgexp}
For the double-well potential \eqref{f_dw} and the Flory--Huggins potential \eqref{f_fh},
it is easy to see that the value of the bulk energy part $E_{1h}(u^n)$ is close to $0$ due to the MBP of $\{u^n\}$.
Since $s^n$ is also an approximation of $E_{1h}(u^n)$,
only the behavior of $\sigma$ near $0$ has relatively large effect on the performance of the proposed GSAV-EI schemes.
By the Taylor expansion, these typical  elementary functions for $\sigma$ given in Remark \ref{sgchoice}
perform like a linear function near $0$ with the $y$-intercept being $1$ and the slope $a=\sigma'(0)$ if parameterized as \eqref{test_sigma}.
Thus, there is  no essential  difference on all these choices
for the above test problems.
\end{remark}

\subsection{Convergence in time}

To verify the temporal convergence rates of the GSAV-EI schemes,
let us consider the problem \eqref{AllenCahn}  with
 a smooth initial value
\[
\uinit(x,y) = 0.1\sin (2\pi x) \sin (2\pi y).
\]

By fixing the uniform spatial mesh size $h=1/2048$,
we compute the numerical solutions at $t=2$ using the GSAV-EI1 and GSAV-EI2 schemes
with various time step sizes $\dt=2^{-k}$, $k=4,5,\dots,12$.
To compute the numerical errors,
the benchmark solution is generated by using the fourth-order integrating factor Runge--Kutta (IFRK4) scheme \cite{JuLiQiYa21}
with the time step size $\dt=0.1\times2^{-12}$.
Figure \ref{fig_conv} plots the $L^2$ norms of the numerical errors versus the time step sizes, produced by
GSAV-EI1 and GSAV-EI2 with $\sigma$ given by \eqref{test_sigma}
with $a=1$, $a=10$, and $a=100$,
where the left graph shows the results for the double-well potential case \eqref{f_dw} and
the right one corresponds to the Flory--Huggins potential case \eqref{f_fh}.
The expected convergence rates in time, first order for GSAV-EI1 and second order for GSAV-EI2,
are clearly observed for all cases.
In addition, we find that  the larger $a$ leads to smaller numerical errors for the GSAV-EI2 scheme, but such effect is not
 obvious for the GSAV-EI1 scheme.

We also repeat all the above convergence tests on the spatial mesh with $h=1/512$
and find the results are almost identical to those with $h=1/2048$ shown in Figure  \ref{fig_conv}.
This suggests that the temporal convergence constants in \eqref{etd1_error} and \eqref{etd2_error} could be
independent of the spatial mesh size $h$,
although we are not able to remove their dependence on $h$ in the theoretical analysis.

\begin{figure}[!ht]
\centerline{
\includegraphics[width=0.48\textwidth]{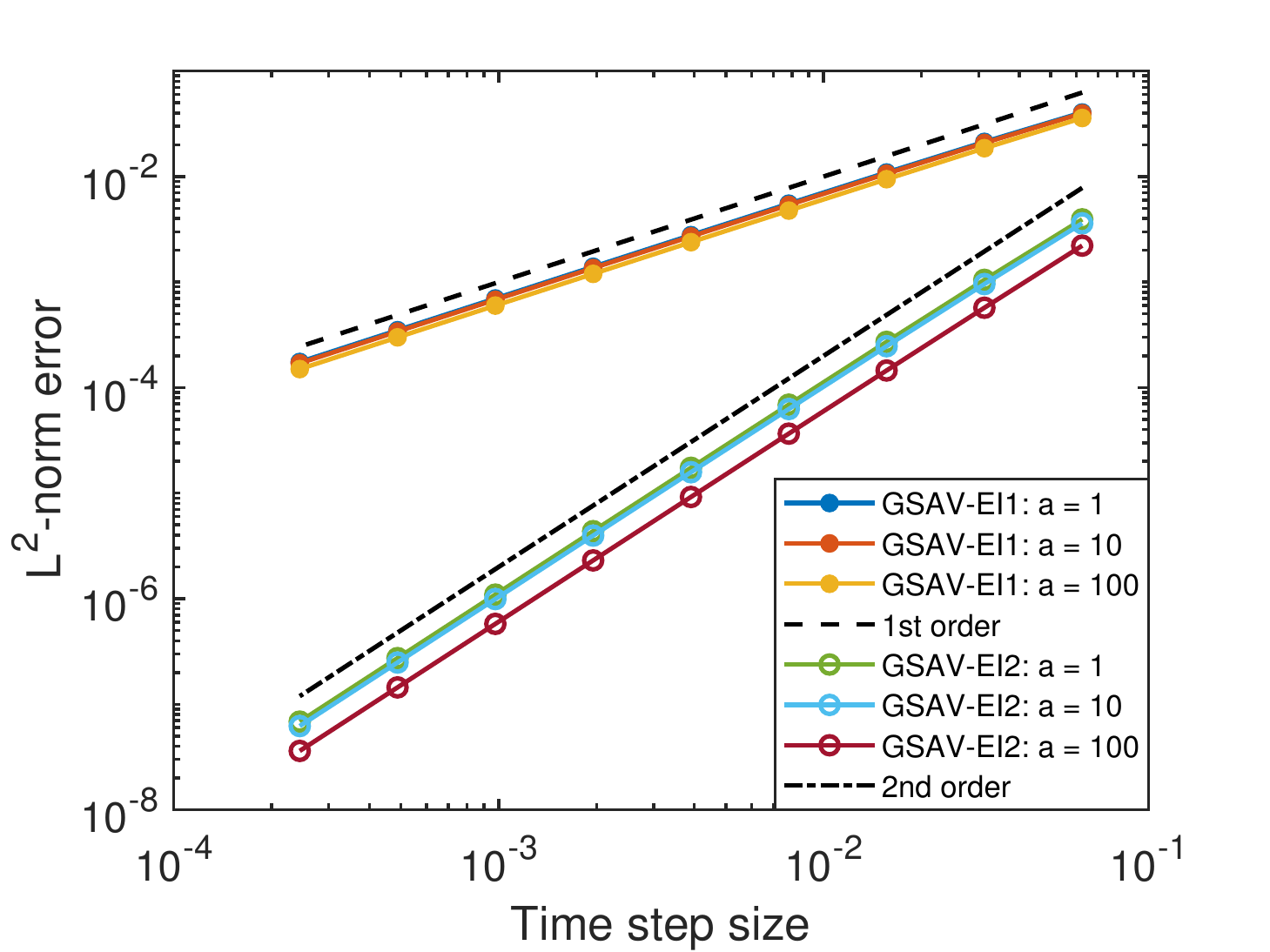}
\includegraphics[width=0.48\textwidth]{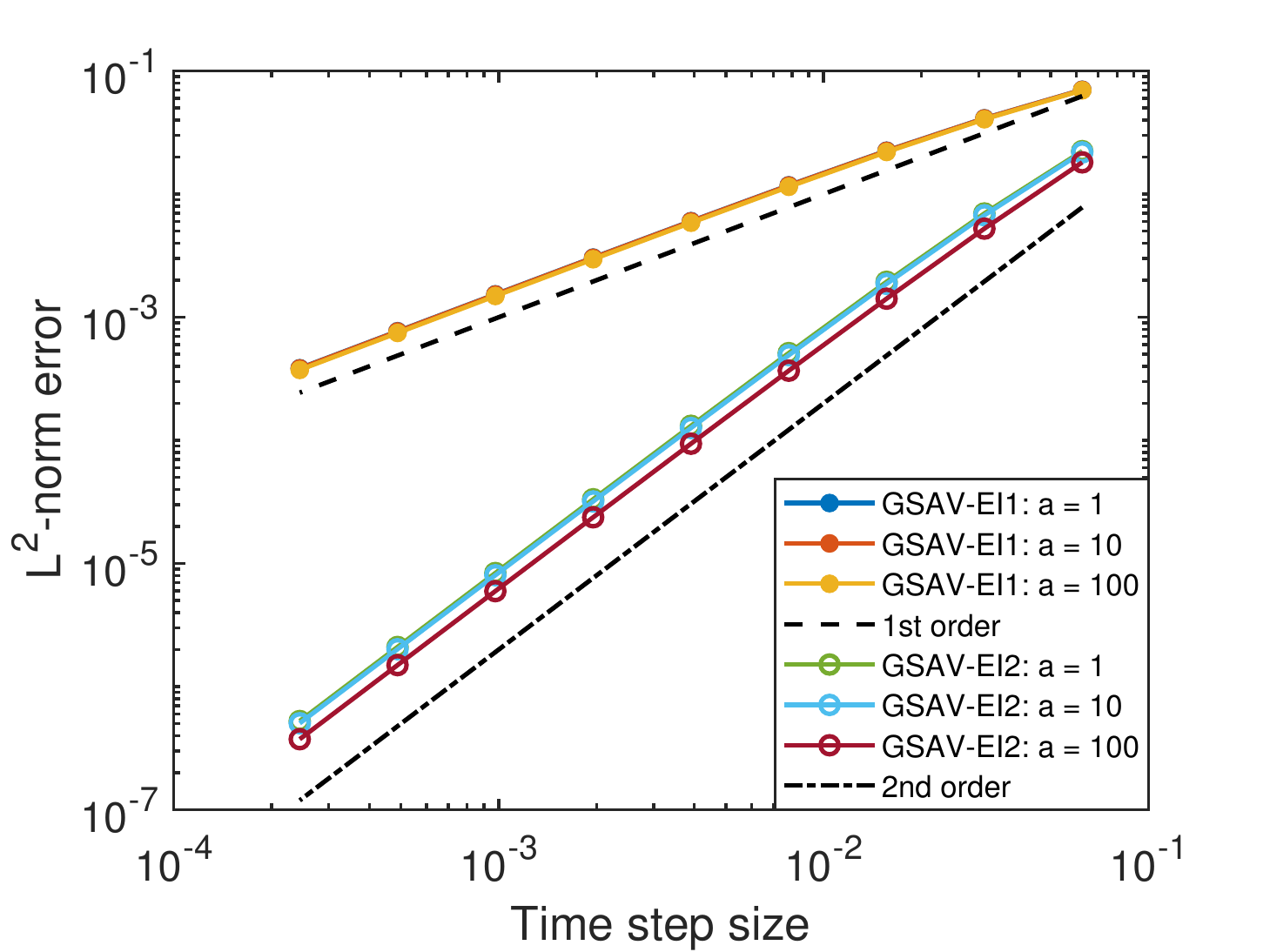}}\vspace{-0.2cm}
\caption{The $L^2$-norm errors vs. the time step sizes produced by the GSAV-EI1 and GSAV-EI2 schemes with the spatial mesh of $h=1/2048$ for the equation \eqref{AllenCahn}.
Left: the double-well potential \eqref{f_dw}; right: the Flory--Huggins potential \eqref{f_fh}.}
\label{fig_conv}
\end{figure}

\subsection{Unconditional preservation of MBP and energy dissipation law}

We numerically verify the MBP and the energy dissipation law of the proposed GSAV-EI1 and GSAV-EI2 schemes
by simulating the phase transition process beginning with a random state.
Though the discrete energy dissipation law is proved with respect to the slightly modified energy \eqref{modified_energy},
we are more concerned about the original energy defined by \eqref{dis_energy}
since it reflects the real physical mechanism of the dynamic process.
We consider the equation \eqref{AllenCahn}
on the uniform spatial mesh with $h=1/512$. Different from  the previous convergence  tests,
the initial state is generated by random numbers ranging from $-0.8$ to $0.8$ on each mesh point, thus it has
highly oscillated values.

We compute the numerical solutions by the GSAV-EI1 and GSAV-EI2 schemes
with $\dt = 0.01$ and various values of $a$ ($a=1$, $5$, $10$, respectively),
and treat the results obtained by the IFRK4 scheme with the time step size $\dt=10^{-4}$ as the benchmark.
First, we adopt the double-well potential \eqref{f_dw},
and the evolutions of the supremum norms and the energies of the numerical solutions
are shown in Figure \ref{fig_polytest}.
Obviously, the MBP and the energy dissipation law are preserved perfectly.
In addition, we observe that the smaller $a$ produces slightly  more accurate numerical solutions in this case.
This behavior is opposite to that  with smooth initial value
shown in the convergence tests. Then, we consider the Flory--Huggins potential \eqref{f_fh}
and Figure \ref{fig_logtest} presents the evolutions of the supremum norms and the energies of the numerical solutions.
Similar to the double-well potential case, the preservation of
the MBP and the energy dissipation law are obvious,
and the smaller value of $a$ in \eqref{test_sigma} yields  slightly  more accurate numerical solution.

\begin{figure}[!ht]
\centerline{
\includegraphics[width=0.43\textwidth]{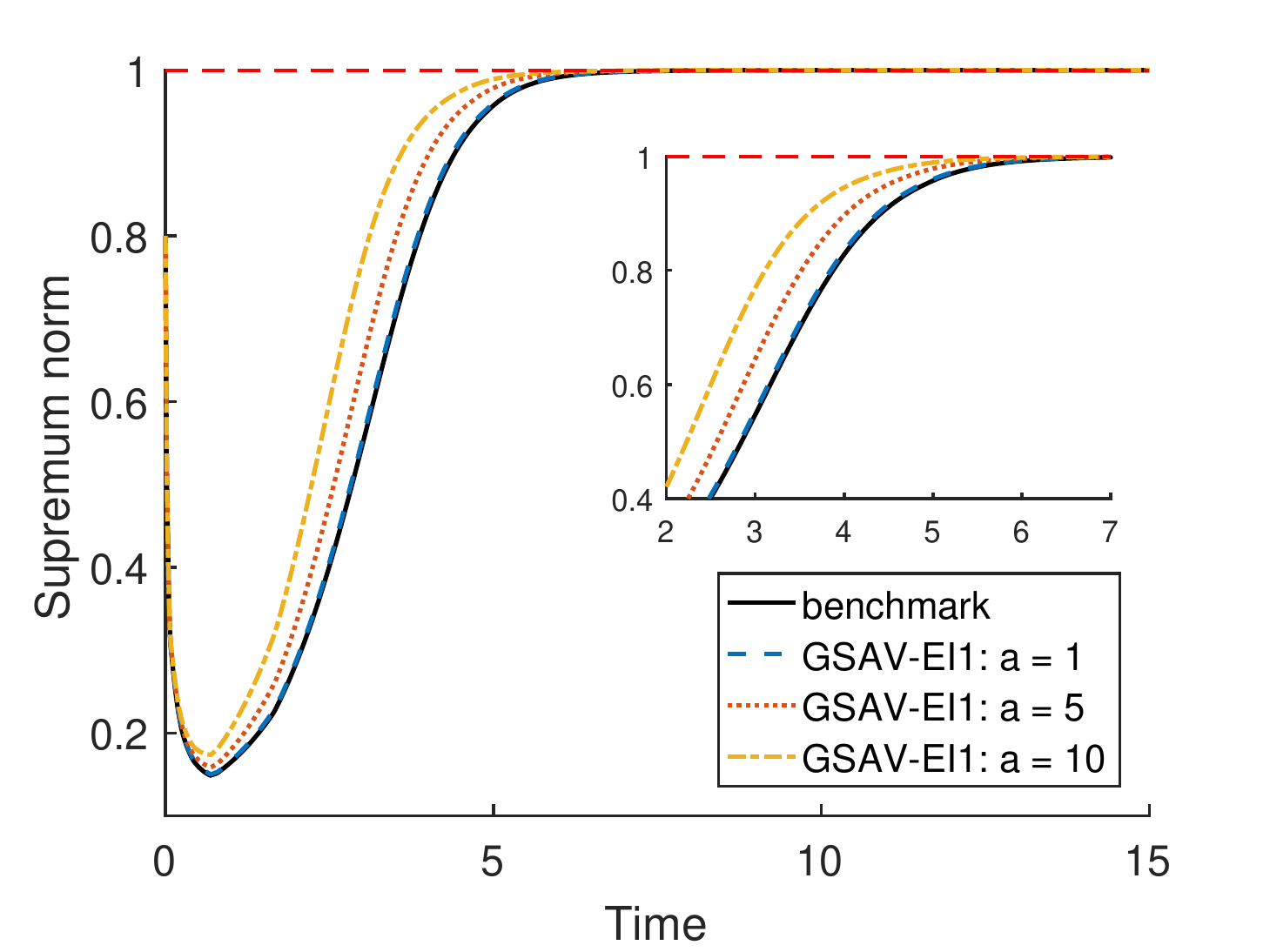}
\includegraphics[width=0.43\textwidth]{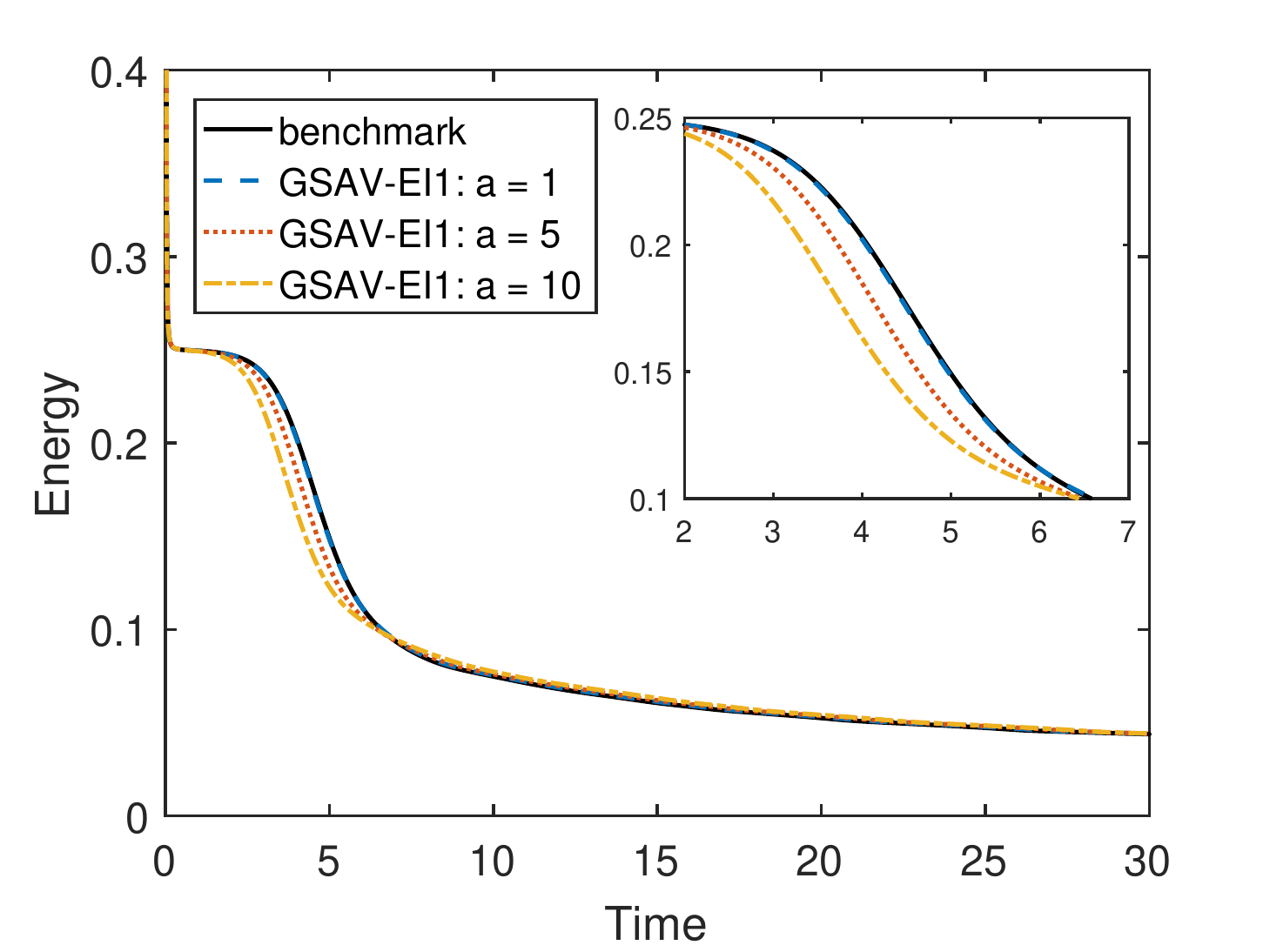}}
\centerline{
\includegraphics[width=0.43\textwidth]{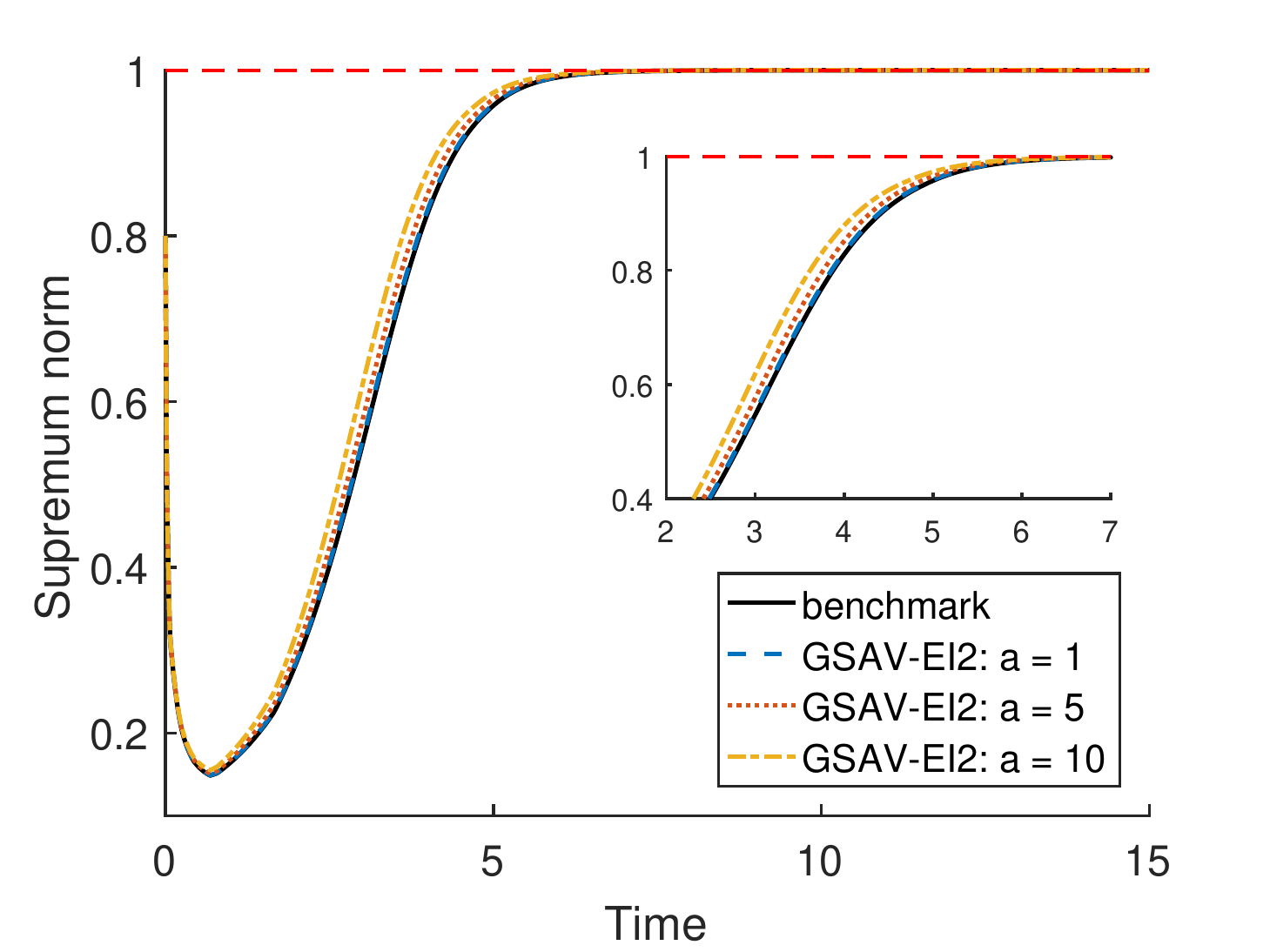}
\includegraphics[width=0.43\textwidth]{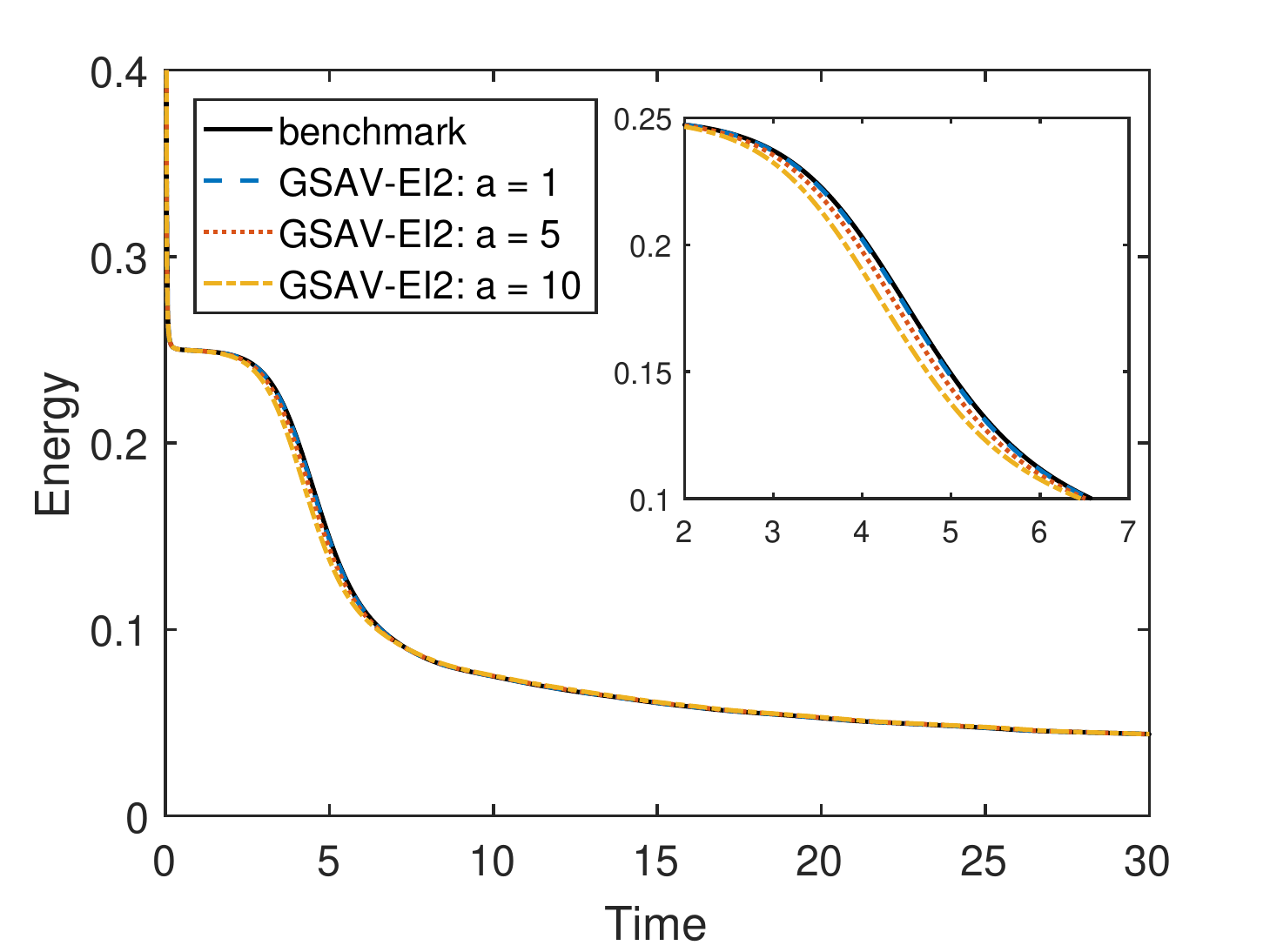}}\vspace{-0.2cm}
\caption{Evolutions of the supremum norms and the energies of simulated solutions computed by the GSAV-EI1 (top row) and
GSAV-EI2 (bottom row) schemes with $\dt = 0.01$ for the equation \eqref{AllenCahn}
with the double-well potential \eqref{f_dw}.}
\label{fig_polytest}
\end{figure}

\begin{figure}[!ht]
\centerline{
\includegraphics[width=0.43\textwidth]{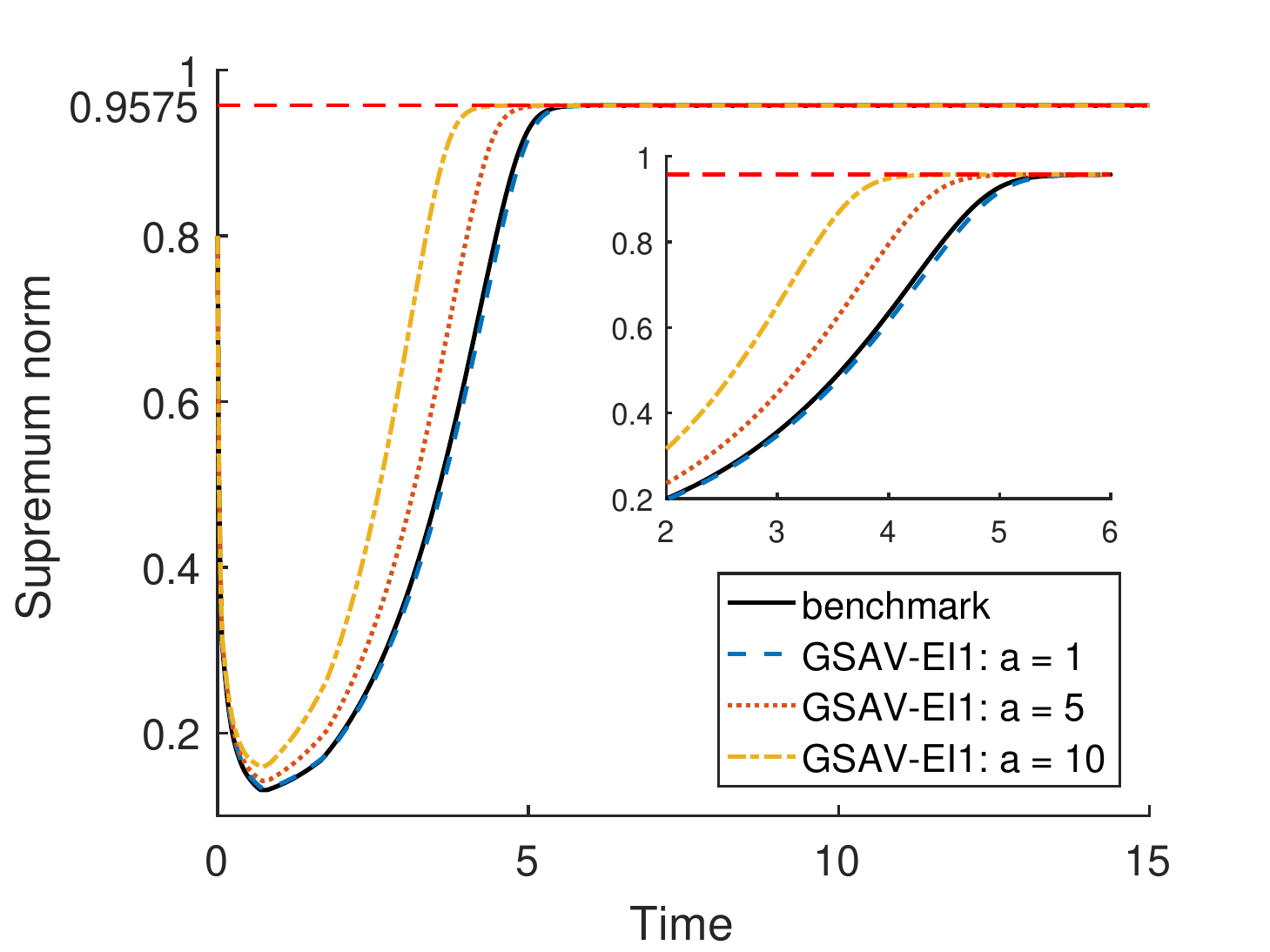}
\includegraphics[width=0.43\textwidth]{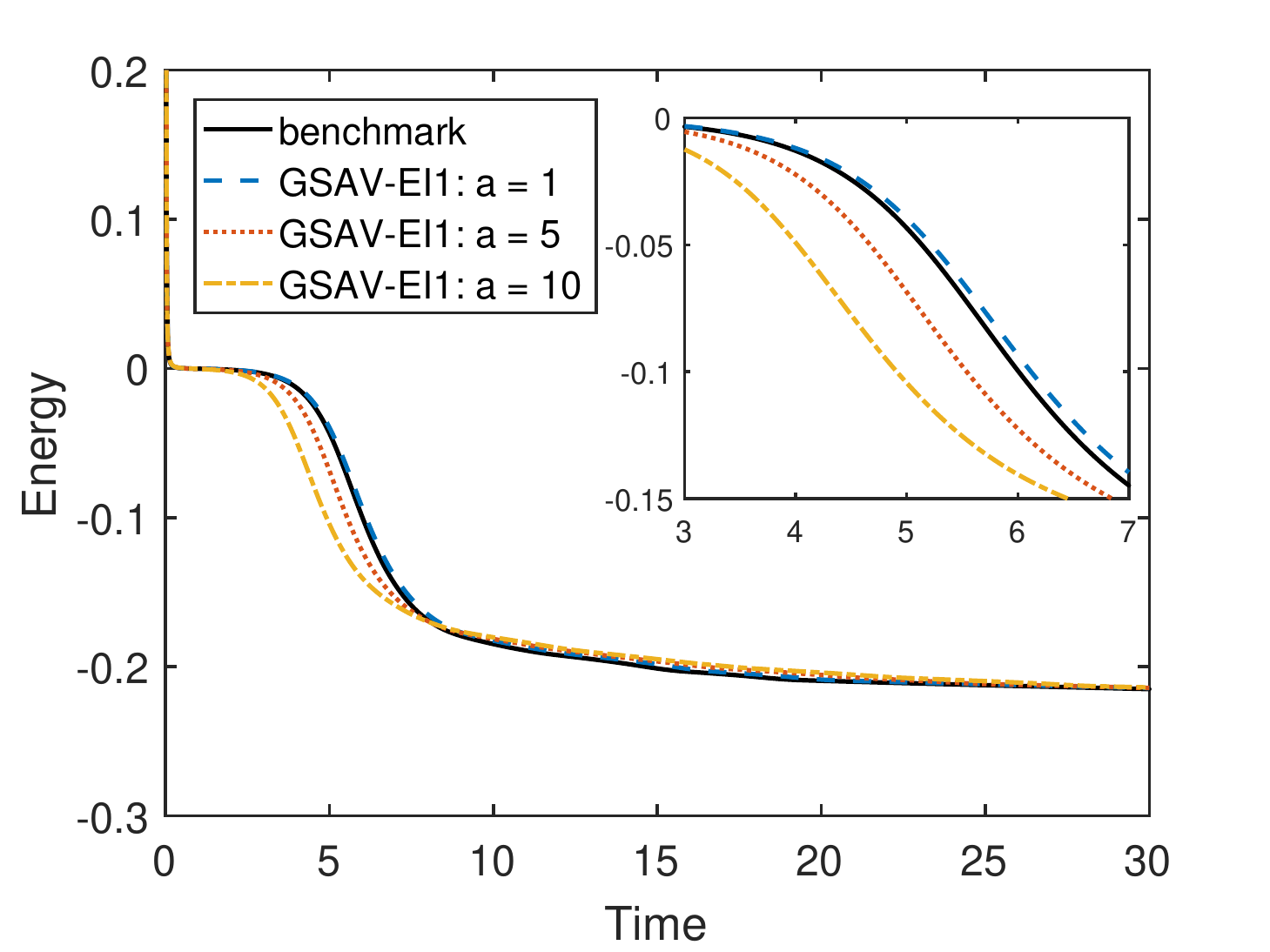}}
\centerline{
\includegraphics[width=0.43\textwidth]{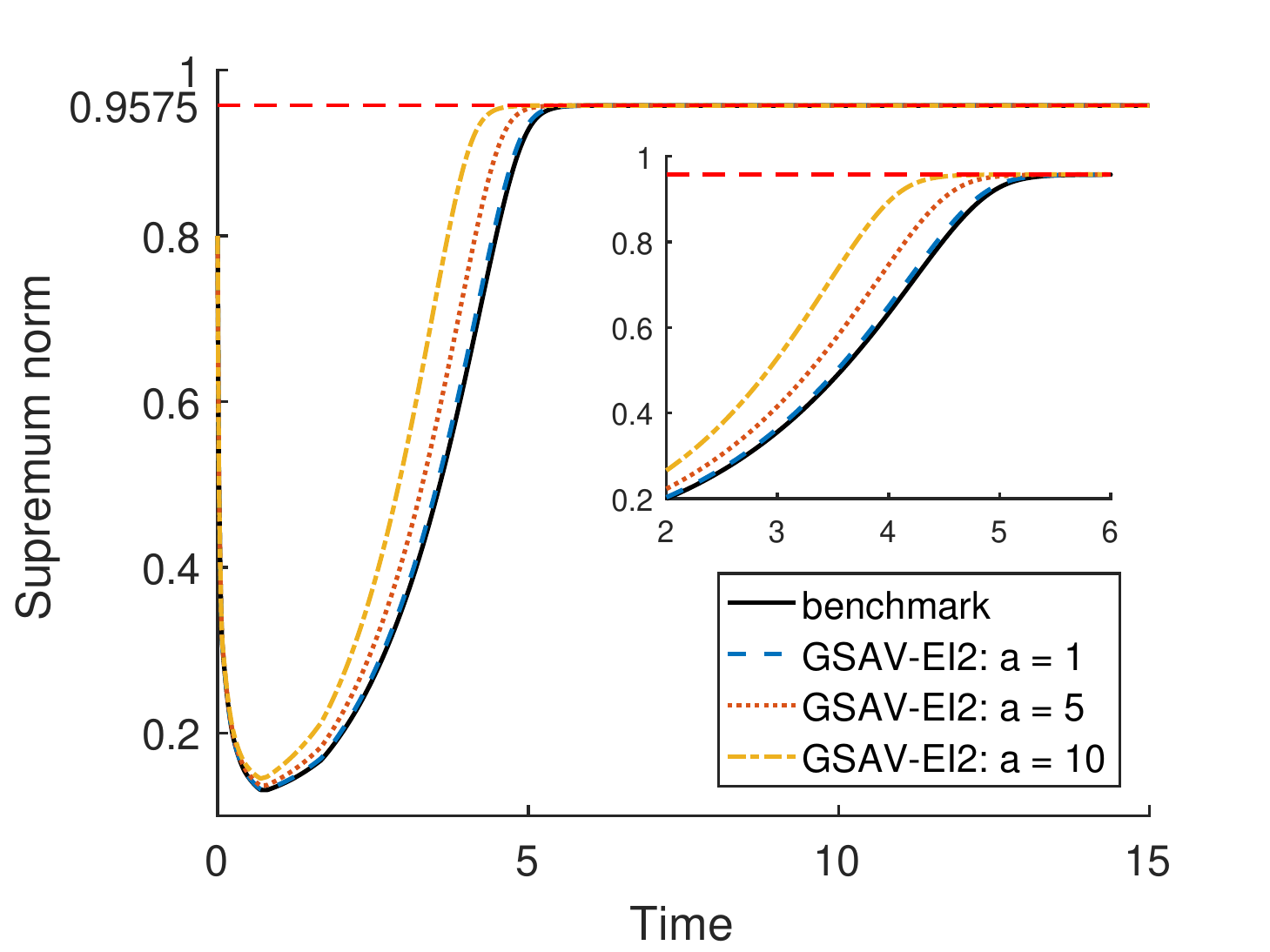}
\includegraphics[width=0.43\textwidth]{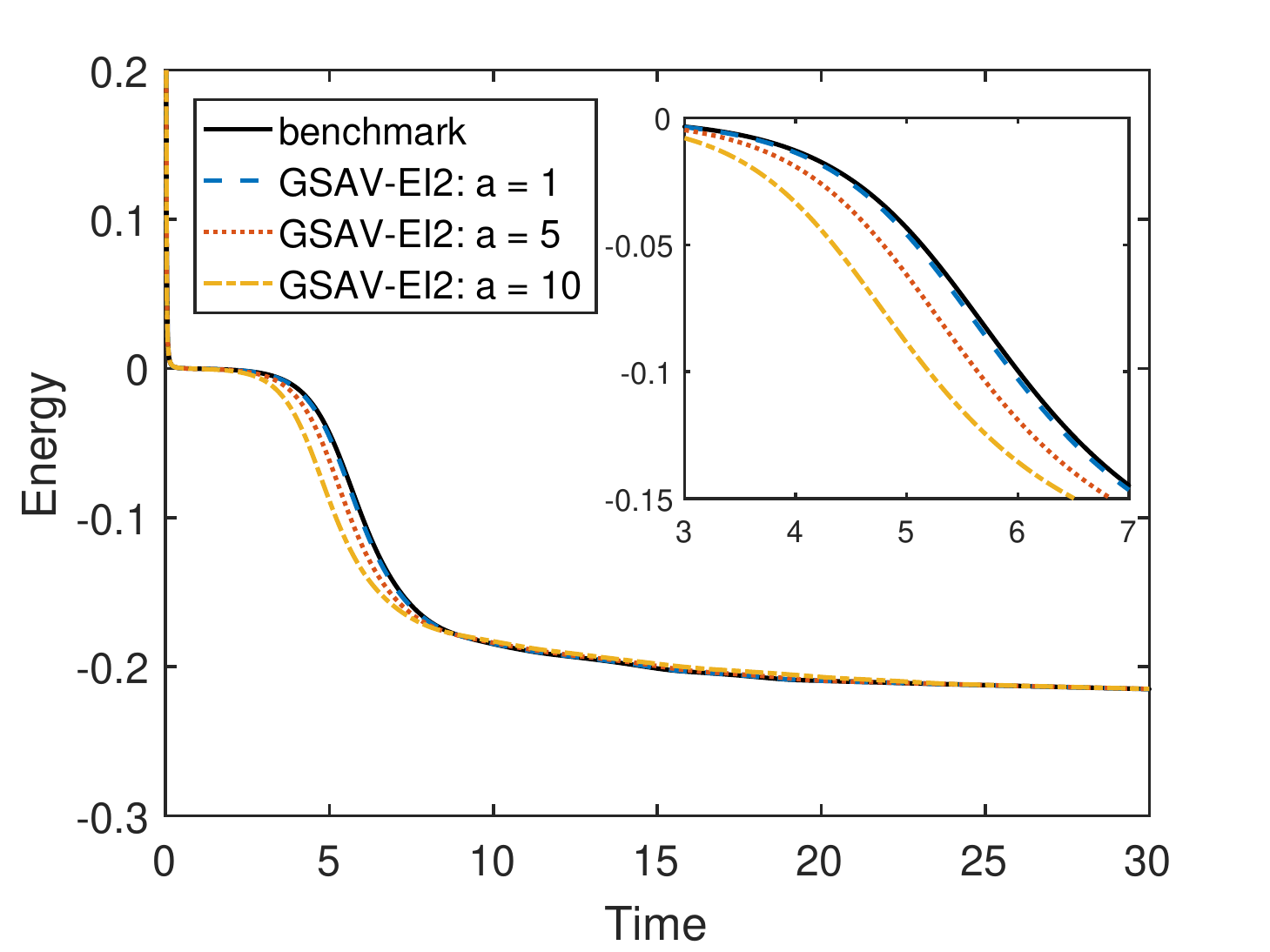}}\vspace{-0.2cm}
\caption{Evolutions of the supremum norms and the energies of simulated solutions computed by the GSAV-EI1 (top row) and
GSAV-EI2 (bottom row) schemes with $\dt = 0.01$ for the equation \eqref{AllenCahn}
with the  Flory--Huggins potential \eqref{f_fh}.}
\label{fig_logtest}
\end{figure}

Next, we repeat the above experiments  by choosing the ($10$ times) larger time step size $\dt=0.1$.
{We can observe the similar results that the MBP and the energy dissipation law are still preserved well
although the large time step size leads to a little less accurate numerical solutions.}

{
\subsection{Adaptive time-stepping and long-time simulation}
\label{sect_neumann}

Since the proposed two GSAV-EI schemes \eqref{eq_eisav1} and \eqref{eq_eisav2} are both one-step approaches,
without sacrificing the energy dissipation law and the MBP,
they can also be applied on
a set of nonuniform temporal nodes $\{t_n\}_{n\ge0}$ with $t_0=0$ and $t_{n+1}=t_n+\dt_{n+1}$,
where the time step size $\dt_{n+1}$ varies in $n$.
Let us consider \eqref{AllenCahn} with $\eps=0.01$ and the Flory--Huggins potential \eqref{f_fh} again but with the homogeneous Neumann boundary condition.
The spatial mesh and the random initial value are the same as aforementioned.
We adopt the GSAV-EI2 scheme \eqref{eq_eisav2} with $\sigma(x)=\e^x$ and variable time step sizes
$\dt_{n+1}$ updated by using the approach from \cite{QiaoZhTa11}
\[
\dt_{n+1} = \max\Big\{ \dt_{\min}, \frac{\dt_{\max}}{\sqrt{1+\alpha|\d_tE_h(u^n)|^2}} \Big\},
\]
where $\d_tE_h(u^n)=(E_h(u^n)-E_h(u^{n-1}))/\dt_n$ and $\alpha>0$ is a constant parameter.
Here, we choose the minimal and maximal time step sizes as $\dt_{\min}=0.0001$ and $\dt_{\max}=0.1$ respectively, and set $\alpha=10^5$ as done in \cite{QiaoZhTa11}.
For comparison, we also conduct the simulation by the GSAV-EI2 scheme  with the uniform time step size $\dt=0.01$.

The coarsening dynamics reach the steady state at around $t=3000$. We find that
the CPU time for the whole simulation with adaptive time-stepping is only about $10\%$ of that with uniform time step size.
One can observe from the left and middle graphs in Figure \ref{fig_neu_step} that
the energy dissipation law and the MBP are preserved perfectly.
The right graph in Figure \ref{fig_neu_step} plots the evolution of the adaptive time step sizes.
In the time interval $[0,20]$, the time step size varies significantly and sometimes are very small since the energy decreases rapidly at most of the time.
Then after $t=20$, the energy changes more and more slowly and the time step size is magnified gradually.
When $t>200$, the time step size remains around $0.1$ (not shown in the graph),
and we find that, although the large step size is used for this period,
the relative error of the energy is only about $1\%$ in comparison with the case of uniform time step size.
These results show that the adaptive time-stepping strategy can  greatly help accelerate the computation without sacrificing the desired properties and the accuracy.

\begin{figure}[!ht]
\centerline{
\includegraphics[width=0.36\textwidth]{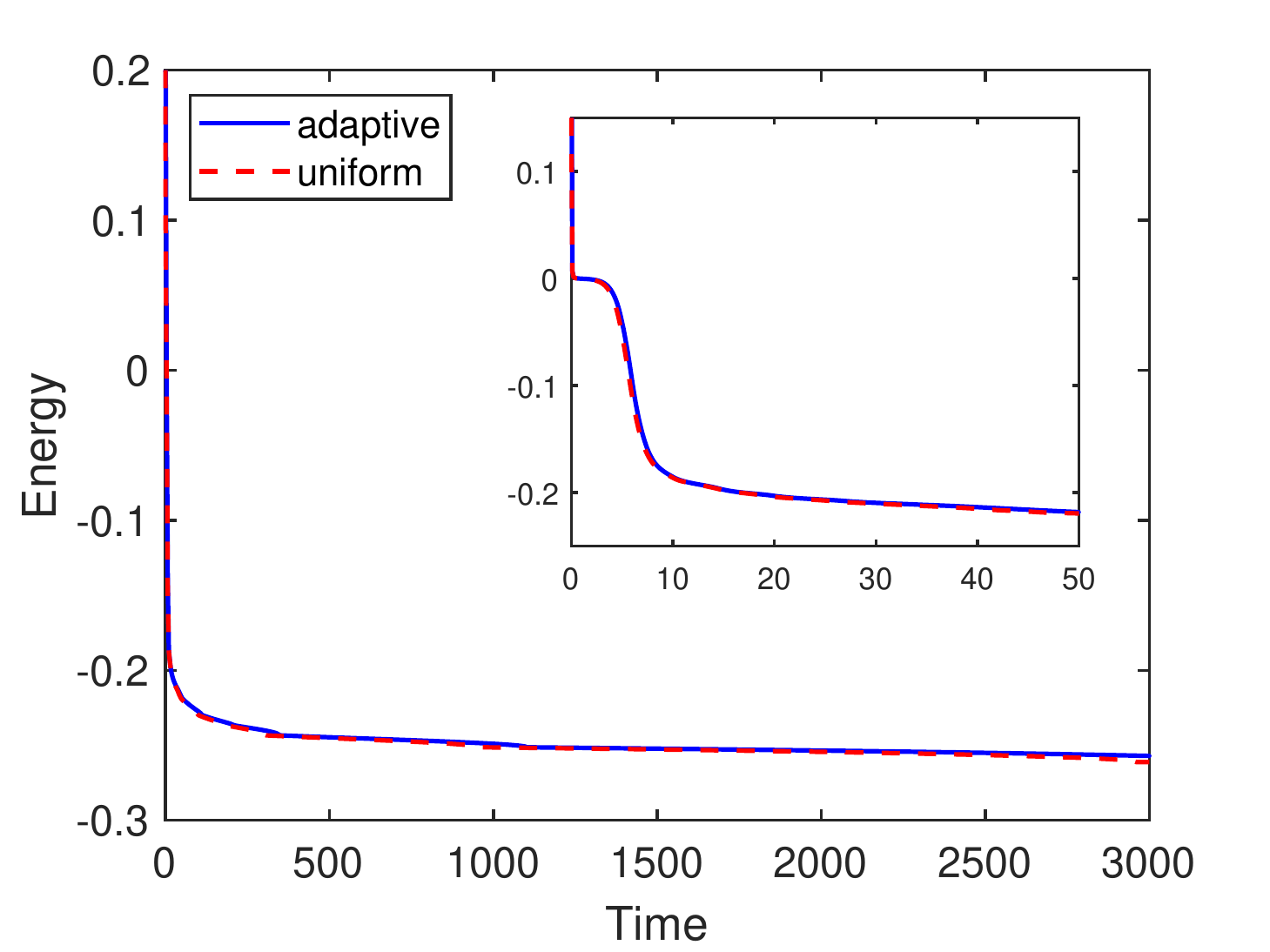} \hspace{-0.5cm}
\includegraphics[width=0.36\textwidth]{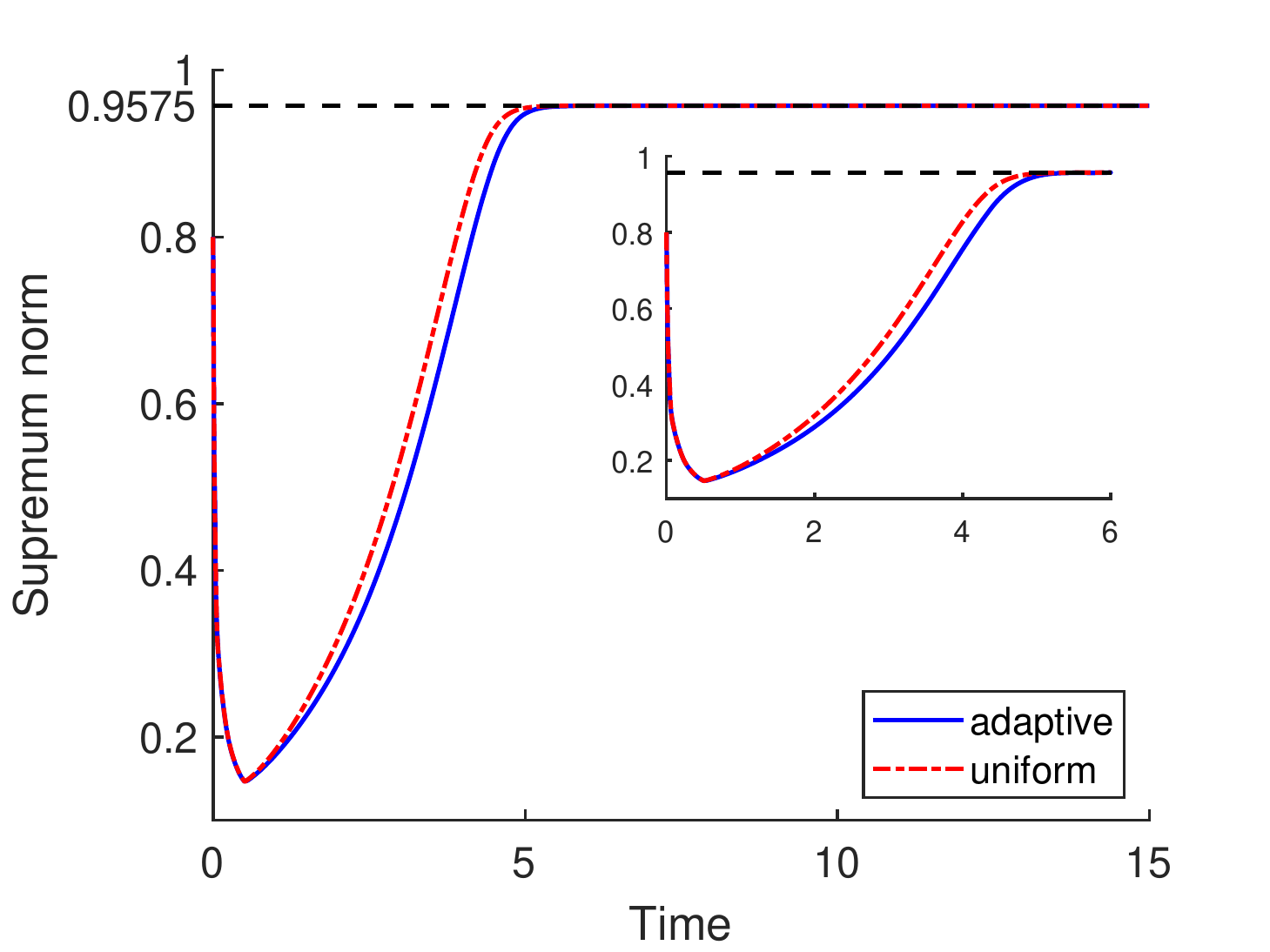} \hspace{-0.5cm}
\includegraphics[width=0.36\textwidth]{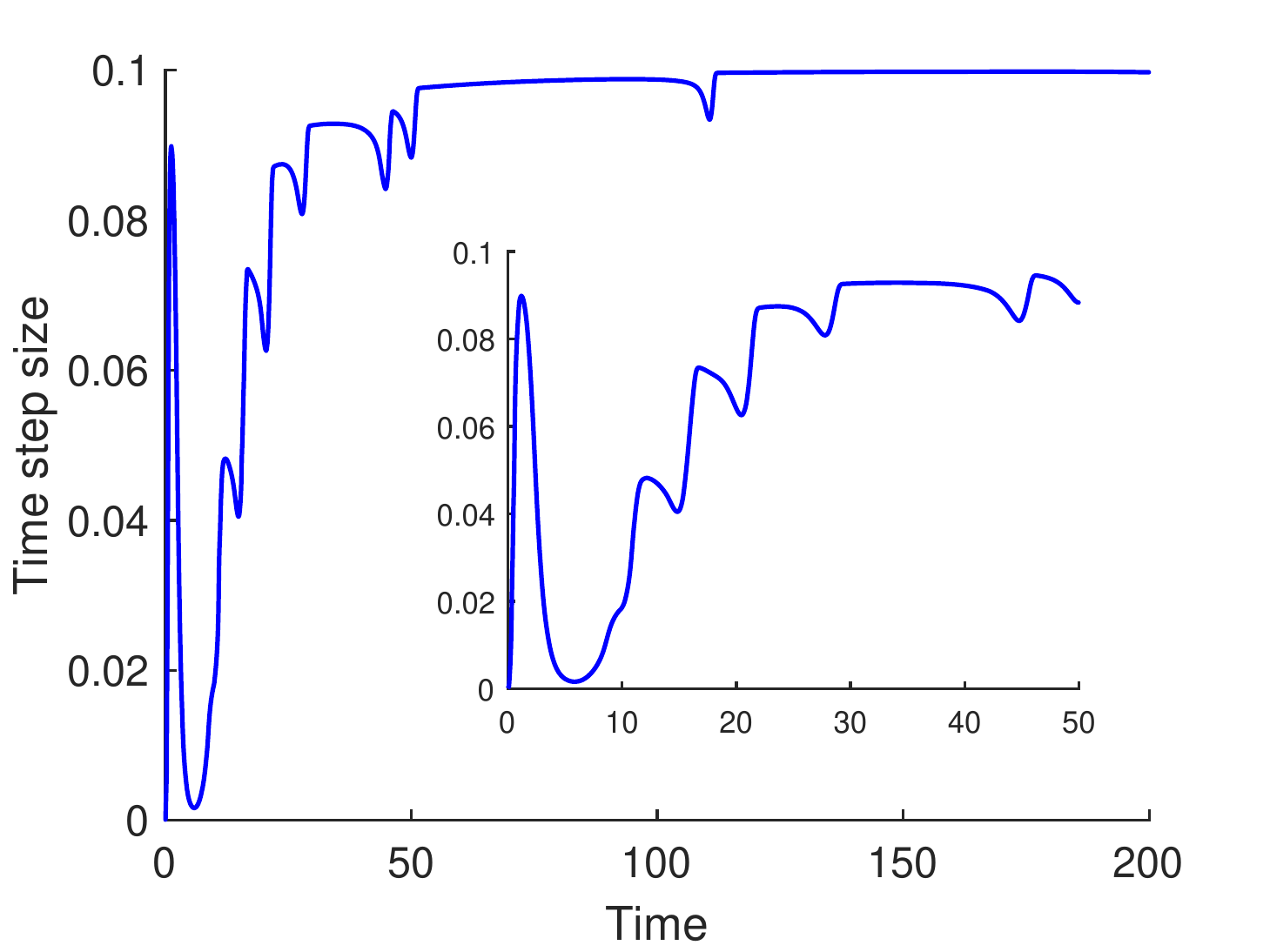}}\vspace{-0.1cm}
\caption{Evolutions of the energies (left), the supremum norms (middle), and the time step sizes (right) of simulated solutions computed by the GSAV-EI2 schemes for \eqref{AllenCahn} with homogeneous Neumann boundary condition and the Flory--Huggins potential \eqref{f_fh}.}
\label{fig_neu_step}
\end{figure}

\begin{figure}[!ht]
\centerline{
\includegraphics[height=0.135\textheight]{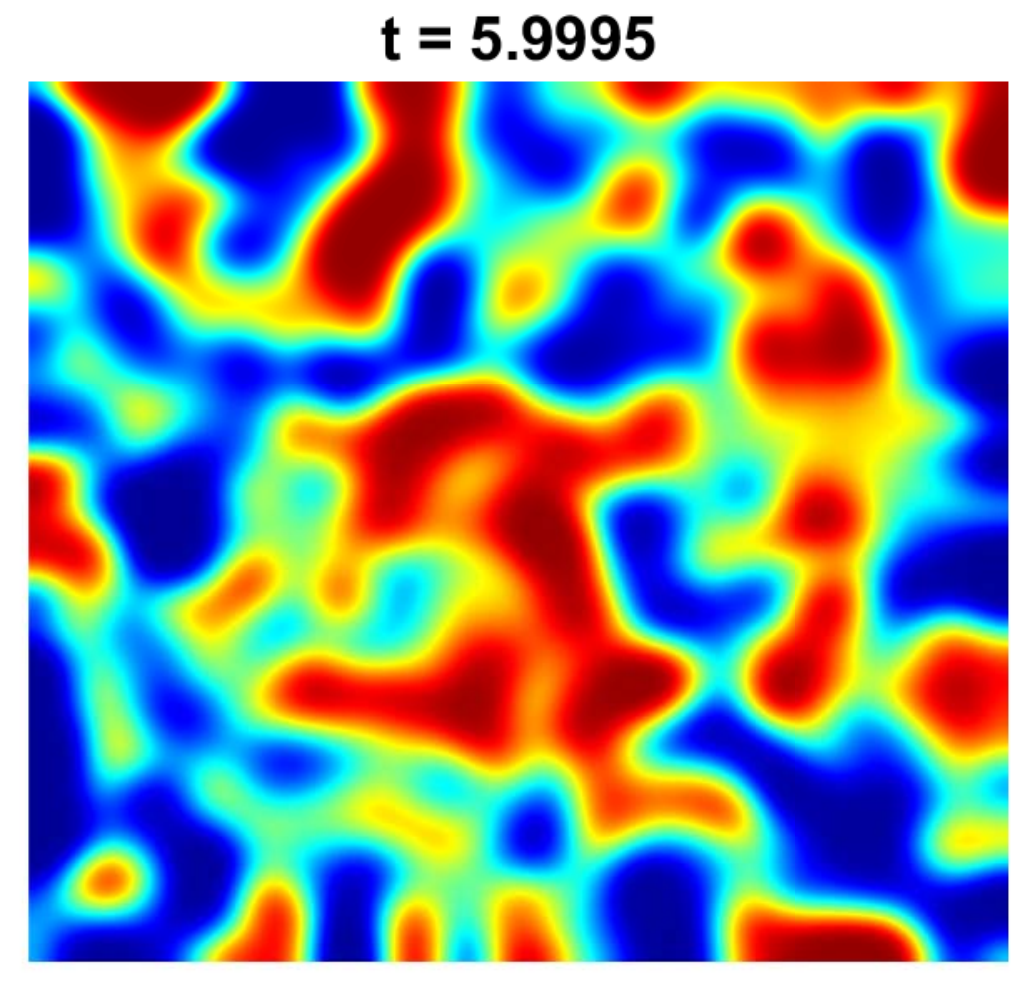}\hspace{-0.1cm}
\includegraphics[height=0.135\textheight]{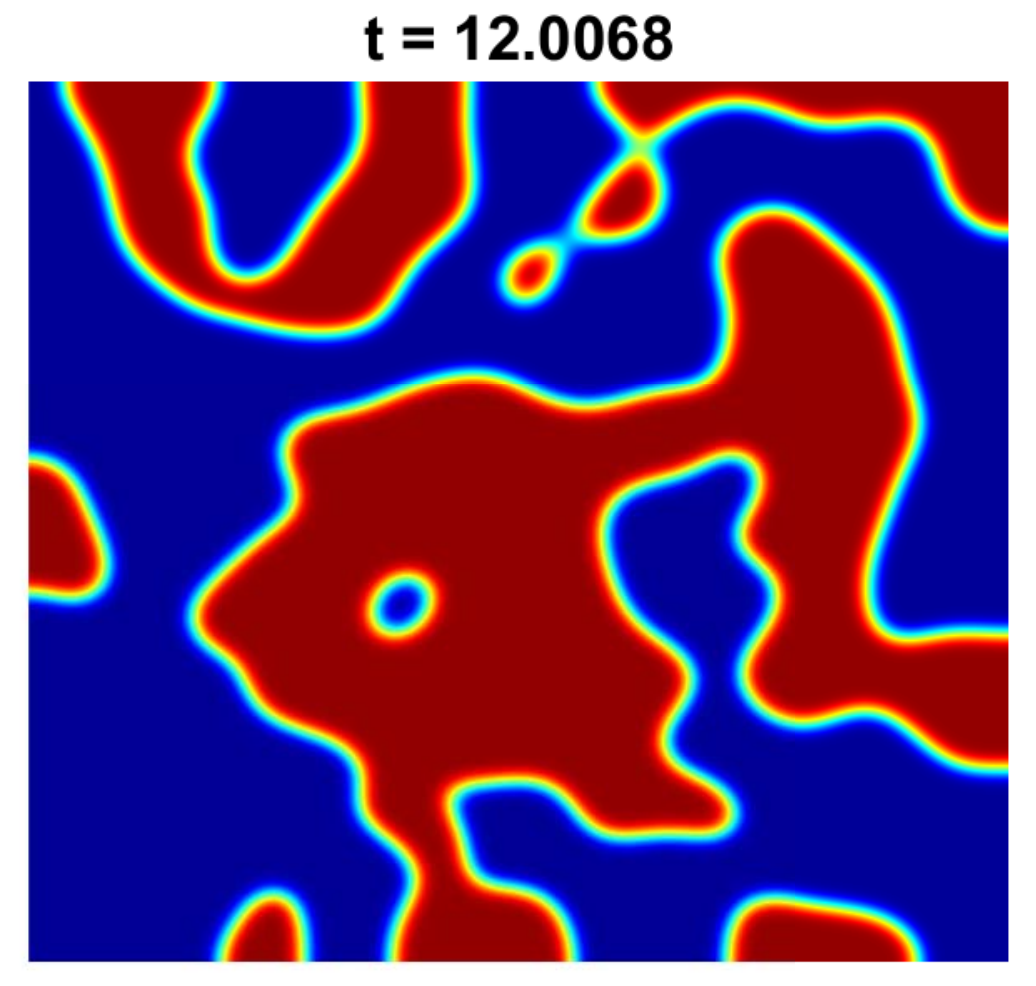}\hspace{-0.1cm}
\includegraphics[height=0.135\textheight]{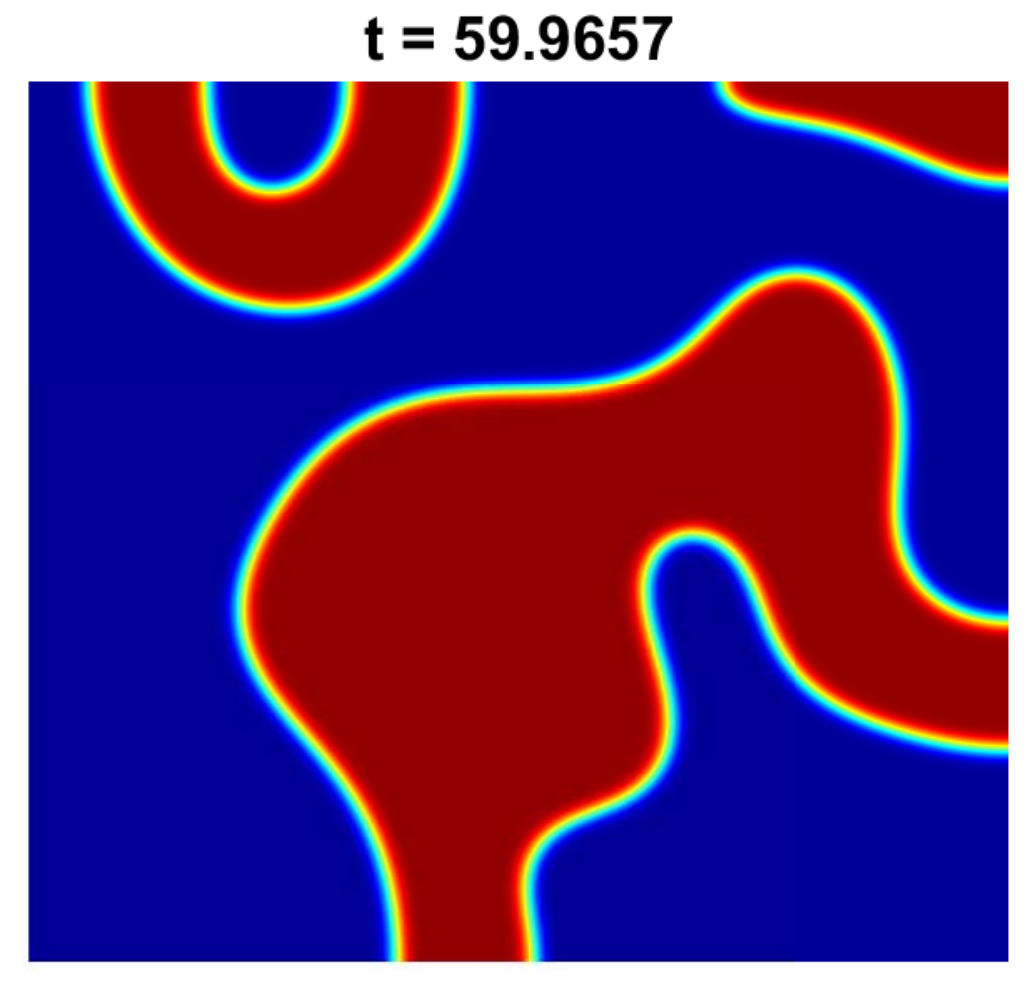}\hspace{-0.1cm}
\includegraphics[height=0.135\textheight]{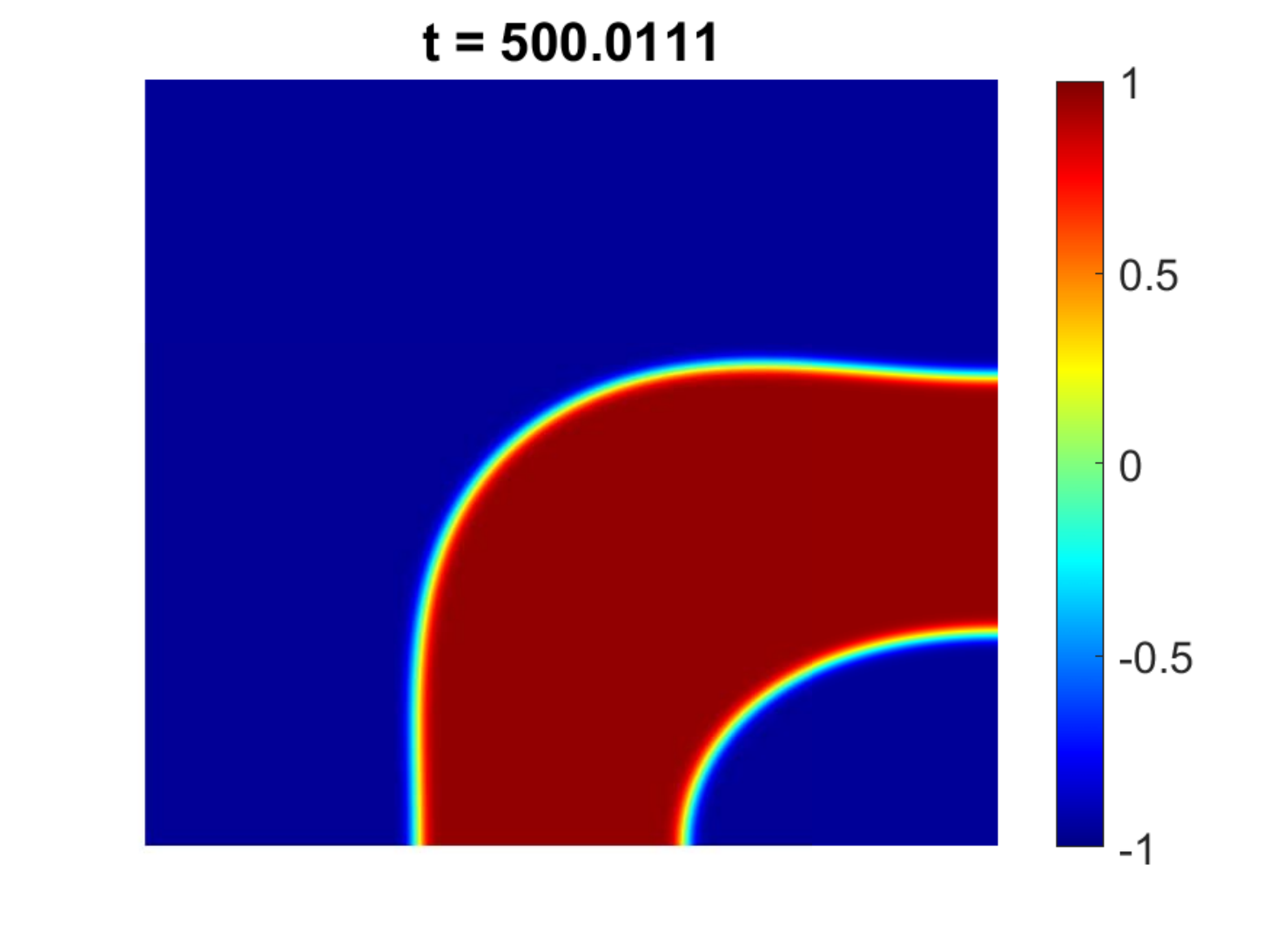}}
\centerline{
\includegraphics[height=0.135\textheight]{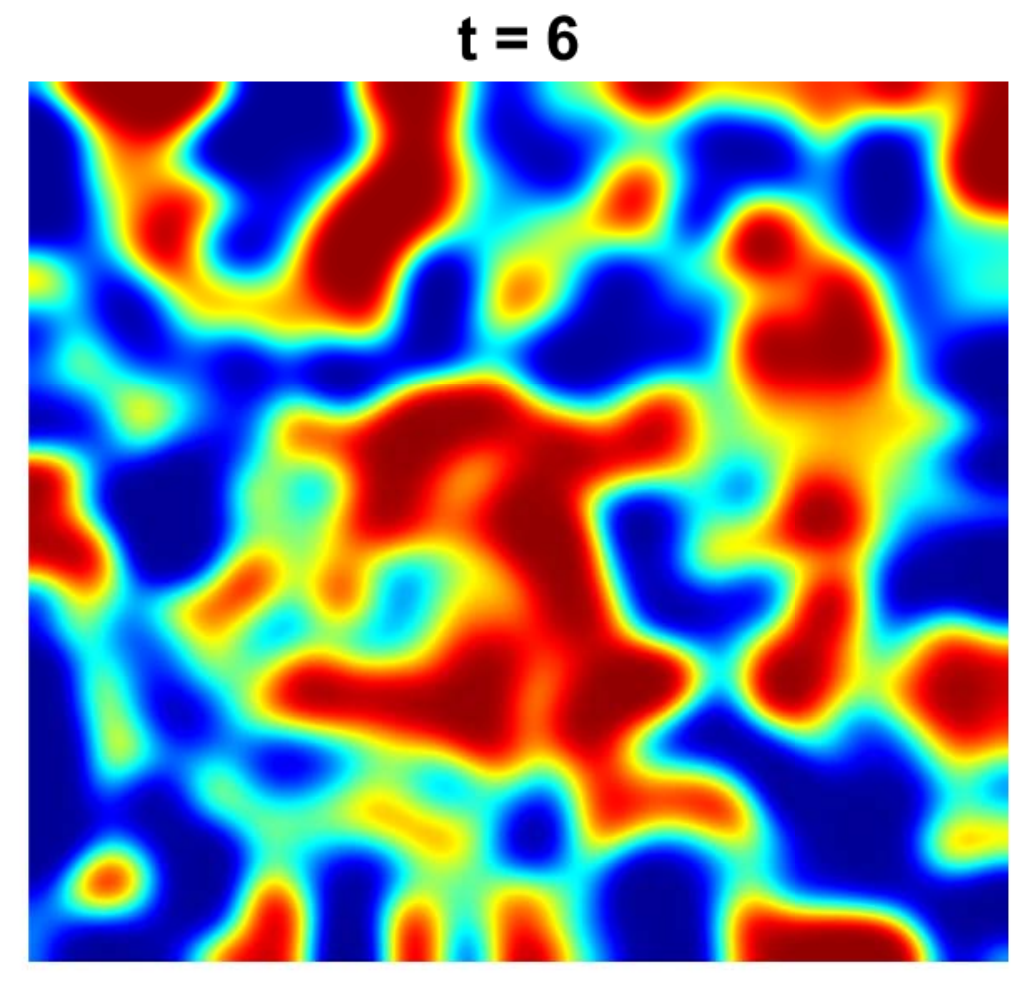}\hspace{-0.1cm}
\includegraphics[height=0.135\textheight]{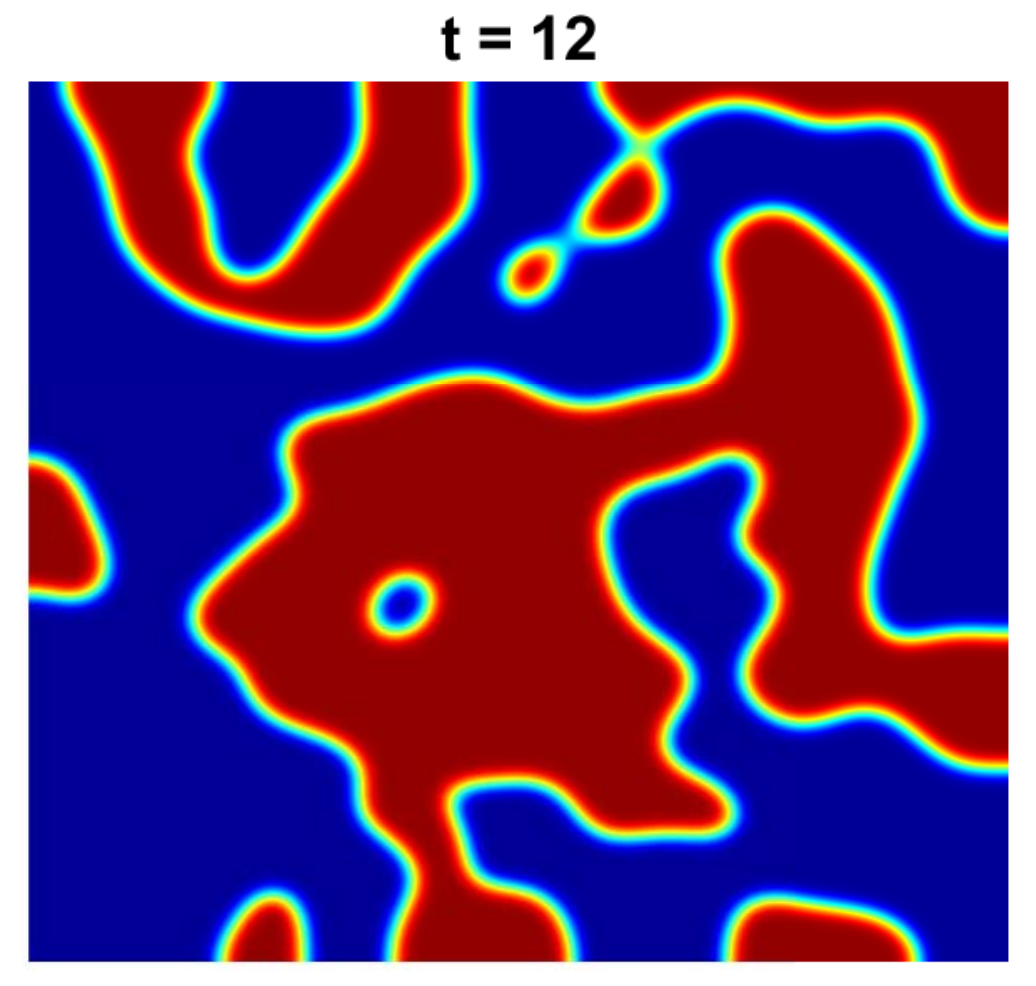}\hspace{-0.1cm}
\includegraphics[height=0.135\textheight]{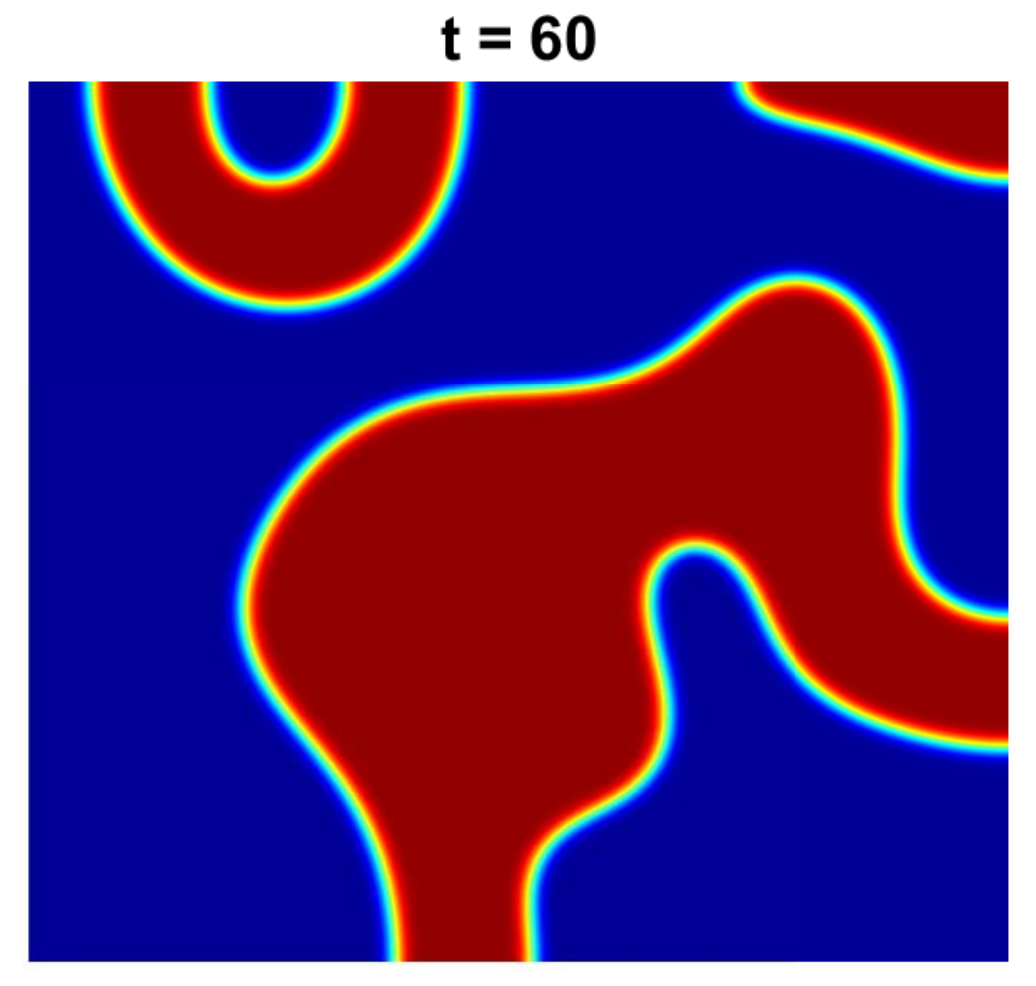}\hspace{-0.1cm}
\includegraphics[height=0.135\textheight]{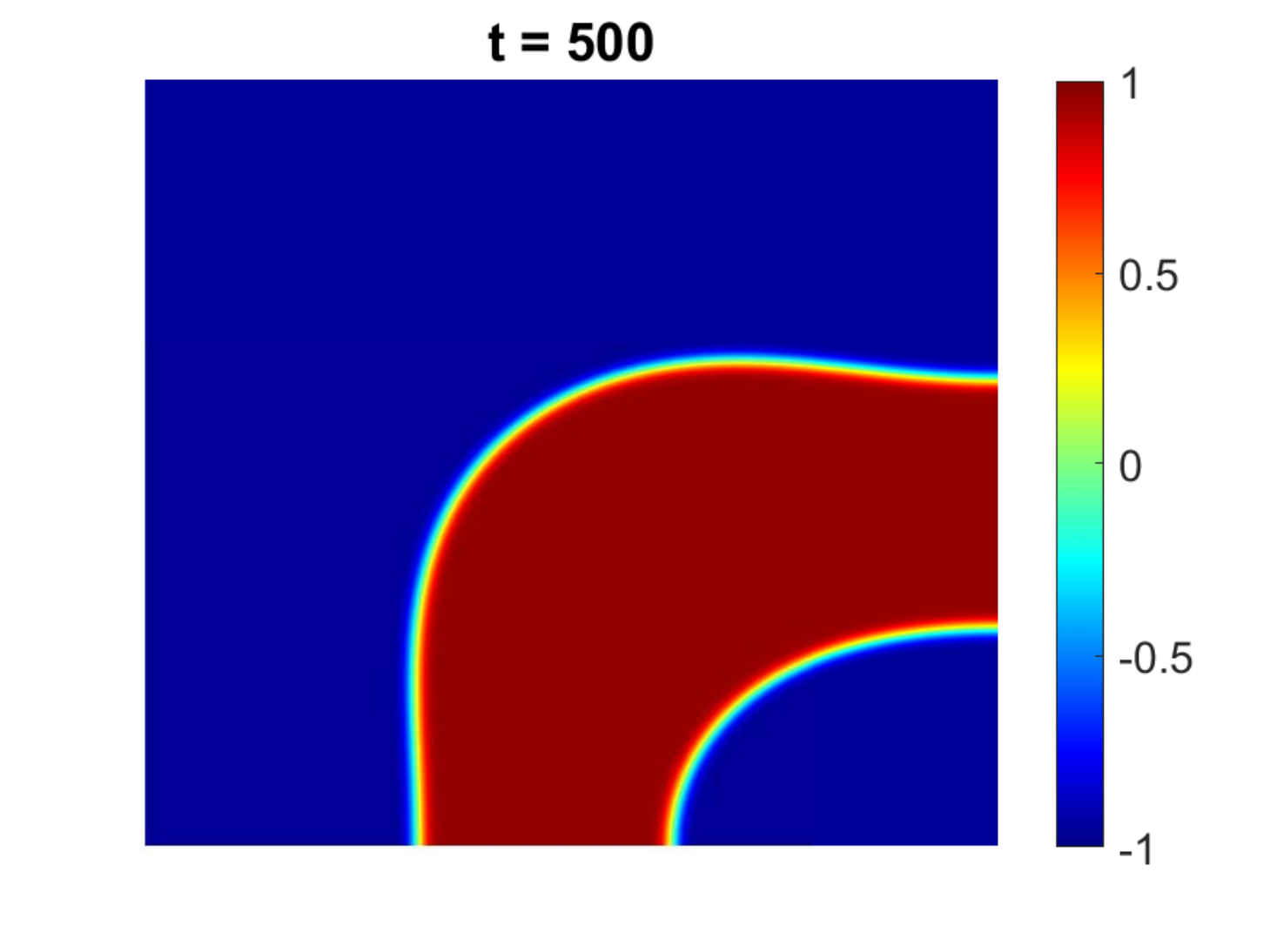}}\vspace{-0.1cm}
\caption{Snapshots of the phase transition generated by the GSAV-EI2 schemes with adaptive time step (top) and uniform time step (bottom) for \eqref{AllenCahn} with homogeneous Neumann boundary condition and the Flory--Huggins potential \eqref{f_fh}.}
\label{fig_neu_snapshot}
\end{figure}
}

\section{Concluding remarks}
\label{sect_conclusion}

In this paper, we study the numerical schemes preserving both the energy dissipation law and the MBP unconditionally
for a class of Allen--Cahn type gradient flows
by combining the exponential integrator method and the generalized SAV approach.
With the appropriate stabilization terms,
we develop first- and second-order GSAV-EI schemes
and prove their unconditional preservation of the energy dissipation law and the MBP in the time discrete sense,
as well as their optimal temporal  error estimates under fixed spatial mesh.
Different from most existing numerical schemes,
the energy dissipation law and the MBP of the proposed GSAV-EI schemes can be established in parallel,
which provides more flexibility to apply the proposed schemes to other types of gradient flow equations
to preserve some important physical properties. We  also note that the fully-discrete error estimate for the case that
the spatial mesh size and the time step size change simultaneously is still
an open question for the proposed GSAV-EI schemes and surely worthy of further study.
{A major difficulty comes from the issue that
the matrix exponential $\e^{\dt L_\kp^n}$, defined by a power series of the sparse matrix $\dt L_\kp^n$,
is dense and affects the solution globally.
In particular, to estimate the temporal  truncation error of the GSAV-EI2 scheme (Lemma \ref{lem_etd2_error}),
an $h$-dependent bound is inevitable, and thus we fix the spatial mesh size to regard such bound as a constant
in this paper.}

{When constructing the second-order GSAV-EI scheme \eqref{eq_eisav2},
we approximate the term $N_\kp^{n+\frac{1}{2}}(u_h(t_n+\theta),s_h(t_n+\theta))$ in \eqref{etd2_exact}
by its value at the midpoint $\theta=\frac{\dt}{2}$
rather than its linear interpolation in $[0,\dt]$.
This allows the cancellation between the nonlinear terms
in the analysis of the energy dissipation (Theorem \ref{thm_eisav2_es}).
Instead, if we adopt the linear interpolation as usually done for the RK2 method,
two terms involving the numerical solutions at $t_n$ and $t_{n+1}$ will be included
with the $\phi$-functions of $\dt L_\kp^{n+\frac{1}{2}}$ as the coefficients,
which makes the cancellation unavailable due to the different coefficients
between the updating formula for $u^{n+1}$ and that for $s^{n+1}$.
For the similar reason, it is an open question whether higher-order GSAV-EI schemes exist
in either RK or multistep form, although there have been third-order multistep schemes based on the standard ETD method
for the epitaxial thin film model \cite{ChenLiWaWaWa20,ChengQiWa19}.
}

It also remains interesting on how to choose the function $\sigma$ appropriately for the GSAV-EI schemes in practical applications.
As we explain in Remark \ref{rmk_sgexp},
we only use the exponential function \eqref{test_sigma} in Section \ref{sect_experiment}
since the differences can hardly be observed for the typical choices of $\sigma$ given in Remark \ref{sgchoice}
for the specific problems we consider in the numerical experiments.
However, their performance could be significantly different for some other situations and gradient flows, and
more careful investigation is needed.  In addition, the effect  of the parameter $a$ on the numerical errors seems completely
opposite for the smooth and non-smooth  initial data based on our observation from numerical experiments, and
such phenomenon also deserves deeper study.



\appendix

\section{Proof of Lemma \ref{lem_etd2_error}}
\label{app1}

\begin{proof}
From \eqref{etd2_exact}, we have
\begin{align*}
u_{h,e}(t_{n+1}) & = \e^{-\dt L_\kp^{n+\frac{1}{2}}} u_{h,e}(t_n) + \int_0^\dt \e^{-(\dt-\theta)L_\kp^{n+\frac{1}{2}}}
N_\kp^{n+\frac{1}{2}}(u_{h,e}(t_n+\theta),s_{h,e}(t_n+\theta)) \, \d \theta \\
& = \e^{-\dt L_\kp^{n+\frac{1}{2}}} u_{h,e}(t_n)
+ \bigg(\int_0^\dt \e^{-(\dt-\theta)L_\kp^{n+\frac{1}{2}}}\, \d\theta\bigg)
N_\kp^{n+\frac{1}{2}}(\widetilde{u}_{h,e}^{n+\frac{1}{2}},\widetilde{s}_{h,e}^{n+\frac{1}{2}}) + \dt R_{2u}^n,
\end{align*}
which gives \eqref{etd2_truna} with
\begin{align*}
R_{2u}^n
& = \frac{1}{\dt} \int_0^\dt \e^{-(\dt-\theta)L_\kp^{n+\frac{1}{2}}} [N_\kp^{n+\frac{1}{2}}(u_{h,e}(t_n+\theta),s_{h,e}(t_n+\theta))
- N_\kp^{n+\frac{1}{2}}(\widetilde{u}_{h,e}^{n+\frac{1}{2}},\widetilde{s}_{h,e}^{n+\frac{1}{2}})] \, \d\theta \\
& =
\frac{1}{\dt} \bigg( \int_0^\dt [N_\kp^{n+\frac{1}{2}}(u_{h,e}(t_n+\theta),s_{h,e}(t_n+\theta))
- N_\kp^{n+\frac{1}{2}}(\widetilde{u}_{h,e}^{n+\frac{1}{2}},\widetilde{s}_{h,e}^{n+\frac{1}{2}})]\, \d\theta \\
&  +\int_0^\dt (\e^{-(\dt-\theta)L_\kp^{n+\frac{1}{2}}}-I) [N_\kp^{n+\frac{1}{2}}(u_{h,e}(t_n+\theta),s_{h,e}(t_n+\theta))-N_\kp^{n+\frac{1}{2}}(\widetilde{u}_{h,e}^{n+\frac{1}{2}},\widetilde{s}_{h,e}^{n+\frac{1}{2}})]\, \d\theta \bigg) \\
& =: \frac{1}{\dt} (R_{2u}^{n,1} + R_{2u}^{n,2}).
\end{align*}
For the function $N_\kp^{n+\frac{1}{2}}(v,r)$ defined in \eqref{Nkp_def2},
let us  denote by $\nabla_v N_\kp^{n+\frac{1}{2}}(v,r)$ and $\partial_r N_\kp^{n+\frac{1}{2}}(v,r)$
the derivatives of $N_\kp^{n+\frac{1}{2}}(v,r)$ with respect to $v$ and $r$, respectively.
By the Taylor expansion, we have
\vspace{-0.1cm}
\begin{align}
& N_\kp(u_{h,e}(t_n+\theta),s_{h,e}(t_n+\theta))
- N_\kp(\widetilde{u}_{h,e}^{n+\frac{1}{2}},\widetilde{s}_{h,e}^{n+\frac{1}{2}}) \label{lem_etd2_error_pf1} \\
&  = \nabla_vN_\kp(\widetilde{u}_{h,e}^{n+\frac{1}{2}},\widetilde{s}_{h,e}^{n+\frac{1}{2}})
(u_{h,e}(t_n+\theta)-\widetilde{u}_{h,e}^{n+\frac{1}{2}}) \nn \\
&\quad + \partial_rN_\kp(\widetilde{u}_{h,e}^{n+\frac{1}{2}},\widetilde{s}_{h,e}^{n+\frac{1}{2}})
(s_{h,e}(t_n+\theta)-\widetilde{s}_{h,e}^{n+\frac{1}{2}}) + r_e \nn \\
& = \frac{2\theta-\dt}{2}
\Big(\nabla_vN_\kp(\widetilde{u}_{h,e}^{n+\frac{1}{2}},\widetilde{s}_{h,e}^{n+\frac{1}{2}}) (u_{h,e})'(t_{n+\frac{1}{2}})
+ \partial_rN_\kp(\widetilde{u}_{h,e}^{n+\frac{1}{2}},\widetilde{s}_{h,e}^{n+\frac{1}{2}}) (s_{h,e})'(t_{n+\frac{1}{2}})\Big) \nn \\
&\quad + \nabla_vN_\kp(\widetilde{u}_{h,e}^{n+\frac{1}{2}},\widetilde{s}_{h,e}^{n+\frac{1}{2}})
\bigg(\frac{(2\theta-\dt)^2}{4} (u_{h,e})''(t_{n+\frac{1}{2}})
- \frac{\theta^2+(\theta-\dt)^2}{4} (u_{h,e})''(t_{n+\frac{1}{2}})\bigg) \nn \\
& \quad + \partial_rN_\kp(\widetilde{u}_{h,e}^{n+\frac{1}{2}},\widetilde{s}_{h,e}^{n+\frac{1}{2}})
\bigg(\frac{(2\theta-\dt)^2}{4} (s_{h,e})''(t_{n+\frac{1}{2}})
- \frac{\theta^2+(\theta-\dt)^2}{4} (s_{h,e})''(t_{n+\frac{1}{2}})\bigg) + r_e, \nn
\end{align}
where $r_e$ represents the higher-order remainder term.

If we integrate both sides of \eqref{lem_etd2_error_pf1} with respect to $\theta$ from $0$ to $\dt$ and
notice that $\int_0^\dt (2\theta-\dt) \, \d\theta=0$, then
we obtain $\|R_{2u}^{n,1}\| \le \widetilde{C}_h' \dt^3$,
where $\widetilde{C}_h'>0$ is a constant depending on $u_{h,e}$, $T$, $h$, and $\kp$,
because we have $\|u_{h,e}(t)\|_\infty\le\beta$ and $-C_*\le s_{h,e}(t)\le E_h(\uinit)$ for all $t$,
and $N_\kp(u_{h,e},s_{h,e})$ is smooth with respect to $u_{h,e}$ and $s_{h,e}$.

According to Lemma \ref{lem_expfuns}, we also know
\vspace{-0.15cm}
\[
\|\e^{-(\dt-\theta)L_\kp^{n+\frac{1}{2}}}-I\| \le (\dt-\theta) \rho(L_\kp^{n+\frac{1}{2}}) \le M_h (\dt-\theta).
\]
Combining it  with \eqref{lem_etd2_error_pf1},
we obtain that the leading term of $\|R_{2u}^{n,2}\|$ is
\[
\int_0^\dt (\dt-\theta) |2\theta-\dt| \, \d\theta = \frac{\dt^3}{4}, \vspace{-0.1cm}
\]
which implies $\|R_{2u}^{n,2}\| \le \widetilde{C}_h'' \dt^3$
for some constant $\widetilde{C}_h''>0$ depending on $u_{h,e}$, $T$, $h$, and $\kp$.
Thus we complete the proof of the first inequality in \eqref{etd2_trunerr}.

The second inequality in \eqref{etd2_trunerr} can be viewed as a direct consequence of the Crank--Nicolson discretization.
\end{proof}

\section{Proof of Lemma \ref{lem_etd20_error}}
\label{app2}

\begin{proof}
According to the proof of Theorem \ref{thm_etd1_error},
the error equations with respect to $\widetilde{e}_u^{n+1}$ and $\widetilde{e}_s^{n+1}$ are given by
\begin{subequations}
\label{etd2_err1}
\begin{align}
\widetilde{e}_u^{n+1} - e_u^n
& = (\e^{-\dt L_\kp^n} - I) e_u^n + \dt \phi_1(-\dt L_\kp^n)
[N_\kp^n(u^n,s^n) \label{etd2_err1a} \\
&\qquad- N_\kp^n(u_{h,e}(t_n),s_{h,e}(t_n))] - \int_0^\dt \e^{-(\dt-\theta)L_\kp^n} R_{1u}^n(\theta) \, \d\theta, \nn \\
\widetilde{e}_s^{n+1}-e_s^n
& = \< g(u_{h,e}(t_n),s_{h,e}(t_n)) f(u_{h,e}(t_n)) - g(u^n,s^n) f(u^n),  \label{etd2_err1b} \\
& \qquad u_{h,e}(t_{n+1})-u_{h,e}(t_n) \>- g(u^n,s^n) \<f(u^n),\widetilde{e}_u^{n+1}-e_u^n\> - \dt R_{1s}^n, \nn
\end{align}
\end{subequations}
where the truncation errors $R_{1u}^n$ and $R_{1s}^n$ are
identical to those in \eqref{etd1_eqv_trun} and \eqref{sav1trunb}, respectively,
and satisfy \eqref{eisav1trunerr}.

Taking the discrete inner product of \eqref{etd2_err1a} with $2\widetilde{e}_u^{n+1}$ and using Lemma \ref{lem_expfuns} and \eqref{etd1_err_nonlinear},
we get
\begin{align*}
& \|\widetilde{e}_u^{n+1}\|^2 - \|e_u^n\|^2 + \|\widetilde{e}_u^{n+1} - e_u^n\|^2 \nn \\
&  \le 4 \|e_u^n\| \|\widetilde{e}_u^{n+1}\|
\!+\! 2\dt \|N_\kp^n(u^n\!,s^n) \!-\! N_\kp^n(u_{h,e}(t_n),s_{h,e}(t_n))\| \|\widetilde{e}_u^{n+1}\|
\!+\! 2\dt \!\!\! \sup_{\theta\in(0,\dt)} \!\!\! \|R_{1u}^n(\theta)\| \|\widetilde{e}_u^{n+1}\| \nn \\
& \le 16 \|e_u^n\|^2 + \frac{1}{4} \|\widetilde{e}_u^{n+1}\|^2
+ 8\dt^2 [(C_{g}+G^*\kp)^2 \|e_u^n\|^2 + C_{g}^2 |e_s^n|^2] + \frac{1}{4} \|\widetilde{e}_u^{n+1}\|^2 \nn \\
& \quad + 4 \dt^2 \sup_{\theta\in(0,\dt)} \|R_{1u}^n(\theta)\|^2 + \frac{1}{4} \|\widetilde{e}_u^{n+1}\|^2 \nn \\
& = 16 \|e_u^n\|^2 + 8(C_{g}+G^*\kp)^2 \dt^2 \|e_u^n\|^2 + 8C_{g}^2 \dt^2 |e_s^n|^2
+ \frac{3}{4} \|\widetilde{e}_u^{n+1}\|^2
+ 4 \dt^2 \!\! \sup_{\theta\in(0,\dt)} \!\! \|R_{1u}^n(\theta)\|^2,
\end{align*}

\vspace{-0.15cm}\noindent
and then,
\vspace{-0.1cm}
\begin{align*}
\frac{1}{4}\|\widetilde{e}_u^{n+1}\|^2 + \|\widetilde{e}_u^{n+1} - e_u^n\|^2
& \le 17 \|e_u^n\|^2 + 8(C_{g}+G^*\kp)^2 \dt^2 \|e_u^n\|^2 \\
& \quad + 8C_{g}^2 \dt^2 |e_s^n|^2 + 4 \dt^2 \sup_{\theta\in(0,\dt)} \|R_{1u}^n(\theta)\|^2.
\end{align*}
When $\dt\le 1$, by using \eqref{eisav1trunerr}, we get
\begin{equation}
\label{etd2_err_pf9}
\|\widetilde{e}_u^{n+1}\|^2 + 4\|\widetilde{e}_u^{n+1} - e_u^n\|^2
\le (68 + 32(C_{g}+G^*\kp)^2) \|e_u^n\|^2 + 32C_{g}^2 |e_s^n|^2 + 16C_{e,h}^2 \dt^4.
\end{equation}

Multiplying \eqref{etd2_err1b} by $2\widetilde{e}_s^{n+1}$ yields
\begin{align*}
& |\widetilde{e}_s^{n+1}|^2 - |e_s^n|^2 + |\widetilde{e}_s^{n+1}-e_s^n|^2 \nn \\
& \quad = 2\widetilde{e}_s^{n+1}
\< g(u_{h,e}(t_n),s_{h,e}(t_n)) f(u_{h,e}(t_n)) - g(u^n,s^n) f(u^n), u_{h,e}(t_{n+1})-u_{h,e}(t_n) \> \nn \\
& \quad\quad - 2 \widetilde{e}_s^{n+1} g(u^n,s^n) \<f(u^n),\widetilde{e}_u^{n+1}-e_u^n\>
- 2\dt R_{1s}^n \widetilde{e}_s^{n+1}.
\end{align*}
The last two terms on the right-hand side of the above equality can be estimated as
\begin{eqnarray*}
& - 2\dt R_{1s}^n \widetilde{e}_s^{n+1}
\le 4\dt^2 |R_{1s}^n|^2 + \frac{1}{4} |\widetilde{e}_s^{n+1}|^2,\\
& - 2 \widetilde{e}_s^{n+1} g(u^n,s^n) \<f(u^n),\widetilde{e}_u^{n+1}-e_u^n\>
\le \frac{1}{4} |\widetilde{e}_s^{n+1}|^2 + C_4 \|\widetilde{e}_u^{n+1}-e_u^n\|^2
\end{eqnarray*}
with $C_4>0$ depending on $C_*$, $|\Omega|$, $\uinit$, and $\|f\|_{C[-\beta,\beta]}$.
 The first term  can be estimated in the similar way to \eqref{sav1err_pf7a}, and then
we obtain
\begin{align*}
|\widetilde{e}_s^{n+1}|^2 - |e_s^n|^2 + |\widetilde{e}_s^{n+1}-e_s^n|^2
& \le C_1 \dt (\|e_u^n\|^2 + |e_s^n|^2 + |\widetilde{e}_s^{n+1}|^2) \nn \\
& \quad + C_4 \|\widetilde{e}_u^{n+1}-e_u^n\|^2 + \frac{1}{2} |\widetilde{e}_s^{n+1}|^2 + 4\dt^2 |R_{1s}^n|^2,
\end{align*}

\vspace{-0.2cm}\noindent
and thus,
\[
(1-2C_1\dt) |\widetilde{e}_s^{n+1}|^2
\le 2|e_s^n|^2 + 2C_1 \dt (\|e_u^n\|^2 + |e_s^n|^2) + 2C_4 \|\widetilde{e}_u^{n+1}-e_u^n\|^2 + 8\dt^2 |R_{1s}^n|^2.
\]
When $\dt\le \frac{1}{4C_1}$, we can get  by using \eqref{eisav1trunerr},
\begin{equation}
\label{etd2_err_pf10}
|\widetilde{e}_s^{n+1}|^2 \le \|e_u^n\|^2 + 5|e_s^n|^2 + 4C_4 \|\widetilde{e}_u^{n+1}-e_u^n\|^2 + 16C_{e,h}^2\dt^4.
\end{equation}
The sum of \eqref{etd2_err_pf9} multiplied by $C_4$ and \eqref{etd2_err_pf10} leads to \eqref{etd2_err_pf11}.
\end{proof}


\begin{thebibliography}{0}

\bibitem{AlCa79}
S.~M. Allen and J.~W. Cahn,
{\it A microscopic theory for antiphase boundary motion and its application to antiphase domain coarsening},
Acta Metall., 27 (1979), 1085--1095.

\bibitem{AkrivisLiLi19}
G. Akrivis, B. Li, and D. Li,
{\it Energy-decaying extrapolated RK-SAV methods for the Allen--Cahn and Cahn--Hilliard equations},
SIAM J. Sci. Comput., 41 (2019), A3703--A3727.

\bibitem{Bates06}
P.~W. Bates,
{\it On some nonlocal evolution equations arising in materials science},
Fields Inst. Commun., 48 (2006), 13--52.

%



\bibitem{ChenLiWaWaWa20}
{
W. Chen, W. Li, C. Wang, S. Wang, and X. Wang,     
{\it Energy stable higher-order linear ETD multi-step methods for gradient flows: application to thin film epitaxy},
Res. Math. Sci., 7 (2020), 13.}

\bibitem{ChenWaWaWi19}
{
W. Chen, C. Wang, X. Wang, and S. Wise, 
{\it Positivity-preserving, energy stable numerical schemes for the Cahn--Hilliard equation with logarithmic potential},
J. Comput. Phys., X 3 (2019), 100031.}

\bibitem{ChengQiWa19}
{
K. Cheng, Z. Qiao, and C. Wang,  
{\it A third order exponential time differencing numerical scheme for no-slope-selection epitaxial thin film model with energy stability}, 
J. Sci. Comput., 81 (2019), 154--185.}

\bibitem{ChengLiSh20}
Q. Cheng, C. Liu, and J. Shen,
{\it A new Lagrange multiplier approach for gradient flows},
Comput. Methods Appl. Mech. Engrg., 367 (2020), 113070.

\bibitem{ChengLiSh21}
Q. Cheng, C. Liu, and J. Shen,
{\it Generalized SAV approaches for gradient systems},
J. Comput. Appl. Math., 394 (2021), 113532.

\bibitem{ChengWa21}
{
Q. Cheng and C. Wang, 
{\it Error estimate of a second order accurate scalar auxiliary variable (SAV) numerical method for the epitaxial thin film equation}, 
Adv. Appl. Math. Mech., 13 (2021), 1318--1354.
}

\bibitem{CherfilsMiZe11}
{
L. Cherfils, A. Miranville, and S. Zelik,
{\it The Cahn--Hilliard equation with logarithmic potentials},
Milan J. Math., 79 (2011), 561--596.
}

\bibitem{CoMa02}
S.~M. Cox and P.~C. Matthews,
{\it Exponential time differencing for stiff systems},
J. Comput. Phys., 176 (2002), 430--455.

\bibitem{CuiXuWaJi21}
J. Cui, Z. Xu, Y. Wang, and C. Jiang,
{\it Mass- and energy-preserving exponential Runge--Kutta methods for the nonlinear Schr\"odinger equation},
Appl. Math. Lett., 112 (2021), 106770.



\bibitem{DongWaZhZh19}
{
L. Dong, C. Wang, H. Zhang, and Z. Zhang, 
{\it A positivity-preserving, energy stable and convergent numerical scheme for the Cahn--Hilliard equation with a Flory--Huggins--deGennes energy},
Commun. Math. Sci., 17 (2019), 921--939. 
}


%


%
%


\bibitem{DuJuLiQi21}
Q. Du, L. Ju, X. Li, and Z. Qiao,
{\it Maximum bound principles for a class of semilinear parabolic equations and exponential time-differencing schemes},
SIAM Rev., 63 (2021), 317--359.

\bibitem{DuNi90}
Q. Du and R.~A. Nicolaides,
{\it Numerical analysis of a continuum model of phase transition},
SIAM J. Numer. Anal., 28 (1991), 1310--1322.


\bibitem{DuYaZh20}
Q. Du, J. Yang, and Z. Zhou,
{\it Time-fractional Allen--Cahn equations: analysis and numerical methods},
J. Sci. Comput., 85 (2020), 42.


\bibitem{ElliottLu91}
{
C. Elliott and S. Luckhaus,
{\it A generalized diffusion equation for phase separation of a multi-component mixture with interfacial energy},
SFB 256 Preprint 195, University of Bonn, 1991.}



\bibitem{FeTaYa13}
X. Feng, T. Tang, and J. Yang,
{\it Stabilized Crank--Nicolson/Adams--Bashforth schemes for phase field models},
East Asian J. Appl. Math., 3 (2013), 59--80.


\bibitem{Furihata01}
D. Furihata,
{\it A stable and conservative finite difference scheme for the Cahn--Hilliard equation},
Numer. Math., 87 (2001), 675--699.


\bibitem{GongZh19}
Y. Gong and J. Zhao,
{\it Energy-stable Runge--Kutta schemes for gradient flow models using the energy quadratization approach},
Appl. Math. Lett., 94 (2019), 224--231.


\bibitem{GoShTa01}
S. Gottlieb, C.-W. Shu, and E. Tadmor,
{\it Strong stability-preserving high-order time discretization methods},
SIAM Rev., 43 (2001), 89--112.

\bibitem{GuWaWi14}
Z. Guan, C. Wang, and S.~M. Wise,
{\it A convergent convex splitting scheme for the periodic nonlocal Cahn--Hilliard equation},
Numer. Math., 128 (2014), 377--406.

\bibitem{GuiZh15}
C.~F. Gui and M.~F. Zhao,
{\it Traveling wave solutions of Allen--Cahn equation with a fractional Laplacian},
Ann. Inst. H. Poincar\'{e}-An., 32 (2015), 785--812.




\bibitem{Matfun08}
N.~J. Higham,
{\it Functions of Matrices: Theory and Computation},
SIAM, Philadelphia, PA, 2008.

\bibitem{HoOs10}
M. Hochbruck and A. Ostermann,
{\it Exponential integrators},
Acta Numer., 19 (2010), 209--286.

\bibitem{HouAzXu19}
D. Hou, M. Azaiez, and C. Xu,
{\it A variant of scalar auxiliary variable approaches for gradient flows},
J. Comput. Phys., 395 (2019), 307--332.


\bibitem{HoTaYa17}
T. Hou, T. Tang, and J. Yang,
{\it Numerical analysis of fully discretized Crank--Nicolson scheme for fractional-in-space Allen--Cahn equations},
J. Sci. Comput., 72 (2017), 1214--1231.


\bibitem{HuangShYa20}
F. Huang, J. Shen, and Z. Yang,
{\it A highly efficient and accurate new scalar auxiliary variable approach for gradient flows},
SIAM J. Sci. Comput., 42 (2020), A2514--A2536.

%

\bibitem{JiangWaCa20}
C. Jiang, Y. Wang, and W. Cai,
{\it A linearly implicit energy-preserving exponential integrator for the nonlinear Klein--Gordon equation},
J. Comput. Phys., 419 (2020), 109690.


\bibitem{JuLiQi21}
L. Ju, X. Li, and Z. Qiao,
{\it Stabilized exponential-SAV schemes  preserving energy dissipation law and maximum bound principle
for the Allen--Cahn type equations},
submitted.

\bibitem{JuLiQiYa21}
L. Ju, X. Li, Z. Qiao, and J. Yang,
{\it Maximum bound principle preserving integrating factor Runge--Kutta methods for semilinear parabolic equations},
J. Comput. Phys., 439 (2021), 110405.

\bibitem{JuLiQiZh18}
L. Ju, X. Li, Z. Qiao, and H. Zhang,
{\it Energy stability and error estimates of exponential time differencing schemes
for the epitaxial growth model without slope selection},
Math. Comp., 87 (2018), 1859--1885.



\bibitem{JuZhZhDu15}
L. Ju, J. Zhang, L. Zhu, and Q. Du,
{\it Fast explicit integration factor methods for semilinear parabolic equations},
J. Sci. Comput., 62 (2015), 431--455.

\bibitem{LiJuCaFe21}
J. Li, L. Ju, Y. Cai, and X. Feng,
{\it Unconditionally maximum bound principle preserving linear schemes
for the conservative Allen--Cahn equation with nonlocal constraint},
J. Sci. Comput., 87 (2021), 98.

\bibitem{LiLiJuFe21}
J. Li, X. Li, L. Ju, and X. Feng,
{\it Stabilized integrating factor Runge--Kutta method and unconditional preservation of maximum bound principle},
SIAM J. Sci. Comput., 43 (2021), A1780--A1802.


\bibitem{LiaoTaZh20}
H. Liao, T. Tang, and T. Zhou,
{\it On energy stable, maximum-principle preserving, second-order BDF scheme
with variable steps for the Allen--Cahn equation},
SIAM J. Numer. Anal., 58 (2020), 2294--2314.

\bibitem{LiuLi20}
Z. Liu and X. Li,
{\it The exponential scalar auxiliary variable (E-SAV) approach for phase field models and its explicit computing},
SIAM J. Sci. Comput., 42 (2020), B630--B655.

\bibitem{McQuRo99}
R.~I. McLachlan, G.~R.~W. Quispel, and N. Robidoux,
{\it Geometric integration using discrete gradients},
R. Soc. Lond. Philos. Trans. Ser. A Math. Phys. Eng. Sci., 357 (1999), 1021--1045.

%
%

\bibitem{QiaoSuZh15}
Z. Qiao, Z. Sun, and Z. Zhang,
{\it Stability and convergence of second-order schemes for the nonlinear epitaxial growth model without slope selection},
Math. Comp., 84 (2015), 653--674.

\bibitem{QiaoZhTa11}
{
Z. Qiao, Z. Zhang, and T. Tang,
{\it An adaptive time-stepping strategy for the molecular beam epitaxy models},
SIAM J. Sci. Comput., 33 (2011), 1395--1414.
}


\bibitem{ShTaYa16}
J. Shen, T. Tang, and J. Yang,
{\it On the maximum principle preserving schemes for the generalized Allen--Cahn equation},
Commun. Math. Sci., 14 (2016), 1517--1534.

\bibitem{ShWaWaWi12}
J. Shen, C. Wang, X. Wang, and S.~M. Wise,
{\it Second-order convex splitting schemes for gradient flows with Ehrlich--Schwoebel type energy:
application to thin film epitaxy},
SIAM J. Numer. Anal., 50 (2012), 105--125.

\bibitem{ShXu18}
J. Shen and J. Xu,
{\it Convergence and error analysis for the scalar auxiliary variable (SAV) schemes to gradient flows},
SIAM J. Numer. Anal., 56 (2018), 2895--2912.

\bibitem{ShXuYa18}
J. Shen, J. Xu, and J. Yang,
{\it The scalar auxiliary variable (SAV) approach for gradient flows},
J. Comput. Phys., 353 (2018), 407--416.

\bibitem{ShXuYa19}
J. Shen, J. Xu, and J. Yang,
{\it A new class of efficient and robust energy stable schemes for gradient flows},
SIAM Rev., 61 (2019), 474--506.

\bibitem{ShYa10b}
J. Shen and X. Yang,
{\it Numerical approximations of Allen--Cahn and Cahn--Hilliard equations},
Discrete Contin. Dyn. Syst., 28 (2010), 1669--1691.

\bibitem{ShenZh21}
{
J. Shen and X. Zhang,
{\it Discrete maximum principle of a high order finite difference scheme for a generalized Allen--Cahn equation},
arXiv preprint arXiv:2104.11813, 2021.
}

%



\bibitem{TaYa16}
T. Tang and J. Yang,
{\it Implicit-explicit scheme for the Allen--Cahn equation preserves the maximum principle},
J. Comput. Math., 34 (2016), 471--481.


\bibitem{TianJiLu18}
{
D. Tian, Y. Jin, and G. Lu,
{\it Discrete maximum principle and energy stability of compact difference scheme for the Allen--Cahn equation}, 
Preprints, 2018, 2018120294.
}


\bibitem{WiWaLo09}
S.~M. Wise, C. Wang, and J.~S. Lowengrub,
{\it An energy stable and convergent finite difference scheme for the phase field crystal equation},
SIAM J. Numer. Anal., 47 (2009), 2269--2288.


\bibitem{XuTa06}
C. Xu and T. Tang,
{\it Stability analysis of large time-stepping methods for epitaxial growth models},
SIAM J. Numer. Anal., 44 (2006), 1759--1779.


\bibitem{XuYaZhXi19}
Z. Xu, X. Yang, H. Zhang, and Z. Xie,
{\it Efficient and linear schemes for anisotropic Cahn--Hilliard model
using the stabilized-invariant energy quadratization (S-IEQ) approach},
Comput. Phys. Commun., 238 (2019), 36--49.


\bibitem{YangYuZh22}
J. Yang, Z. Yuan, and Z. Zhou,
{\it Arbitrarily high-order maximum bound preserving schemes with cut-off postprocessing for Allen--Cahn equations},
J. Sci. Comput., 90 (2022), 76.


\bibitem{YangZh20}
X. Yang and G. Zhang,
{\it Convergence analysis for the invariant energy quadratization (IEQ) schemes for
solving the Cahn--Hilliard and Allen--Cahn equations with general nonlinear potential},
J. Sci. Comput., 82 (2020), 55.



\bibitem{ZhangYaQiSo21}
H. Zhang, J. Yan, X. Qian, and S. Song,
{\it Numerical analysis and applications of explicit high order maximum principle preserving
integrating factor Runge--Kutta schemes for Allen--Cahn equation},
Appl. Numer. Math., 161 (2021), 372--390.

\bibitem{ZhuJuZh16}
L. Zhu, L. Ju, and W. Zhao,
{\it Fast high-order compact exponential time differencing Runge--Kutta methods
for second-order semilinear parabolic equations},
J. Sci. Comput., 67 (2016), 1043--1065.


\end{thebibliography}
\end{document}